\documentclass[11pt,reqno]{amsart} \usepackage[margin=1in]{geometry}
\usepackage[foot]{amsaddr}

\usepackage{amsmath}
\usepackage{amssymb}
\usepackage{amsthm}
\usepackage{color}
\usepackage{enumerate}
\usepackage[CJKbookmarks=true,
            bookmarksnumbered=true,
            bookmarksopen=true,
            colorlinks=true,
            citecolor=red,
            linkcolor=blue,
            anchorcolor=red,
            urlcolor=blue
            ]{hyperref}

\newtheorem{theorem}{Theorem}[section]
\newtheorem{lemma}[theorem]{Lemma}
\newtheorem{proposition}[theorem]{Proposition}

\theoremstyle{definition}

\newtheorem{assumption}[theorem]{Assumption}

\numberwithin{equation}{section}
\numberwithin{figure}{section}

\allowdisplaybreaks

\newcommand\numberthis{\addtocounter{equation}{1}\tag{\theequation}}

\usepackage{xspace,prettyref}
\newrefformat{eq}{(\ref{#1})}
\newrefformat{chap}{Chapter~\ref{#1}}
\newrefformat{sec}{Section~\ref{#1}}
\newrefformat{alg}{Algorithm~\ref{#1}}
\newrefformat{fig}{Fig.~\ref{#1}}
\newrefformat{tab}{Table~\ref{#1}}
\newrefformat{rmk}{Remark~\ref{#1}}
\newrefformat{clm}{Claim~\ref{#1}}
\newrefformat{def}{Definition~\ref{#1}}
\newrefformat{cor}{Corollary~\ref{#1}}
\newrefformat{lmm}{Lemma~\ref{#1}}
\newrefformat{lemma}{Lemma~\ref{#1}}
\newrefformat{prop}{Proposition~\ref{#1}}
\newrefformat{app}{Appendix~\ref{#1}}
\newrefformat{hyp}{Hypothesis~\ref{#1}}
\newrefformat{thm}{Theorem~\ref{#1}}
\newrefformat{assump}{Assumption~\ref{#1}}
\newrefformat{conj}{Conjecture~\ref{#1}}

\def\R{\mathbb{R}}
\def\E{\mathbb{E}}
\def\P{\mathbb{P}}
\def\O{\mathbb{O}}
\def\SO{\mathbb{SO}}
\def\RS{\mathrm{RS}}
\def\H{\mathsf{H}}
\def\G{\mathsf{G}}
\def\X{\mathsf{X}}
\def\Y{\mathsf{Y}}

\def\N{\mathcal{N}}
\def\eps{\varepsilon}
\def\cF{\mathcal{F}}
\def\cG{\mathcal{G}}
\def\cD{\mathcal{D}}
\def\cH{\mathcal{H}}

\def\cL{\mathcal{L}}

\DeclareMathOperator{\Haar}{Haar}
\DeclareMathOperator{\diag}{diag}
\DeclareMathOperator{\supp}{supp}
\DeclareMathOperator{\Tr}{Tr}

\DeclareMathOperator{\Unif}{Unif}

\def\I{\mathrm{I}}
\def\II{\mathrm{II}}
\def\III{\mathrm{III}}
\def\IV{\mathrm{IV}}

\newcommand{\Opnorm}[1]{\|#1\|_{\rm op}}
\newcommand{\Fnorm}[1]{\|#1\|_{\rm F}}
\newcommand{\Var}{\mathrm{Var}}
\newcommand{\pth}[1]{\left( #1 \right)}

\newcommand{\sth}[1]{\left\{ #1 \right\}}
\newcommand{\reals}{\mathbb{R}}

\newcommand{\Expect}{\mathbb{E}}

\newcommand{\prob}[1]{\mathbb{P}\left[#1\right]}

\newcommand{\Indc}{\mathbf{1}}

\newcommand{\calD}{{\mathcal{D}}}
\newcommand{\calE}{{\mathcal{E}}}
\newcommand{\calF}{{\mathcal{F}}}

\newcommand{\calH}{{\mathcal{H}}}

\newcommand{\calK}{{\mathcal{K}}}
\newcommand{\calL}{{\mathcal{L}}}

\newcommand{\calN}{{\mathcal{N}}}

\newcommand{\calV}{{\mathcal{V}}}

\begin{document}

\title[Orthogonally-invariant spin glass models]{The replica-symmetric free
energy for Ising spin glasses with orthogonally invariant couplings}

\author{Zhou Fan}\email{zhou.fan@yale.edu (corresponding)}
\author{Yihong Wu}
\address{Department of Statistics and Data Science \\ Yale University \\ New
Haven \\ CT \\ USA}
\email{yihong.wu@yale.edu}
\thanks{Z.~Fan is supported in part by NSF Grants DMS-1916198 and DMS-2142476. Y.~Wu is
supported in part by NSF Grant CCF-1900507, an NSF CAREER award CCF-1651588,
and an Alfred Sloan fellowship.}

\maketitle

\begin{abstract}
We study a variant of the Sherrington-Kirkpatrick (S-K) spin glass model
with external field, where the random symmetric couplings matrix does not
consist of i.i.d.\ entries but is instead orthogonally invariant in law. For
sufficiently high temperature, we prove a replica-symmetric formula for the
first-order limit of the model free energy. Our analysis
is an adaptation of a conditional second-moment-method argument previously
introduced by Bolthausen for studying the high-temperature regime of the
S-K model, where one conditions on the iterates of an Approximate Message
Passing (AMP) algorithm for solving the TAP equations for the model
magnetization. We apply this method using a memory-free version of AMP that
is tailored to the orthogonally invariant structure of the model couplings.\\
\end{abstract}

\section{Introduction}

We study a probability model on the hypercube $\sigma \in \{+1,-1\}^n$ given by
\begin{equation}\label{eq:intromodel}
P(\sigma)=\frac{1}{Z}\exp\left(\frac{\beta}{2}\sigma^\top J\sigma+h^\top
\sigma\right).
\end{equation}
Here $h \in \R^n$ is a deterministic vector, $J \in \R^{n \times n}$ is a
random symmetric matrix which we will assume satisfies the orthogonal
invariance in law
\[J\overset{L}{=} O^\top JO \text{ for any orthogonal matrix } O \in \R^{n
\times n},\]
and $Z$ is the partition function
\[Z=\sum_{\sigma \in \{+1,-1\}^n} \exp\left(\frac{\beta}{2}\sigma^\top
J\sigma+h^\top \sigma\right).\]
We will refer to $J$ as the couplings matrix, $h$ as the external field,
and $\beta$ as the inverse temperature.

The specific example of the Sherrington-Kirkpatrick (S-K) model
\cite{sherrington1975solvable},
where $J$ has i.i.d.\ Gaussian entries above the diagonal, is
well-studied and known to exhibit rich phenomena. At high
temperatures, $P(\sigma)$ is ``replica-symmetric'', the large-$n$ limit of the
free energy is described by a simple replica-symmetric
formula \cite{sherrington1975solvable,aizenman1987some,bolthausen2018morita},
and the magnetization 
$m=\sum_{\sigma \in \{+1,-1\}^n} \sigma \cdot P(\sigma)$ satisfies in this
limit the Thouless-Anderson-Palmer (TAP) mean-field equations
\cite{thouless1977solution,chatterjee2010spin,talagrand2010mean}. At low
temperatures, the limit free energy is described more generally by Parisi's
variational formula
\cite{parisi1979infinite,parisi1980sequence,guerra2003broken,talagrand2006parisi}.
The solution of the variational problem may be understood as corresponding to
an ultrametric tree structure for $P(\sigma)$, and the TAP equations
describe the conditional means of the ``pure states'' in this ultrametric tree
\cite{mezard1984nature,mezard1984replica,mezard1985microstructure}. This
picture has been formalized and proven rigorously for certain mixed 
$p$-spin analogues of the S-K model
in \cite{panchenko2013parisi,auffinger2019thouless}.

Here, we are interested in the more general setting of (\ref{eq:intromodel})
where $J$ is orthogonally invariant, but can have arbitrary spectral
distribution and dependent entries. Examples include the
random orthogonal model (ROM) \cite{marinari1994replica} where $J$ has all
eigenvalues equal to $\{+1,-1\}$, and the Gaussian Hopfield model
\cite{hopfield1982neural} where $J=G^\top G$ and $G$ is a rectangular Gaussian
matrix. In the physics literature,
the replica-symmetric and 1-RSB free energies
for the ROM were computed by Marinari et al.\ in \cite{marinari1994replica},
and extended to models with general orthogonally invariant couplings in
\cite{cherrier2003role}. Parisi and Potters derived in \cite{parisi1995mean}
the TAP mean-field equations for the ROM, using a perturbative expansion
approach of \cite{plefka1982convergence,georges1991expand} and a conjectured
resummation of the terms of this expansion. Opper and Winther provided in
\cite{opper2001adaptive} an alternative derivation of the TAP equations
using the cavity method and a system of self-consistent equations for the
cavity fields (which we review in Appendix \ref{appendix:cavity}),
and verified also via a replica calculation that the TAP free
energy evaluated at the model magnetization coincides with the free energy
given by the replica-symmetric formula. At present, few rigorous mathematical
results are known for models with general orthogonally invariant couplings
matrices $J$.

Our work is in large part motivated by a renewed interest in
these types of mean-field models in information theory, statistics, and
machine learning
\cite{takeda2006analysis,tulino2013support,reeves2017additivity,takeuchi2017rigorous,ma2017orthogonal,barbier2018mutual,rangan2019vector,dudeja2020universality,fan2020approximate,gerbelot2020asymptotic,liu2020memory,maillard2020phase,takeuchi2020convolutional,takeuchi2020bayes},
where orthogonally invariant matrices may serve as more robust models
of regression and sensing designs or more accurate models of
noise in data applications. Indeed, following the initial posting
of this work, several dynamical universality results were obtained in
\cite{dudeja2023universality,wang2022universality,dudeja2022spectral}
showing that the mean-field dynamics of Approximate Message Passing (AMP)
and other first-order iterative algorithms applied to orthogonally invariant matrices are universal
across broad classes of matrices with delocalized eigenvectors. In many of
these applications, replica predictions for the model free energy 
are conjectured, but not rigorously known. Maillard et al.\ studied in
\cite{maillard2019high} a class of computational algorithms in the context of
such orthogonally invariant models, extending the diagrammatic expansion
method of \cite{plefka1982convergence,georges1991expand,parisi1995mean} to describe the
connections between these algorithms and the predicted mean field
theory; the authors of \cite{maillard2019high} highlighted the mathematical
verification of these predictions as an open question.

We study in this work the specific model (\ref{eq:intromodel}), and prove
a replica-symmetric formula for the first-order limit of its free
energy in a sufficiently high temperature regime. This extends previous work of
\cite{bhattacharya2016high}, which showed such a result in the absence of an
external field ($h=0$). Similar to the S-K
model \cite{aizenman1987some}, the $h=0$ setting is special in that the
quenched free energy $n^{-1} \E \log Z$ coincides asymptotically with the
annealed free energy $n^{-1} \log \E Z$, as was verified in
\cite{bhattacharya2016high} using the second moment method.
This no longer holds when $h \neq 0$, and our proof applies instead a
conditional version of this idea developed in \cite{bolthausen2018morita}
for the S-K model, where one establishes that the quenched and
annealed free energies coincide upon conditioning on an appropriately chosen
sigma-field that is informative about the random magnetization. This method
was refined for the S-K model in \cite{brennecke2022replica} to cover a large
and explicit part of the high-temperature regime, and is also related to
analyses of \cite{ding2019capacity,bolthausen2022gardner} for the Ising perceptron model.

Our construction of the conditioning sigma-field relies on recent
developments on iterative algorithms for
solving the TAP equations for the model (\ref{eq:intromodel}). We summarize
these developments and the proof strategy in Section \ref{subsec:proof}
below, after presenting our main result. Following the initial posting of this
work, our analyses have been extended in \cite{fan2022tap} to show also the
validity of the TAP equations for the magnetization under a similar
high-temperature assumption, and in \cite{li2023random}
to obtain analogous results in a statistical linear model with
orthogonally invariant regression design.

\subsection{Model and main result}

Consider the Gibbs distribution (\ref{eq:intromodel}) on the binary hypercube,
under the following assumptions for the couplings matrix $J$ and external field $h$.

\begin{assumption}\label{assump:main}
Let $J=O^\top DO$ be the eigen-decomposition of $J$.
\begin{enumerate}[(a)]
\item $O \sim \Haar(\O(n))$ is a random Haar-distributed orthogonal matrix.
\item $D=\diag(d_1,\ldots,d_n)$ is a deterministic diagonal matrix of
eigenvalues, whose empirical distribution converges weakly to a limit law
\[\frac{1}{n}\sum_{i=1}^n \delta_{d_i} \to \mu_D\]
as $n \to \infty$. This law $\mu_D$ has strictly positive variance and
a compact support $\supp(\mu_D)$. Furthermore,
\[\lim_{n \to \infty} \max(d_1,\ldots,d_n)=d_+ \triangleq
\max(x:x \in \supp(\mu_D)), \qquad
\liminf_{n \to \infty} \min(d_1,\ldots,d_n)>-\infty.\]
\item $h=(h_1,\ldots,h_n) \in \R^n$ is a deterministic vector, whose empirical
distribution of entries converges weakly to a limit law
\[\frac{1}{n}\sum_{i=1}^n \delta_{h_i} \to \mu_H\]
as $n \to \infty$. For every $p \geq 1$, the law $\mu_H$ has
finite $p^\text{th}$ moment, and $n^{-1}\sum_{i=1}^n h_i^p \to \E_{\H \sim
\mu_H}[\H^p]$.\footnote{This moment condition for $h$ is used to apply the AMP
state evolution analysis of \cite{fan2020approximate} to deduce
Theorem \ref{thm:SE}, and is not used in the rest of the argument.}
\end{enumerate}
\end{assumption}
\noindent We remark that our results apply also to models with random
$(D,h)$ independent of $O$ which satisfy these conditions almost
surely as $n \to \infty$, by applying the results conditionally on $(D,h)$.

We are interested in the asymptotic free energy
\begin{equation}\label{eq:freeenergy}
\Psi=\lim_{n \to \infty} \frac{1}{n} \log Z.
\end{equation}
For sufficiently small $\beta>0$, we prove that this limit exists almost surely
and is given by the following replica-symmetric formula: Denote the
Cauchy- and R-transforms of $\mu_D$ by
\[G(z)=\int \frac{1}{z-x}\mu_D(dx), \qquad R(z)=G^{-1}(z)-\frac{1}{z}.\]
We define $G(z)$ for real arguments $z \in (d_+,\infty)$. The
function $G:(d_+,\infty) \to (0,G(d_+))$ is strictly decreasing, where
we denote
\[G(d_+) \triangleq \lim_{z \downarrow d_+} G(z) \in (0,\infty].\]
We define $R(z)$ for real arguments $z \in (0,G(d_+))$, where
$G^{-1}$ is the functional inverse of $G$ over the domain $(d_+,\infty)$.

\begin{proposition}\label{prop:qstarunique}
Under Assumption \ref{assump:main},
for some $\beta_0=\beta_0(\mu_D)>0$ and all
$\beta \in (0,\beta_0)$, there is
a unique solution $q_* \in [0,1)$ to the fixed-point equation
\begin{equation}\label{eq:qstarunscaled}
q_*=\E[\tanh(\H+\sigma_*\G)^2],
\qquad \sigma_*^2=\beta^2 q_*R'(\beta(1-q_*))
\end{equation}
where the expectation is over independent random variables $\G \sim \N(0,1)$
and $\H \sim \mu_H$.
\end{proposition}

Let $q_*,\sigma_*^2$ and $\G,\H$ be as above.
Then the replica-symmetric prediction for the free energy $\Psi$ is (see
e.g.\ \cite[Eq.\ (56)]{opper2001adaptive})
\begin{align}
\Psi_{\RS}&=\E\Big[\log 2 \cosh(\H+\sigma_* \G)\Big]
+\frac{\beta q_*}{2}R(\beta(1-q_*))\nonumber\\
&\hspace{1in}-\frac{\beta^2q_*(1-q_*)}{2}R'(\beta(1-q_*))
+\frac{1}{2}\int_0^{1-q_*} \beta R(\beta z)dz.\label{eq:PsiRSbeta}
\end{align}
The correctness of this prediction for sufficiently high temperature is justified by the next theorem, which is the main result of the paper.
\begin{theorem}\label{thm:replicasymmetric}
Suppose Assumption \ref{assump:main} holds. Then for some
$\beta_0=\beta_0(\mu_D)>0$ depending only on $\mu_D$,
and for any fixed $\beta \in (0,\beta_0)$, almost surely
\[\lim_{n \to \infty} \frac{1}{n}\log Z=\Psi_{\RS}.\]
\end{theorem}
We mention that for the spherical counterpart of the Ising model \prettyref{eq:intromodel}, 
the free energy can be computed directly and also agrees with its replica-symmetric prediction.
We carry out this computation in \prettyref{app:sphere} by applying one of the technical results in \prettyref{sec:haar}.

\subsection{Overview of the proof}\label{subsec:proof}

We will adopt the conditional second moment method of
\cite{bolthausen2018morita}, and show that
\begin{equation}\label{eq:conditionalmoments}
\lim_{n \to \infty} \frac{1}{n}\log \E[Z \mid \cG] \approx \Psi_{RS},
\qquad \lim_{n \to \infty} \frac{1}{n}\log \E[Z^2 \mid \cG] \approx 2\Psi_{RS}
\end{equation}
for an appropriately chosen sigma-field $\cG$. Together with classical
concentration-of-measure results for Haar measure over the orthogonal group,
this will be enough to show Theorem \ref{thm:replicasymmetric}.

We define $\cG=\cG_t$ as the sigma-field generated by a fixed number $t$ of
iterations of an AMP algorithm designed 
to solve the TAP mean-field equations described in
\cite{parisi1995mean,opper2001adaptive}---see~(\ref{eq:TAP}) below. Such an
algorithm was introduced for the S-K model in \cite{bolthausen2014iterative}
and for applications in compressed sensing in
\cite{donoho2009message,bayati2011dynamics}. For the Ising model
(\ref{eq:intromodel}) with orthogonally invariant couplings,
a general class of AMP-type procedures was described in \cite{opper2016theory},
including a ``single-step memory'' algorithm for solving the TAP equations
that reduces to the one of
\cite{bolthausen2014iterative} in the S-K setting. Our analyses will rely on a
rigorous characterization of the state evolution of such algorithms
obtained in \cite{fan2020approximate}.

The specific algorithm we use to construct $\cG_t$ is not of the
single-step memory form of \cite{bolthausen2014iterative,opper2016theory},
but rather an alternative ``memory-free'' form introduced
by \c{C}akmak and Opper in \cite{ccakmak2019memory}, applying the
general procedure in \cite{opper2016theory} with a resolvent of $J$
instead of the couplings matrix $J$ itself.
This memory-free algorithm is related to a class of Vector/Orthogonal AMP
algorithms developed for compressed sensing applications in
\cite{ma2017orthogonal,takeuchi2017rigorous,rangan2019vector}, and
may be derived also from the Expectation Propagation
framework of Minka \cite{minka2001family}. For our purposes,
use of this algorithm leads to two important simplifications:
First, its state evolution has a simple description when $J$ is
non-Gaussian, whereas that of alternative iterative procedures may
have a complicated dependence on the spectral free cumulants of $J$.
Second, the analysis of \cite{fan2020approximate} reveals that iterates of this 
algorithm have an asymptotic freeness property with respect to $J$, which we
describe below in Proposition \ref{prop:AMPfree}. Both simplifications are
important in enabling our computations of the conditional moments in
(\ref{eq:conditionalmoments}).

The details of our strategy for showing (\ref{eq:conditionalmoments}) are
somewhat different from those presented in \cite{bolthausen2018morita}, and 
we proceed in two high-level steps: We first apply the AMP
state evolution and a large deviations argument to give exact expressions for
the large-$n$ limits of the conditional moments (\ref{eq:conditionalmoments})
in terms of low-dimensional variational problems. This leverages and
extends some results of Guionnet and Ma{\i}da \cite{guionnet2005fourier} that
relate exponential integrals over the orthogonal group to the R-transform of
$J$. We then analyze these
variational problems, proving upper and lower bounds for their values
that are tight for $\Psi_{\RS}$ and $2\Psi_{\RS}$
in the limit as the number of algorithm iterations $t \to
\infty$. The assumption of small $\beta$ (i.e.\ sufficiently high
temperature) is used in a crucial way in the upper bounds, to show a global
concavity property of these variational problems.

The remainder of the paper is organized as follows: In Section \ref{sec:prelim},
we collect the general ingredients of the proof, including a more detailed
description of the AMP algorithm and its state evolution, and the evaluations
of the required exponential integrals over the orthogonal group. In Sections
\ref{sec:firstmoment} and \ref{sec:secondmoment}, we analyze the conditional first moment 
$\E[Z \mid \cG_t]$ and second moment $\E[Z^2 \mid \cG_t]$ respectively, leading to the proof of Theorem \ref{thm:replicasymmetric} in \prettyref{sec:replicasymmetric-pf}.

\subsection*{Notation}

$\O(n)$ and $\SO(n)$ are the orthogonal and special orthogonal groups of
$n \times n$ matrices. $\Haar(\cdot)$ denotes the Haar-measure on these groups.

$\|\cdot\|$ is the $\ell_2$-norm for vectors and $\ell_2 \to \ell_2$ operator
norm for matrices; we may write the latter as $\Opnorm{\cdot}$ in situations
where this is unclear. $\Fnorm{\cdot}$ is the Frobenius norm for matrices.
We use the convention that for scalar values
$x_1,\ldots,x_k$, $(x_1,\ldots,x_k) \in \R^k$ denotes the \emph{column} vector
containing these values.
We write $\triangleq$ for a definition or assignment.
We reserve the sans-serif font $\G,\H,\X,\Y$ for scalar random variables.

\section{Preliminaries}\label{sec:prelim}

\subsection{Centering and rescaling}

Adding a multiple of the identity to $J$ shifts the free energy $\Psi$
and $\Psi_{\RS}$ by the
same additive constant. Thus, we may assume without loss of generality that
\begin{equation}\label{eq:meanzero}
\int x\,\mu_D(dx)=0.
\end{equation}
Since $\mu_D$ has positive variance by Assumption \ref{assump:main}(b),
we may also assume without loss of generality that
\begin{equation}\label{eq:varianceone}
\int x^2\,\mu_D(dx)=1,
\end{equation}
by rescaling $J=O^\top DO$ and incorporating this scaling into $\beta$.

For most of the proof, it will be notationally convenient to absorb the
parameter $\beta$ into the couplings matrix $J$, after this centering and
rescaling. We define
\begin{equation}
\bar{J}=\beta J, \qquad \bar{D}=\diag(\bar{d}_1,\ldots,\bar{d}_n)=\beta D,
\qquad \mu_{\bar{D}}=\lim_{n \to \infty} \frac{1}{n}\sum_{i=1}^n \delta_{
\bar{d_i}}, \qquad \bar{d}_+=\beta d_+.
\label{eq:scaling}
\end{equation}
Thus $\mu_{\bar{D}}$ is the rescaling of the limit spectral
law $\mu_D$, and $\bar{d}_+=\max(x:x \in \supp(\mu_{\bar{D}}))$ is its maximum
point of support.

We denote the Cauchy- and R-transforms of $\mu_{\bar{D}}$ by
\begin{equation}\label{eq:GR}
\bar{G}(z)=\int \frac{1}{z-x}\mu_{\bar{D}}(dx), \qquad
\bar{R}(z)=\bar{G}^{-1}(z)-\frac{1}{z},
\end{equation}
where $\bar{G}(z)$ is defined on $(\bar{d}_+,\infty)$, and $\bar{R}(z)$ on
$(0,\bar{G}(\bar{d}_+))$.
These are related to the Cauchy- and R-transforms of $\mu_D$ by
\begin{equation}\label{eq:barGR}
\bar{G}(z)=\frac{1}{\beta}G\left(\frac{z}{\beta}\right),
\qquad \bar{R}(z)=\beta R(\beta z).
\end{equation}
Let $\{\kappa_k\}_{k \geq 1}$ be the free cumulants of the law $\mu_D$.
Since $\kappa_1$ and $\kappa_2$ correspond to the mean and variance of $\mu_D$ (cf.~\cite[Examples 11.6]{nica2006lectures}), 
(\ref{eq:meanzero}) and (\ref{eq:varianceone}) imply that
$\kappa_1=0$ and $\kappa_2=1$. Writing $\|\mu_D\|_\infty=\max(|x|:x \in
\supp(\mu_D))$, we have
\begin{equation}\label{eq:cumulantbound}
|\kappa_k| \leq (16\|\mu_D\|_\infty)^k
\end{equation}
for all $k \geq 1$, and the R-transform admits the convergent series expansion
for small $z$ given by
\[R(z)=\sum_{k \geq 1} \kappa_k z^{k-1}\]
(cf.~\cite[Notation 12.6, Proposition 13.15]{nica2006lectures}). The free
cumulants of $\mu_{\bar{D}}$ are then $\bar{\kappa}_k=\beta^k \kappa_k$,
satisfying $\bar{\kappa}_1=0$, $\bar{\kappa}_2=\beta^2$, and
$|\bar{\kappa}_k| \leq (16\|\mu_D\|_\infty \beta)^k$ for $k \geq 3$.
The R-transform of $\mu_{\bar{D}}$ for small $z$ is
\begin{equation}\label{eq:barRseries}
\bar{R}(z)=\sum_{k \geq 1} \bar{\kappa}_k z^{k-1}.
\end{equation}

The Gibbs distribution and partition function in (\ref{eq:intromodel})
may be written in this rescaled notation as
\begin{equation}\label{eq:gibbsmeasure}
P(\sigma)=\frac{1}{Z}\exp\left(\frac{1}{2}\sigma^\top \bar{J}\sigma
+h^\top \sigma\right), \qquad Z=\sum_{\sigma \in \{+1,-1\}^n}
\exp\left(\frac{1}{2}\sigma^\top \bar{J}\sigma+h^\top \sigma\right).
\end{equation}
The fixed-point equation (\ref{eq:qstarunscaled}) for $q_*$ is written in terms
of $\bar{R}(z)$ as
\begin{equation}\label{eq:qstar}
q_*=\E[\tanh(\H+\sigma_*\G)^2], \qquad \sigma_*^2=q_*\bar{R}'(1-q_*)
\end{equation}
and the replica-symmetric free energy (\ref{eq:PsiRSbeta}) is
\begin{equation}\label{eq:PsiRS}
\Psi_{\RS}=\E\Big[\log 2\cosh(\H+\sigma_*\G)\Big]+\frac{q_*}{2}\bar{R}(1-q_*)
-\frac{q_*(1-q_*)}{2}\bar{R}'(1-q_*)+\frac{1}{2}\int_0^{1-q_*}
\bar{R}(z)dz.
\end{equation}

\subsection{AMP for solving the TAP equations}

Denote by
\[m=\sum_{\sigma \in \{+1,-1\}^n} \sigma \cdot P(\sigma) \in
(-1,1)^n \]
the magnetization vector of the Gibbs distribution (\ref{eq:gibbsmeasure}). It
is predicted that for sufficiently small $\beta>0$, this vector $m$
approximately satisfies the TAP mean-field
equations \cite{parisi1995mean,opper2001adaptive}
\begin{equation}\label{eq:TAP}
m=\tanh\Big(h+\bar{J}m-\bar{R}(1-q_*)m\Big).
\end{equation}
Here and below, $\tanh(\cdot)$ is applied coordinatewise.
For the S-K model where $\bar{J}$ is Gaussian,
we have $\bar{R}(x)=\beta^2 x$, and this coincides with the classical
TAP equations of \cite{thouless1977solution}. We provide a brief review of the
cavity-method derivation of these TAP equations
from \cite{opper2001adaptive} in Appendix~\ref{appendix:cavity}.

Our proof of Theorem \ref{thm:replicasymmetric} will compute the
first and second moments of the partition function $Z$ conditioned on a
sigma-field generated by an iterative AMP
algorithm for solving the TAP equations. We consider the following
algorithm from \cite{ccakmak2019memory} having ``memory-free'' dynamics: Define
\begin{equation}\label{eq:lambdastar}
\lambda_*=\bar{G}^{-1}(1-q_*)=\bar{R}(1-q_*)+\frac{1}{1-q_*}
\end{equation}
so that $\bar{G}(\lambda_*)=1-q_*$. This is well-defined for any $\beta \in
(0,G(d_+))$, since $1-q_*\leq 1 < \bar G(\bar d_+) = G(d_+)/\beta$. Consider the matrix
\[\Gamma=\frac{1}{1-q_*}(\lambda_*I-\bar{J})^{-1}-I,\]
which admits the eigen-decomposition
\begin{equation}\label{eq:Gamma}
\Gamma=O^\top \Lambda O, \qquad
\Lambda=\frac{1}{1-q_*}(\lambda_*I-\bar{D})^{-1}-I.
\end{equation}
In particular, $\Gamma$ is also orthogonally invariant in law. Let
$y^0 \in \R^n$ be an initialization of the AMP algorithm with entries
\begin{equation}\label{eq:initialy}
y_1^0,\ldots,y_n^0 \overset{iid}{\sim} \N(0,\sigma_*^2),
\end{equation}
where $\sigma_*^2$ is defined in (\ref{eq:qstar}). Then the AMP
algorithm is given by the iterations
\begin{align}
x^t&=\frac{1}{1-q_*}\tanh(h+y^{t-1})-y^{t-1}, \label{eq:AMPsmall1}\\
y^t&=\Gamma x^t. \label{eq:AMPsmall2}
\end{align}
An approximate solution of the TAP equations (\ref{eq:TAP}) is obtained from
the iterates of this algorithm as
$m^t=(1-q_*)(x^t+y^{t-1})=\tanh(h+y^{t-1})$.
For any fixed point $(x,y) \in \R^n \times \R^n$ of this algorithm, it is easily
checked that $m=(1-q_*)(x+y)=\tanh(h+y)$ exactly satisfies (\ref{eq:TAP}).

Applying the diagonalization $\Gamma=O^\top \Lambda O$ in (\ref{eq:Gamma}),
let us write the AMP iterations in an expanded form
\begin{align}
x^t&=\frac{1}{1-q_*}\tanh(h+y^{t-1})-y^{t-1},\label{eq:AMPx}\\
s^t&=Ox^t,\label{eq:AMPs}\\
y^t&=O^\top \Lambda s^t.\label{eq:AMPy}
\end{align}
For each fixed $t \geq 1$, we define the sigma-field (in the
probability space of $O$)
\begin{equation}\label{eq:Gt}
\cG_t=\cG\Big(y^0,x^1,s^1,y^1,\ldots,x^t,s^t,y^t\Big)
\end{equation}
generated by all iterates of (\ref{eq:AMPx}--\ref{eq:AMPy})
up to $y^t$. The proof of Theorem \ref{thm:replicasymmetric} will compute the
first and second moments of $Z$ conditioned on $\cG_t$.

A key property of this algorithm is that the scalar function
$f(h,y)=(1-q_*)^{-1}\tanh(h+y)-y$ applied entrywise in (\ref{eq:AMPx}) is
divergence-free in $y$, in the sense
\begin{equation}\label{eq:divergencefree}
\E[\partial_y f(\H,\sigma_*\G)]=0
\end{equation}
for independent random variables $\H \sim \mu_H$ and $\G \sim \N(0,1)$,
which follows from the definition of $q_*$ in \prettyref{eq:qstarunscaled}.
This substantially simplifies the state evolution that describes the AMP
iterates $x^t,y^t$---discussed in the next section---when $\bar{J}$ is a
non-Gaussian orthogonally invariant couplings matrix.

\subsection{State evolution for AMP}

The state evolution for general AMP algorithms of this form was described in
\cite{opper2016theory,ccakmak2019memory,fan2020approximate}.
We first review the specialization of
these results to the specific algorithm (\ref{eq:AMPx}--\ref{eq:AMPy}).
Proofs are deferred to Appendix \ref{appendix:AMP}.

Define
\begin{equation}\label{eq:kappadelta}
\kappa_*=\lim_{n \to \infty} \frac{1}{n}\Tr \Gamma^2,
\qquad \delta_*=\sigma_*^2/\kappa_*.
\end{equation}
These quantities are given more explicitly as follows.
\begin{proposition}\label{prop:kappadelta}
We have
\begin{equation}\label{eq:kappastar}
\kappa_*=\frac{1}{1-(1-q_*)^2\bar{R}'(1-q_*)}-1
\end{equation}
and
\begin{equation}\label{eq:deltastar}
\delta_*=\frac{q_*}{(1-q_*)^2}-\sigma_*^2
=\E\left[\left(\frac{1}{1-q_*}\tanh(\H+\sigma_* \G)-\sigma_*
\G\right)^2\right]
\end{equation}
for independent random variables $\H \sim \mu_H$ and $\G \sim \N(0,1)$.
\end{proposition}

Under Assumption \ref{assump:main}, let
$\H \sim \mu_H$ and $\Y_0 \sim \N(0,\sigma_*^2)$ be 
independent of each other. Then, iteratively for each $s=1,\ldots,t$, set
\begin{align}
\X_s&=\frac{1}{1-q_*}\tanh(\H+\Y_{s-1})-\Y_{s-1},\label{eq:XYrelation}\\
\Delta_s &= \E[(\X_1,\ldots,\X_s)(\X_1,\ldots,\X_s)^\top],
\label{eq:Deltadef}
\end{align}
and draw $\Y_s$ independently of $(\H,\Y_0)$ so that
$(\Y_1,\ldots,\Y_s) \sim \calN(0,\kappa_* \Delta_s)$.
This defines a joint law for the variables
$(\H,\Y_0,\Y_1,\ldots,\Y_t,\X_1,\ldots,\X_t)$, for any $t \geq 1$.

\begin{theorem}\label{thm:SE}
Fix any $t \geq 1$, and let $Y_t=(y^1,\ldots,y^t) \in \R^{n \times t}$
and $X_t=(x^1,\ldots,x^t) \in \R^{n \times t}$ collect the iterates of
(\ref{eq:AMPx}--\ref{eq:AMPy}), starting from the initialization
(\ref{eq:initialy}). Then, under Assumption \ref{assump:main},
almost surely as $n \to \infty$, the empirical distribution of rows of
$(h,y^0,Y_t,X_t)$ satisfies the convergence
\begin{equation}\label{eq:AMPconvergence}
\frac{1}{n}\sum_{i=1}^n
\delta_{(h_i,y_i^0,y_i^1,\ldots,y_i^t,x_i^1,\ldots,x_i^t)}
\to (\H,\Y_0,\Y_1,\ldots,\Y_t,\X_1,\ldots,\X_t)
\end{equation}
weakly and in $p^\text{th}$ moment for each fixed order $p \geq 1$.

Furthermore, $\Delta_t$ is non-singular, and almost surely as $n \to \infty$,
\begin{align}
n^{-1}X_t^\top X_t \to
\E[(\X_1,\ldots,\X_t)(\X_1,\ldots,\X_t)^\top]&=\Delta_t,\label{eq:XX}\\
n^{-1}Y_t^\top Y_t \to
\E[(\Y_1,\ldots,\Y_t)(\Y_1,\ldots,\Y_t)^\top]&=\kappa_* \Delta_t,\label{eq:YY}\\
n^{-1}X_t^\top Y_t \to
\E[(\X_1,\ldots,\X_t)(\Y_1,\ldots,\Y_t)^\top]&=0.\label{eq:XY}
\end{align}
\end{theorem}

By definition, the second-moment matrix $\Delta_t$ in Theorem \ref{thm:SE} is the
upper-left $t \times t$ submatrix of $\Delta_{t+1}$. Thus it is unambiguous to
write the entries of these matrices as
\[\Delta_t=(\delta_{ss'})_{1 \leq s,s' \leq t}.\]
For our purposes, we will require only the
following property of the entries of $\Delta_t$.

\begin{proposition}\label{prop:AMPconvergent}
In the setting of Theorem \ref{thm:SE}, for some $\beta_0=\beta_0(\mu_D)>0$
and all $\beta \in (0,\beta_0)$, we have
\[\delta_{tt}=\delta_* \text{ and } \kappa_*\delta_{tt}=\sigma_*^2
\text{ for all } t \geq 1, \qquad
\lim_{\min(s,t) \to \infty} \delta_{st}=\delta_*, \qquad \lim_{\min(s,t)
\to \infty} \kappa_*\delta_{st}=\sigma_*^2.\]
\end{proposition}
\noindent
Thus the algorithm (\ref{eq:AMPx}--\ref{eq:AMPy}) is convergent for sufficiently
small $\beta$,\footnote{It is shown in \cite{ccakmak2019memory} that this
convergence in fact holds in the entirety of a high-temperature region defined
by an Almeida-Thouless type condition for stability of the replica-symmetric
phase, which depends on $\mu_D$ and $\mu_H$.}
in the sense
\begin{align*}
\lim_{\min(s,t) \to \infty} \left(\lim_{n \to \infty}
\frac{1}{n}\|x^t-x^s\|^2\right)&=\lim_{\min(s,t) \to \infty}
(\delta_{ss}+\delta_{tt}-2\delta_{st})=0,\\
\lim_{\min(s,t) \to \infty} \left(\lim_{n \to \infty}
\frac{1}{n}\|y^t-y^s\|^2\right)&=\lim_{\min(s,t) \to \infty}
\kappa_*(\delta_{ss}+\delta_{tt}-2\delta_{st})=0.
\end{align*}

Defining $S_t=(s^1,\ldots,s^t)=OX_t$, where the second equality holds
by (\ref{eq:AMPs}), the convergence (\ref{eq:XX}) implies that
\[n^{-1}X_t^\top X_t=n^{-1}S_t^\top S_t \to \Delta_t.\]
A second important property of the memory-free dynamics
(\ref{eq:AMPx}--\ref{eq:AMPy}) is the following more general statement.

\begin{proposition}\label{prop:AMPfree}
In the setting of Theorem \ref{thm:SE},
fix any $t \geq 1$, and let $X_t=(x^1,\ldots,x^t) \in \R^{n \times t}$ and
$S_t=(s^1,\ldots,s^t) \in \R^{n \times t}$ collect the iterates of
(\ref{eq:AMPx}--\ref{eq:AMPs}). Let $f:\R \to \R$ be any function which is
continuous and bounded in a neighborhood of $\supp(\mu_{\bar{D}})$,
and define $f(\bar{J})$ by the functional calculus.
Then almost surely as $n \to \infty$,
\[n^{-1}X_t^\top f(\bar{J})X_t=n^{-1} S_t^\top f(\bar{D})S_t
\to \Delta_t \cdot \int f(x)\mu_{\bar{D}}(dx).\]
\end{proposition}
Informally, this states that for large $n$,
\[n^{-1}X_t^\top f(\bar{J})X_t
\approx n^{-1} X_t^\top X_t \cdot n^{-1} \Tr f(\bar{J}).\]
Thus, in a certain sense, the AMP iterates $X_t$ are ``free'' of
the couplings matrix $\bar{J}$, despite being dependent on $\bar{J}$.
This result is a consequence of the divergence-free property
(\ref{eq:divergencefree}), and it follows from the
state evolution analysis in \cite[Lemma A.4(b)]{fan2020approximate}.
We provide a proof in Appendix~\ref{appendix:AMP}.

Finally, we record here the leading-order behaviors of the above constants
$q_*,\sigma_*^2,\lambda_*,\kappa_*,\delta_*$ for small $\beta$.

\begin{proposition}\label{prop:smallbeta}
Under Assumption \ref{assump:main}, let $O(f(\beta,z))$ denote a quantity having
magnitude at most $C \cdot f(\beta,z)$, for some constants $C,\beta_0>0$
depending only on $\mu_D$ and for all $\beta \in (0,\beta_0)$ and
$z \in (0,1)$. Then
\begin{equation}\label{eq:Rsmallbeta}
\bar{R}(z)=\beta^2z\big(1+O(\beta z)\big), \quad
\bar{R}'(z)=\beta^2\big(1+O(\beta z)\big), \quad \bar{R}''(z)=O(\beta^3)
\end{equation}
and
\[q_*=\E[\tanh(\H)^2]+O(\beta^2),
\quad \sigma_*^2=\beta^2q_*+O(\beta^3),
\quad \lambda_*=\frac{1}{1-q_*}+\beta^2(1-q_*)\big(1+O(\beta(1-q_*))\big)\]
\[\kappa_*=\beta^2(1-q_*)^2\big(1+O(\beta(1-q_*))\big),
\quad \delta_*=\frac{q_*}{(1-q_*)^2}+O(\beta^2).\]
\end{proposition}

\subsection{Conditioning and large deviations for Haar-orthogonal matrices}
\label{sec:haar}

We collect here several results on the conditioning of Haar-orthogonal
matrices, and large deviations for integrals over the orthogonal group.

\begin{proposition}[Lemma 4 of \cite{rangan2019vector}]
\label{prop:Haarconditioning}
Let $A,B \in \R^{n \times k}$ be deterministic matrices of rank $k$, such
that $A=QB$ for some orthogonal matrix $Q \in \O(n)$.
Let $V_{A^\perp},V_{B^\perp} \in \R^{n \times (n-k)}$ be matrices
with orthonormal columns spanning the orthogonal complements of the column
spans of $A$ and $B$, respectively. Let $O \sim \Haar(\O(n))$. Then the law of
$O$ conditioned on the event $A=OB$ is given by
\[O|_{A=OB}\overset{L}{=}
V_{A^\perp} \tilde{O}V_{B^\perp}^\top+A(A^\top A)^{-1}B^\top
=V_{A^\perp} \tilde{O}V_{B^\perp}^\top+A(B^\top B)^{-1}B^\top,\]
where $\tilde{O} \sim \Haar(\O(n-k))$.
\end{proposition}

\begin{proposition}\label{prop:Ointegralrank1}
Let $O \sim \Haar(\O(n))$. Let $D \in \R^{n \times n}$ be a deterministic
symmetric matrix whose eigenvalue distribution satisfies Assumption
\ref{assump:main}(b) as $n \to \infty$. Let $\mu_D$ be its limit eigenvalue
distribution, let $d_+=\max(x:x \in \supp(\mu_D))$, and let $G(z)$ be the Cauchy
transform of $\mu_D$. Fix any constants $C,\eps>0$, and define the domain
\[\Omega_n=\left\{(a,b) \in \R^n \times \R^n:\;
0<\frac{\|a\|^2}{n} \leq G(d_++\eps)-\eps,\;
\frac{\|b\|^2}{n} \leq C\right\}.\]
Then
\begin{align}
\lim_{n \to \infty} \sup_{(a,b) \in \Omega_n}
\left|\frac{1}{n}\log \E\left[
\exp\left(b^\top Oa+\frac{a^\top O^\top DOa}{2}\right)\right]
-\frac{1}{2}E_n(a,b) \right|=0,
\label{eq:Ointegralrank1}
\end{align}
where
\begin{equation}
E_n(a,b)=\inf_{\gamma \geq d_++\eps} \left\{
\frac{\gamma \|a\|^2}{n}+\frac{b^\top (\gamma I-D)^{-1}b}{n}
-\frac{1}{n}\log \det (\gamma I-D)
-\left(1+\log \frac{\|a\|^2}{n}\right)\right\}.\label{eq:Ealpha}
\end{equation}
\end{proposition}

\begin{proposition}\label{prop:Ointegralrank2}
Let $O$, $D$, $\mu_D$, $d_+$, and $G(z)$ be as in Proposition
\ref{prop:Ointegralrank1}. Fix any constants $C,\eps>0$, and define the domains
\begin{equation}\label{eq:Deps}
\cD_\eps=\left\{(\gamma,\nu,\rho) \in \R^3:
\begin{pmatrix} \gamma & \nu \\ \nu & \rho \end{pmatrix} \succeq
(d_++\eps)I_{2 \times 2} \right\},
\end{equation}
\[\Omega_n=\left\{(a,b,c,d) \in (\R^n)^4:\;
0 \prec \frac{1}{n}\begin{pmatrix} \|a\|^2 & a^\top c \\
a^\top c & \|c\|^2 \end{pmatrix} \preceq 
\Big(G(d_++\eps)-\eps\Big)I_{2 \times 2},\;\;\frac{\|b\|^2}{n},\frac{\|d\|^2}{n}
\leq C\right\}.\]
Then
\begin{equation}
\lim_{n \to \infty} \sup_{(a,b,c,d) \in \Omega_n}
\left|\frac{1}{n}\log \E\left[\exp\left(b^\top Oa+d^\top Oc
+\frac{a^\top O^\top DOa}{2}+\frac{c^\top O^\top DOc}{2}\right) \right]
-\frac{1}{2}E_n(a,b,c,d)\right|=0
\label{eq:Ointegralrank2}
\end{equation}
where
\begin{align}
E_n(a,b,c,d)&=\inf_{(\gamma,\nu,\rho)\in \cD_\eps} \Bigg\{\frac{1}{n} \Tr
\begin{pmatrix} \gamma & \nu \\ \nu & \rho \end{pmatrix}
\begin{pmatrix} \|a\|^2 & a^\top c \\ a^\top c & \|c\|^2 \end{pmatrix}
+\frac{1}{n}\begin{pmatrix} b \\ d \end{pmatrix}^\top
\begin{pmatrix} \gamma I-D & \nu I \\ \nu I & \rho I-D \end{pmatrix}^{-1}
\begin{pmatrix} b \\ d \end{pmatrix}\nonumber\\
&\hspace{0.5in}-\frac{1}{n}\log \det \begin{pmatrix} \gamma I-D &
\nu I \\ \nu I & \rho I-D \end{pmatrix}
-\left(2+\log \det \frac{1}{n}\begin{pmatrix} \|a\|^2 & a^\top c \\ a^\top c &
\|c\|^2 \end{pmatrix}\right)\Bigg\}.\label{eq:En2}
\end{align}
\end{proposition}

When $b=d=0$, the expectations evaluated in
Propositions \ref{prop:Ointegralrank1} and \ref{prop:Ointegralrank2} are
finite-rank HCIZ integrals over the orthogonal group, and
such results were obtained in \cite[Theorems 2 and 7]{guionnet2005fourier}. The above propositions
extend these results to $b,d \neq 0$, and also establish the approximations in a
more uniform sense. We note that the content of Proposition
\ref{prop:Ointegralrank1} for $b \neq 0$ is essentially the calculation of the
limit free energy in the spherical analogue of the model (\ref{eq:intromodel})
with external field, and we discuss this in \prettyref{app:sphere}.

For $b=d=0$, asymptotic versions of the infima in
Propositions \ref{prop:Ointegralrank1} and \ref{prop:Ointegralrank2}
may be explicitly evaluated, and we record these evaluations here.

\begin{proposition}\label{prop:infgamma}
Let $\mu_D$ be a compactly supported probability distribution on $\R$.
Let $G(z)$ and $R(z)$ be the Cauchy- and R-transforms of $\mu_D$,
and let $d_+=\max(x:x \in \supp(\mu_D))$.
\begin{enumerate}[(a)]
\item Suppose that $\alpha \in (0,G(d_+))$. Then
\[\inf_{\gamma>d_+} \gamma\alpha-\int \log(\gamma-x)\mu_D(dx)
-(1+\log \alpha)=\int_0^\alpha R(z)dz\]
and the infimum is achieved at $\gamma=G^{-1}(\alpha)=R(\alpha)+1/\alpha$.
\item Suppose that $A \in \R^{2 \times 2}$ is symmetric
and satisfies $0 \prec A \prec G(d_+)I$. Define $f(A) \in \R^{2 \times 2}$
for any function $f:(0,G(d_+)) \to \R$ by the functional calculus. Let
\[\cD_+=\left\{(\gamma,\nu,\rho) \in \R^3:
\begin{pmatrix} \gamma & \nu \\ \nu & \rho \end{pmatrix} \succ
d_+I_{2 \times 2} \right\}.\]
Then
\begin{align*}
&\inf_{(\gamma,\nu,\rho) \in \cD_+} \Tr \begin{pmatrix} \gamma & \nu \\
\nu & \rho \end{pmatrix} A-\int \log \det \begin{pmatrix} \gamma-x & \nu \\
\nu & \rho-x \end{pmatrix} \mu_D(dx)-(2+\log \det A)=\Tr f(A),
\end{align*}
where $f(\alpha)=\int_0^\alpha R(z)dz$. The infimum is achieved at
$\begin{pmatrix} \gamma & \nu \\ \nu & \rho \end{pmatrix}
=G^{-1}(A)=R(A)+A^{-1}$.
\end{enumerate}
\end{proposition}

We prove Propositions \ref{prop:Ointegralrank1}, \ref{prop:Ointegralrank2}, and
\ref{prop:infgamma} in Appendix \ref{appendix:HCIZ}, building
on the large-deviations arguments of \cite{guionnet2005fourier}.

\section{Conditional first moment}\label{sec:firstmoment}

Let $Z$ be the partition function in (\ref{eq:gibbsmeasure}), and let $\cG_t$ be
the sigma-field defined by (\ref{eq:Gt}). We show in this section the following
result.

\begin{lemma}\label{lemma:firstmoment}
In the setting of Theorem \ref{thm:replicasymmetric},
\[\lim_{t \to \infty} \lim_{n \to \infty} \frac{1}{n}\log \E[Z \mid \cG_t]
=\Psi_{\RS},\]
where the inner limit as $n \to \infty$ exists almost surely for each fixed $t$.
\end{lemma}

\subsection{Derivation of the variational formula}

For scalar arguments $\gamma>\bar{d}_+$ and $u,U \in \R$, and vector arguments
$v,w,V,W \in \R^t$ with $\|v\|^2+\|w\|^2<1$, we define the function
\begin{align}
&\Phi_{1,t}(u,v,w;\gamma,U,V,W)\nonumber\\
&=\E\Big[\log 2\cosh\Big(U \cdot \H+V^\top \Delta_t^{-1/2}
(\X_1,\ldots,\X_t)+\kappa_*^{-1/2}W^\top \Delta_t^{-1/2}(\Y_1,\ldots,\Y_t)
\Big)\Big]\nonumber\\
&\hspace{0.2in}-u \cdot U-v^\top V-w^\top W
+u+\bar{R}(1-q_*)\kappa_*^{-1/2}v^\top w
+\frac{\lambda_*-\bar{R}(1-q_*)\kappa_*^{-1}}{2}\|w\|^2\nonumber\\
&\hspace{0.2in}+\frac{1}{2}\cF(\gamma)
\|v-\kappa_*^{-1/2}w\|^2
+\frac{1}{2}\cH(\gamma,1-\|v\|^2-\|w\|^2)\label{eq:Phi1}
\end{align}
and the variational formula
\begin{equation}\label{eq:Psi1}
\Psi_{1,t}=\mathop{\sup_{u \in \R}}_{v,w \in \R^t:\|v\|^2+\|w\|^2<1}
\inf_{\gamma>\bar{d}_+} \inf_{U \in \R,\;V,W \in \R^t}
\Phi_{1,t}(u,v,w;\gamma,U,V,W).
\end{equation}

Here, the random variables
$(\H,\Y_1,\ldots,\Y_t,\X_1,\ldots,\X_t)$ and the positive-definite matrix
$\Delta_t$ are as described
in Theorem \ref{thm:SE}, and the functions $\cF$ and $\cH$ are given by
\begin{align}
\cF(\gamma)&\triangleq \cF_{22}(\gamma)-\cF_{12}(\gamma)^\top
\cF_{11}(\gamma)^{-1}\cF_{12}(\gamma),
\label{eq:Fgamma}\\
\cH(\gamma,\alpha)&\triangleq \gamma \alpha-\int \log(\gamma-x)\mu_{\bar{D}}(dx)
-(1+\log \alpha)\label{eq:Hgamma}
\end{align}
where we set
\begin{equation}\label{eq:LambdaThetax}
\lambda(x)\triangleq \frac{1}{(1-q_*)(\lambda_*-x)}-1,
\qquad \theta(x)\triangleq x+\frac{\bar{R}(1-q_*)}{\kappa_*}
\left(1-\frac{1}{(1-q_*)(\lambda_*-x)}\right),
\end{equation}
and
\begin{align}
\cF_{11}(\gamma)&\triangleq \int \frac{1}{\gamma-x}
\begin{pmatrix} 1 & \lambda(x) \\ \lambda(x) & \lambda(x)^2 \end{pmatrix}
\mu_{\bar{D}}(dx) \in \R^{2 \times 2}\label{eq:F11}\\
\cF_{12}(\gamma)&\triangleq \int \frac{1}{\gamma-x}\begin{pmatrix} \theta(x) \\
\lambda(x)\theta(x) \end{pmatrix}
\mu_{\bar{D}}(dx) \in \R^2 \label{eq:F12}\\
\cF_{22}(\gamma)&\triangleq \int \frac{1}{\gamma-x}\theta(x)^2\mu_{\bar{D}}(dx).
\label{eq:F22}
\end{align}
Note that under Assumption \ref{assump:main}(b), $\mu_{\bar{D}}$ is supported on
at least two points, and $\lambda_*>\bar{d}_+$ by definition so
$x \mapsto \lambda(x)$ is one-to-one on $\supp(\mu_{\bar{D}})$. As a result, $\calF_{11}(\gamma)$ is strictly positive-definite and
invertible for $\gamma>\bar{d}_+$ and thus $\cF(\gamma)$ is well-defined.

\begin{lemma}\label{lemma:Psi1correct}
In the setting of Theorem \ref{thm:replicasymmetric}, for any fixed $t \geq 1$,
almost surely
\[\lim_{n \to \infty} \frac{1}{n}\log \E[Z \mid \cG_t]=\Psi_{1,t}.\]
\end{lemma}
\begin{proof}
Recall the $n \times t$ matrices $X_t=(x^1,\ldots,x^t)$,
$Y_t=(y^1,\ldots,y^t)$, and $S_t=(s^1,\ldots,s^t)$ which collect the AMP
iterates. We fix $t$ and write $\cG,X,Y,S,\Delta$ for
$\cG_t,X_t,Y_t,S_t,\Delta_t$. From the definition of $Z$ in
(\ref{eq:gibbsmeasure}),
\begin{equation}\label{eq:fnsigma}
\E[Z \mid \cG]=\sum_{\sigma \in \{+1,-1\}^n}
\exp\left(h^\top \sigma+\frac{n}{2} \cdot f_n(\sigma)\right),
\quad f_n(\sigma) \triangleq \frac{2}{n}\log
\E\left[\exp\left(\frac{1}{2}\sigma^\top O^\top\bar{D} O\sigma\right)
\;\bigg|\;\cG\right].
\end{equation}
The function $f_n(\sigma)$ is well-defined for any $\sigma \in \R^n$.
We first approximate $f_n(\sigma)$ over the sphere where $\|\sigma\|^2=n$.\\

\noindent {\bf Conditional law of $O$.}
Theorem \ref{thm:SE} guarantees that $\Delta$ is non-singular. The assumption of
positive variance in (\ref{eq:varianceone}) and the definitions of $\Gamma$ and
$\kappa_*$ in (\ref{eq:Gamma}) and (\ref{eq:kappadelta}) ensure that
$\kappa_*>0$. Then applying (\ref{eq:XX}--\ref{eq:XY}), 
almost surely for all large $n$, $n^{-1}(X,Y)^\top (X,Y) \in \R^{2t \times 2t}$
is also non-singular and $(X,Y) \in \R^{n \times 2t}$
has full column rank $2t$. Furthermore, we have the bounds
\begin{equation}\label{eq:XYSbounds}
\limsup_{n \to \infty} n^{-1/2}\|X\|<\infty,\quad
\limsup_{n \to \infty} n^{-1/2}\|Y\|<\infty,\quad
\limsup_{n \to \infty} n^{-1/2}\|S\|<\infty,
\end{equation}
which follow from $\|n^{-1}X^\top X\|=\|n^{-1}S^\top S\| \to \|\Delta\|$ 
and $\|n^{-1}Y^\top Y\| \to \|\kappa_*\Delta\|$.

Conditional on $\cG$, the law of $O$ is that of a Haar-orthogonal matrix
conditioned on the event
\[(S,\Lambda S)=O(X,Y).\]
By Proposition \ref{prop:Haarconditioning}, we may represent this conditional
law of $O$ as
\[O|_{\cG} \overset{L}{=} V_{(S,\Lambda S)^\perp}
\tilde{O}V_{(X,Y)^\perp}^\top
+(S,\Lambda S)\begin{pmatrix} X^\top X & X^\top Y \\
Y^\top X & Y^\top Y \end{pmatrix}^{-1}(X,Y)^\top,\]
where $V_{(X,Y)^\perp},V_{(S,\Lambda S)^\perp}\in \R^{n \times
(n-2t)}$ have orthonormal columns orthogonal to the column spans
of $(X,Y) \in \R^{n \times 2t}$ and $(S,\Lambda S) \in \R^{n \times 2t}$
respectively, and $\tilde{O} \sim \Haar(\O(n-2t))$ is an independent
Haar-orthogonal matrix.
Let us write as shorthand
\[V=V_{(S,\Lambda S)^\perp}.\]
For any vector $\sigma \in \R^n$, let us denote
\begin{equation}\label{eq:sigmadecomp}
\sigma_\perp=V_{(X,Y)^\perp}^\top \sigma \in \R^{n-2t},
\qquad \sigma_\parallel=(S,\Lambda S)\begin{pmatrix} X^\top X & X^\top Y \\
Y^\top X & Y^\top Y \end{pmatrix}^{-1}(X,Y)^\top \sigma \in \R^n.
\end{equation}
This yields the equality in conditional law
$O\sigma|_{\cG} \overset{L}{=} V \tilde{O} \sigma_\perp+\sigma_\parallel$, so 
\prettyref{eq:fnsigma} reduces to
\begin{equation}\label{eq:firstmomentform}
f_n(\sigma)=
\frac{1}{n}\sigma_\parallel^\top \bar{D}\sigma_\parallel
+\frac{2}{n}\log
\E\left[\exp\left(\frac{1}{2}\sigma_\perp^\top \tilde{O}^\top V^\top
\bar{D}V \tilde{O}\sigma_\perp+\sigma_\parallel^\top \bar{D}V
\tilde{O}\sigma_\perp\right)\right].
\end{equation}\\

\noindent {\bf Expectation over $\tilde{O}$.}
We first restrict to the domain
\[U_n=\{\sigma \in \R^n:\|\sigma\|^2=n,\;\sigma_\perp \neq 0\}\]
and evaluate the expectation over
$\tilde{O} \sim \Haar(\O(n-2t))$ using Proposition \ref{prop:Ointegralrank1}.
Throughout the proof, we write $r_n(\sigma)$ to indicate any $\sigma$-dependent
scalar, vector, or matrix remainder term with dimension independent of $n$,
satisfying the uniform convergence almost surely
\begin{equation}\label{eq:rnuniform}
\lim_{n \to \infty} \sup_{\sigma \in U_n} \|r_n(\sigma)\|=0,
\end{equation}
and changing from instance to instance.
We check the conditions of Proposition \ref{prop:Ointegralrank1}:
\begin{itemize}
\item The matrix $V=V_{(S,\Lambda S)^\perp}$ has $n-2t$ orthonormal
columns, where $t$ is independent of $n$. Then by
Assumption \ref{assump:main}(b) and Weyl eigenvalue interlacing,
as $n \to \infty$, the empirical eigenvalue distribution of
$V^\top \bar{D} V$ has the same weak limit $\mu_{\bar{D}}$ as that of $\bar{D}$.
Furthermore, from the conditions on $\max(d_1,\ldots,d_n)$ and
$\min(d_1,\ldots,d_n)$ in Assumption \ref{assump:main}(b),
the largest eigenvalue of $V^\top \bar{D} V$ also converges to $\bar{d}_+$,
and the smallest eigenvalue remains bounded away from $-\infty$.
\item Take $a=\sigma_\perp$ in Proposition \ref{prop:Ointegralrank1}.
Applying (\ref{eq:barGR}), we have $\bar{G}(\bar{d}_+)
=\beta^{-1}G(d_+)$, where $G(d_+) \in (0,\infty]$ depends only
on $\mu_D$. Then for some $\beta_0=\beta_0(\mu_D)>0$,
any $\beta \in (0,\beta_0)$, and any sufficiently small constant
$\eps>0$, we have 
\begin{equation}
\bar{G}(\bar{d}_++\eps)-\eps>1
\label{eq:Glarge}
\end{equation}
so that $\|\sigma_\perp\|^2/n \leq \|\sigma\|^2/n=1
<\bar{G}(\bar{d}_++\eps)-\eps$.
\item Take $b=V^\top \bar{D}\sigma_\parallel$ in
Proposition \ref{prop:Ointegralrank1}. Observe that $(S,\Lambda S)=O(X,Y)$, so
$\sigma_\parallel=O\Pi_{X,Y}\sigma$ where
$\Pi_{X,Y}=I-V_{(X,Y)^\perp}V_{(X,Y)^\perp}^\top \in \R^{n \times
n}$ is the orthogonal projection onto the column span of $(X,Y)$.
Then $\|V^\top\bar{D}\sigma_\parallel\|^2/n
\leq \|\bar{D}\|^2 \cdot \|\sigma_\parallel\|^2/n
\leq \|\bar{D}\|^2 \cdot \|\sigma\|^2/n=\|\bar{D}\|^2$.
\end{itemize}
Thus Proposition \ref{prop:Ointegralrank1} (applied with dimension
$n-2t$) yields uniformly over $\sigma \in U_n$
\begin{equation}\label{eq:Oexpectation}
f_n(\sigma)=\frac{1}{n} \sigma_\parallel^\top
\bar{D}\sigma_\parallel+E_n(\sigma)+r_n(\sigma)
\end{equation}
where
\begin{align}
E_n(\sigma)&=\inf_{\gamma \geq \bar{d}_++\eps} \Bigg\{
\frac{\gamma \|\sigma_\perp\|^2}{n}
+\frac{\sigma_\parallel^\top\bar{D}V(\gamma I-V^\top \bar{D}V)^{-1}
V^\top\bar{D}\sigma_\parallel}{n}\nonumber\\
&\hspace{2in} -\frac{1}{n}\log\det (\gamma
I-V^\top \bar{D}V)-\left(1+\log \frac{\|\sigma_\perp\|^2}{n}\right)\Bigg\}.
\label{eq:Ensigma}
\end{align}\\

\noindent {\bf Approximation by $v,w$.}
For $\sigma \in U_n$, define the low-dimensional linear functionals
\begin{equation}\label{eq:uvw}
u(\sigma)=\frac{1}{n}h^\top \sigma, \qquad
\begin{pmatrix} v(\sigma) \\ w(\sigma) \end{pmatrix}
=\left[\frac{1}{n}\begin{pmatrix} X^\top X & X^\top Y \\ Y^\top X & Y^\top Y
\end{pmatrix}\right]^{-1/2} \cdot \frac{1}{n} (X,Y)^\top \sigma
\end{equation}
where $u(\sigma) \in \R$ and $v(\sigma),w(\sigma) \in \R^t$. Note that
\begin{equation}\label{eq:vwbound}
\|v(\sigma)\|^2+\|w(\sigma)\|^2=\frac{1}{n}\|\Pi_{(X,Y)} \sigma\|^2
=1-\frac{\|\sigma_\perp\|^2}{n}<1.
\end{equation}
Let us approximate the terms of (\ref{eq:Oexpectation}) by functions of
$v(\sigma)$ and $w(\sigma)$.

We begin with $\sigma_\parallel^\top \bar{D} \sigma_\parallel/n$:
Applying (\ref{eq:XX}--\ref{eq:XY}) to (\ref{eq:sigmadecomp}),
\begin{align}
\sigma_\parallel
&=(S,\Lambda S)\left[\frac{1}{n}\begin{pmatrix} X^\top X & X^\top Y \\ Y^\top X
& Y^\top Y \end{pmatrix}\right]^{-1/2}\begin{pmatrix} v(\sigma) \\
w(\sigma) \end{pmatrix}\nonumber\\
&=S \cdot \Delta^{-1/2}v(\sigma)
+\Lambda S \cdot (\kappa_*\Delta)^{-1/2}w(\sigma)+(S,\Lambda S) \cdot
r_n(\sigma).\label{eq:sigmaparallelapprox}
\end{align}
From the definition of $\lambda_*$ in (\ref{eq:lambdastar}) and the definition
of the Cauchy-transform in (\ref{eq:GR}), as $n \to \infty$,
\begin{equation}\label{eq:TrDG}
n^{-1}\Tr \bar{D}(\lambda_* I-\bar{D})^{-1}
=n^{-1}\Tr \Big[\lambda_*(\lambda_* I-\bar{D})^{-1}-I\Big]
\to \lambda_* \bar{G}(\lambda_*)-1=\lambda_*(1-q_*)-1.
\end{equation}
Differentiating the R-transform in (\ref{eq:GR}),
\[\bar{R}'(z)=\frac{1}{\bar{G}'(\bar{G}^{-1}(z))}+\frac{1}{z^2}.\]
Then applying the form of $\kappa_*$ in (\ref{eq:kappastar}), also
\begin{align}
n^{-1}\Tr \bar{D}(\lambda_* I-\bar{D})^{-2}
&=n^{-1}\Tr \Big[\lambda_*(\lambda_* I-\bar{D})^{-2}-(\lambda_*
I-\bar{D})^{-1}\Big]\nonumber\\
&\to -\lambda_* \bar{G}'(\lambda_*)-\bar{G}(\lambda_*)
=\lambda_*(\kappa_*+1)(1-q_*)^2-(1-q_*).\label{eq:TrDG2}
\end{align}
Let us write as a shorthand
\begin{equation}\label{eq:astar}
a_*\triangleq\bar{R}(1-q_*)=\lambda_*-\frac{1}{1-q_*}.
\end{equation}
Then in view of the definition of $\Gamma$ in (\ref{eq:Gamma}), applying (\ref{eq:meanzero}), (\ref{eq:TrDG}), and (\ref{eq:TrDG2}) yields
\[n^{-1} \Tr \bar{D} \to 0, \qquad
n^{-1}\Tr \bar{D}\Lambda \to a_*, \qquad
n^{-1}\Tr \bar{D}\Lambda^2 \to \lambda_*\kappa_*-a_*.\]
So Proposition \ref{prop:AMPfree} yields almost surely
\begin{equation}\label{eq:SDS}
\frac{1}{n}(S,\Lambda S)^\top \bar{D}(S,\Lambda S)
\to \begin{pmatrix} 0 & a_*\Delta \\ a_*\Delta &
(\lambda_*\kappa_*-a_*)\Delta \end{pmatrix}.
\end{equation}
Combining this with (\ref{eq:sigmaparallelapprox}), we obtain for the first term
of (\ref{eq:Oexpectation}) that
\begin{equation}\label{eq:Ifirstterm}
\frac{\sigma_\parallel^\top \bar{D} \sigma_\parallel}{n}
=\frac{2a_*}{\kappa_*^{1/2}}v(\sigma)^\top w(\sigma)
+\left(\lambda_*-\frac{a_*}{\kappa_*}\right)\|w(\sigma)\|^2+r_n(\sigma).
\end{equation}

Next, we approximate $E_n(\sigma)$ in (\ref{eq:Ensigma}) by approximating each term inside the infimum 
uniformly over $\gamma \geq \bar{d}_++\eps$ and
$\sigma \in U_n$.
Note that for all large $n$,
all eigenvalues of $V^\top \bar{D}V$ are contained in a compact interval
$\calK \subset (-\infty,\bar{d}_++\eps/2)$ that is disjoint from
$[\bar{d}_++\eps,\infty)$. Fixing $\gamma \geq \bar{d}_++\eps$,
the function $x \mapsto \log(\gamma-x)$ is bounded and
continuous on $\calK$, so by weak convergence in \prettyref{assump:main}(b),
\[\frac{1}{n}\log \det(\gamma I-V^\top \bar{D}V)
=\int \log (\gamma-x)\mu_{\bar{D}}(dx)+r_n(\gamma)\]
where $r_n(\gamma) \to 0$ as $n \to \infty$.
The function $\gamma \mapsto n^{-1}\log \det(\gamma
I-V^\top \bar{D}V)$ on the left side is uniformly Lipschitz
over $\gamma \geq \bar{d}_++\eps$ for all large $n$,
so by Arzel\`a-Ascoli, in fact $r_n(\gamma) \to 0$ uniformly in $\gamma$ over any compact
subset $\calK' \subset [\bar{d}_++\eps,\infty)$. For any $\delta>0$,
we may take a sufficiently large such compact subset $\calK_\delta'$ and bound
\[\Big|\log (\gamma-x)-\log \gamma\Big|
\leq |x| \cdot \frac{1}{|\gamma|-|x|}<\delta \text{ for all }
x \in \calK,\;\gamma \in [\bar{d}_++\eps,\infty) \setminus \calK_\delta'.\]
Then also $|r_n(\gamma)|<2\delta$ for all $\gamma \in [\bar{d}_++\eps,\infty)
\setminus \calK_\delta'$, implying that the convergence $r_n(\gamma) \to 0$ is
uniform over all $\gamma \geq \bar{d}_++\eps$. Then, applying
also $\|\sigma_\perp\|^2/n=1-\|v(\sigma)\|^2-\|w(\sigma)\|^2$ from
(\ref{eq:vwbound}) and recalling the function $\cH$ defined in (\ref{eq:Hgamma}), we obtain
\begin{equation}\label{eq:Hconvergence}
\frac{\gamma \|\sigma_\perp\|^2}{n}
-\frac{1}{n}\log \det(\gamma I-V^\top\bar{D}V)
-\left(1+\log \frac{\|\sigma_\perp\|^2}{n}\right)
=\cH(\gamma,1-\|v(\sigma)\|^2-\|w(\sigma)\|^2)+r_n(\gamma).
\end{equation}

To analyze the remaining second term of $E_n(\sigma)$ in
(\ref{eq:Ensigma}), let us introduce
\begin{equation}\label{eq:Wdef}
W=(S,\Lambda S)\begin{pmatrix} S^\top S & S^\top \Lambda S \\
S^\top \Lambda S & S^\top \Lambda^2 S \end{pmatrix}^{-1/2}
=(S,\Lambda S)\begin{pmatrix} X^\top X & X^\top Y \\
Y^\top X & Y^\top Y \end{pmatrix}^{-1/2} \in \R^{n \times 2t}
\end{equation}
whose columns are the orthogonalization of $(S,\Lambda S)$.
Then the columns of $(V,W)$ form a full orthonormal basis for $\R^n$. We write
$\Pi=VV^\top=I-WW^\top$ as the projection orthogonal to $(S,\Lambda S)$.
Applying (\ref{eq:sigmaparallelapprox}), (\ref{eq:SDS}) and
(\ref{eq:XX}--\ref{eq:XY}), observe that
\begin{align*}
&\begin{pmatrix} X^\top X & X^\top Y \\
Y^\top X & Y^\top Y\end{pmatrix}^{-1}
(S,\Lambda S)^\top \bar{D}\sigma_\parallel\\
&=\begin{pmatrix} \Delta^{-1} & 0 \\ 0 & (\kappa_*\Delta)^{-1}
\end{pmatrix} \begin{pmatrix} 0 & a_*\Delta \\ a_*\Delta &
(\lambda_*\kappa_*-a_*) \Delta \end{pmatrix}
\begin{pmatrix} \Delta^{-1/2}v(\sigma) \\ (\kappa_*\Delta)^{-1/2}w(\sigma)
\end{pmatrix}+r_n(\sigma)\\
&=\begin{pmatrix} 0 & a_*I \\ a_*\kappa_*^{-1}I &
(\lambda_*-a_*\kappa_*^{-1})I \end{pmatrix}
\begin{pmatrix} \Delta^{-1/2}v(\sigma) \\ (\kappa_*\Delta)^{-1/2}w(\sigma)
\end{pmatrix}+r_n(\sigma).
\end{align*}
Then applying
\[\Pi=I-WW^\top=I-(S,\Lambda S)\begin{pmatrix} X^\top X & X^\top Y \\
Y^\top X & Y^\top Y \end{pmatrix}^{-1}(S,\Lambda S)^\top,\]
we obtain
\begin{align*}
\Pi \bar{D}\sigma_\parallel
&=\Big(\bar{D}S-a_*\kappa_*^{-1}\Lambda S\Big) \cdot \Delta^{-1/2}v(\sigma)
+\Big(\bar{D}\Lambda S-a_*S-(\lambda_*-a_*\kappa_*^{-1})\Lambda S\Big)
\cdot (\kappa_*\Delta)^{-1/2}w(\sigma)\\
&\hspace{2in}+(\bar{D}S,\bar{D}\Lambda S,S,\Lambda S)r_n(\sigma).
\end{align*}
Substituting
\[\Lambda=\frac{1}{1-q_*}(\lambda_* I-\bar{D})^{-1}-I,
\qquad \bar{D}\Lambda=\frac{1}{1-q_*}\Big(\lambda_*(\lambda_* I-\bar{D})^{-1}
-I\Big)-\bar{D}\]
and applying the identity (\ref{eq:astar}) and some algebraic simplification,
\begin{equation}\label{eq:PiDsigmaparallel}
\Pi\bar{D}\sigma_\parallel
=\tilde{D}S \cdot \Big(\Delta^{-1/2}v(\sigma)
-(\kappa_*\Delta)^{-1/2}w(\sigma)\Big)
+(\bar{D}S,\bar{D}\Lambda S,S,\Lambda S)r_n(\sigma)
\end{equation}
where $\tilde{D}$ is the diagonal matrix
\begin{equation}\label{eq:Dtilde}
\tilde{D} \triangleq \bar{D}+\frac{a_*}{\kappa_*}\left(I-\frac{1}{1-q_*}
(\lambda_*I-\bar{D})^{-1}\right).
\end{equation}

Now let us apply $(V,W)^\top (V,W)=I$ to write
\begin{align*}
\begin{pmatrix} \gamma I-V^\top \bar{D} V &
-V^\top \bar{D} W \\ -W^\top \bar{D} V & \gamma I-W^\top \bar{D} W \end{pmatrix}
&=\begin{pmatrix} V^\top \\ W^\top \end{pmatrix}
(\gamma I-\bar{D})\begin{pmatrix} V & W \end{pmatrix}\\
&=\left[\begin{pmatrix} V^\top \\ W^\top \end{pmatrix}
(\gamma I-\bar{D})^{-1}\begin{pmatrix} V & W \end{pmatrix}\right]^{-1}\\
&=\begin{pmatrix} V^\top (\gamma I-\bar{D})^{-1} V &
V^\top (\gamma I-\bar{D})^{-1} W \\ W^\top (\gamma I-\bar{D})^{-1} V &
W^\top (\gamma I-\bar{D})^{-1} W \end{pmatrix}^{-1}.
\end{align*}
Equating the upper-left blocks and applying the Schur-complement formula to the
right side yields
\[\gamma I-V^\top \bar{D} V
=\Big[V^\top(\gamma I-\bar{D})^{-1}V-
V^\top (\gamma I-\bar{D})^{-1}W(W^\top(\gamma I-\bar{D})^{-1}W)^{-1}
W^\top(\gamma I-\bar{D})^{-1}V\Big]^{-1}.\]
Thus, recalling $\Pi=VV^\top$, the second term of (\ref{eq:Ensigma}) is
\begin{align}
&\frac{1}{n}\sigma_\parallel^\top \bar{D} V(\gamma I-V^\top \bar{D} V)^{-1}
V^\top \bar{D}\sigma_\parallel\nonumber\\
&=\frac{1}{n}\sigma_\parallel^\top \bar{D} \Pi(\gamma I-\bar{D})^{-1}\Pi
\bar{D}\sigma_\parallel
-\frac{1}{n}\sigma_\parallel^\top \bar{D} \Pi (\gamma I-\bar{D})^{-1}W
(W^\top (\gamma I-\bar{D})^{-1}W)^{-1}W^\top(\gamma I-\bar{D})^{-1}
\Pi \bar{D}\sigma_\parallel.\label{eq:schurcomplementdecomp}
\end{align}

We apply (\ref{eq:PiDsigmaparallel}) and Proposition \ref{prop:AMPfree}
to approximate these two terms: By Proposition \ref{prop:AMPfree}, we have
almost surely
\[\frac{1}{n}S^\top \tilde{D}(\gamma I-\bar{D})^{-1}\tilde{D}S
\to \cF_{22}(\gamma) \cdot \Delta\]
for each fixed $\gamma \geq \bar{d}_++\eps$, where $\cF_{22}(\gamma)$ is as
defined in (\ref{eq:F22}) and $\tilde D$ in \prettyref{eq:Dtilde}.
Applying (\ref{eq:XYSbounds}), the left side is a $t \times t$ matrix that is
entrywise uniformly Lipschitz as a function of $\gamma \geq \bar{d}_++\eps$ for
all large $n$. So this convergence is again uniform in $\gamma$ over any compact subset
$\calK' \subset [\bar{d}_++\eps,\infty)$ by Arzel\`a-Ascoli. For any
$\delta>0$, we may take a sufficiently large such subset $\calK_\delta'$ so that
the left side is entrywise bounded by
$\delta$ for all $\gamma$ outside $\calK_\delta'$. In all, we conclude the
above convergence is uniform over all $\gamma \geq \bar{d}_++\eps$. Since
\[\frac{1}{n}S^\top \tilde{D}
(\gamma I-\bar{D})^{-1}(\bar{D}S,\bar{D}\Lambda S,S,\Lambda S),\qquad
\frac{1}{n}(\bar{D}S,\bar{D}\Lambda S,S,\Lambda S)^\top
(\gamma I-\bar{D})^{-1}(\bar{D}S,\bar{D}\Lambda S,S,\Lambda S)\]
are also uniformly bounded over $\gamma \geq \bar{d}_++\eps$ for all large $n$,
this combined with (\ref{eq:PiDsigmaparallel}) shows for the first term of
(\ref{eq:schurcomplementdecomp}) that
\begin{equation}\label{eq:schurcompfirstterm}
\frac{1}{n}\sigma_\parallel^\top \bar{D} \Pi(\gamma I-\bar{D})^{-1}\Pi
\bar{D}\sigma_\parallel=\cF_{22}(\gamma)
\cdot \|v(\sigma)-\kappa_*^{-1/2}w(\sigma)\|^2+r_n(\sigma,\gamma)
\end{equation}
where $r_n(\sigma,\gamma) \to 0$ uniformly over $\gamma \geq \bar{d}_++\eps$ and
$\sigma \in U_n$ as $n \to \infty$.

For the second term of (\ref{eq:schurcomplementdecomp}), 
recalling $\Lambda = \frac{1}{1-q_*}(\lambda_*I-\bar{D})^{-1}-I$ from \prettyref{eq:Gamma} and again applying
Proposition \ref{prop:AMPfree}, we have
\[\frac{1}{n}(S,\Lambda S)^\top(\gamma I-\bar{D})^{-1}(S,\Lambda S)
\to \cF_{11}(\gamma) \otimes \Delta \in \R^{2t \times 2t}\]
where
\begin{align*}
\cF_{11}(\gamma)&=\lim_{n \to \infty}
\begin{pmatrix} \frac{1}{n}\Tr (\gamma I-\bar{D})^{-1} &
\frac{1}{n}\Tr (\gamma I-\bar{D})^{-1} \Lambda \\
\frac{1}{n}\Tr \Lambda (\gamma I-\bar{D})^{-1} &
\frac{1}{n}\Tr \Lambda (\gamma I-\bar{D})^{-1} \Lambda
\end{pmatrix},
\end{align*}
and this coincides with the matrix defined in (\ref{eq:F11}).
Then, recalling the form of $W$ from (\ref{eq:Wdef}),
\begin{align*}
W^\top(\gamma I-\bar{D})^{-1}W
&\to \begin{pmatrix} \Delta & 0 \\ 0 & \kappa_*\Delta \end{pmatrix}^{-1/2}
[\cF_{11}(\gamma) \otimes \Delta]
\begin{pmatrix} \Delta & 0 \\ 0 & \kappa_*\Delta \end{pmatrix}^{-1/2}\\
&=\left[\begin{pmatrix} 1 & 0 \\ 0 & \kappa_*^{-1/2} \end{pmatrix} \cF_{11}(\gamma)
\begin{pmatrix} 1 & 0 \\ 0 & \kappa_*^{-1/2} \end{pmatrix}\right] \otimes I.
\end{align*}
Similarly, for $\cF_{12}(\gamma)$ as defined in (\ref{eq:F12}),
\[\frac{1}{\sqrt{n}}W^\top(\gamma I-\bar{D})^{-1}\tilde{D}S \to
\begin{pmatrix} \Delta & 0 \\ 0 & \kappa_*\Delta \end{pmatrix}^{-1/2}
[\cF_{12}(\gamma) \otimes \Delta]
=\left[\begin{pmatrix} 1 & 0 \\ 0 & \kappa_*^{-1/2} \end{pmatrix}
\cF_{12}(\gamma)\right] \otimes \Delta^{1/2}.\]
Thus
\[\frac{1}{n}S^\top \tilde{D}(\gamma I-\bar{D})^{-1} W
(W^\top(\gamma I-\bar{D})^{-1}W)^{-1}W^\top(\gamma I-\bar{D})^{-1}
\tilde{D}S \to \cF_{12}(\gamma)^\top \cF_{11}(\gamma)^{-1}\cF_{12}(\gamma) \cdot \Delta.\]
Applying the bounds $\|(W^\top(\gamma I-\bar{D})^{-1}W)^{-1}\| \leq
\gamma-\bar{d}_-$ and $\|W^\top(\gamma I-\bar{D})^{-1}\| \leq
\frac{1}{\gamma-\bar{d}_+}$, we may check that
the left side is again uniformly Lipschitz over $\gamma \geq \bar{d}_++\eps$
and, for any $\delta>0$, is bounded in magnitude
by $\delta$ when $\gamma$ lies outside a
compact subset $\calK_\delta' \subset [\bar{d}_++\eps,\infty)$.
Thus this convergence is again uniform over
$\gamma \geq \bar{d}_++\eps$. Then, combining with (\ref{eq:PiDsigmaparallel})
and applying the same argument as leading to (\ref{eq:schurcompfirstterm}),
we have for the second term of (\ref{eq:schurcomplementdecomp}) that
\begin{align*}
&\frac{1}{n}\sigma_\parallel^\top \bar{D}\Pi(\gamma I-\bar{D})^{-1} W
(W^\top(\gamma I-\bar{D})^{-1}W)^{-1}
W^\top(\gamma I-\bar{D})^{-1}\Pi\bar{D}\sigma_\parallel\\
&=\cF_{12}(\gamma)^\top \cF_{11}(\gamma)^{-1}\cF_{12}(\gamma)
\cdot \|v(\sigma)-\kappa_*^{-1/2}w(\sigma)\|^2+r_n(\sigma,\gamma)
\end{align*}
where $r_n(\sigma,\gamma) \to 0$ uniformly over $\gamma \geq \bar{d}_++\eps$ and
$\sigma \in U_n$.
Defining $\cF=\cF_{22}-\cF_{12}^\top \cF_{11}^{-1}\cF_{12}$
as in (\ref{eq:Fgamma}), this shows that almost surely as $n \to \infty$,
the second term of (\ref{eq:Ensigma}) satisfies
\begin{equation}\label{eq:Fconvergencefixed}
\frac{1}{n}\sigma_\parallel^\top \bar{D}V(\gamma I-V^\top \bar{D}V)^{-1}
V^\top \bar{D}\sigma_\parallel=
\cF(\gamma)\cdot \|v(\sigma)-\kappa_*^{-1/2}w(\sigma)\|^2
+r_n(\sigma,\gamma).
\end{equation}

Observe that this also implies
\begin{equation}\label{eq:Fmonotonicity}
\cF(\gamma) \text{ is non-increasing and convex over } \gamma>\bar{d}_+.
\end{equation}
Indeed, fixing any $\gamma>\bar{d}_+$, let us take $\eps$ above small enough
such that $\gamma \geq \bar{d}_++\eps$. For each $n$,
let us take $\sigma \in U_n$ such that
$\|v(\sigma)\|^2 \to 1$ and $\|w(\sigma)\|^2 \to 0$ as $n \to \infty$. (For
example, we may choose $\sigma=\sqrt{n}(x+\delta_n r)/\|x+\delta_n r\|$
where $x$ is the first column of $X$, $r$ is a unit vector orthogonal to the
column span of $(X,Y)$, and $\delta_n \to 0$ as $n \to \infty$.) Then as
$n \to \infty$, the right side of (\ref{eq:Fconvergencefixed}) converges to
$\cF(\gamma)$. The left side is non-increasing and convex at $\gamma$ for each
finite $n$, so the same properties hold for the limit $\cF(\gamma)$.

Combining (\ref{eq:Ifirstterm}), (\ref{eq:Hconvergence}),
and (\ref{eq:Fconvergencefixed}) and applying this to (\ref{eq:Oexpectation}),
we obtain the approximation for $\sigma \in U_n$
\begin{align*}
f_n(\sigma)&=\inf_{\gamma \geq \bar{d}_++\eps} \Bigg(
\frac{2a_*}{\kappa_*^{1/2}}v(\sigma)^\top w(\sigma)+\left(\lambda_*-\frac{a_*}{\kappa_*}\right)
\|w(\sigma)\|^2+\cF(\gamma) \cdot \|v(\sigma)-\kappa_*^{-1/2}w(\sigma)\|^2\\
&\hspace{3in}+\cH(\gamma,1-\|v(\sigma)\|^2-\|w(\sigma)\|^2)\Bigg)
+r_n(\sigma),
\end{align*}
where $r_n(\sigma) \to 0$ uniformly over $\sigma \in U_n$.
Observe that for any fixed $\sigma \in U_n$, we have
$\|v(\sigma)\|^2+\|w(\sigma)\|^2<1$ strictly, so
the argument to this infimum is a well-defined and convex function of
$\gamma \in (\bar{d}_+,\infty)$. Its derivative in $\gamma$ is
\[\cF'(\gamma) \cdot \|v(\sigma)-\kappa_*^{-1/2}w(\sigma)\|^2
+1-\|v(\sigma)\|^2-\|w(\sigma)\|^2-\bar{G}(\gamma).\]
For any $\gamma \in (\bar{d}_+,\bar{d}_++\eps]$,
$\cF'(\gamma)\leq 0$ as shown in (\ref{eq:Fmonotonicity}),
and $1<\bar{G}(\bar{d}_++\eps)-\eps$ as
previously argued in \prettyref{eq:Glarge}, so $\bar{G}(\gamma)>1+\eps$.
Thus this derivative is negative for $\gamma \in
(\bar{d}_+,\bar{d}_++\eps]$,
so it is equivalent to write this infimum over the range
$\gamma>\bar{d}_+$, i.e.\
\begin{align}
f_n(\sigma)&=f(v(\sigma),w(\sigma))+r_n(\sigma)\label{eq:fnapprox}
\end{align}
where the function $f$ on the domain $\calV \triangleq \{(v,w):\|v\|^2+\|w\|^2<1\}$ is defined by
\begin{equation}\label{eq:fvw}
f(v,w) \triangleq \inf_{\gamma>\bar{d}_+}
\frac{2a_*}{\kappa_*^{1/2}}v^\top w+\left(\lambda_*-\frac{a_*}{\kappa_*}\right)
\|w\|^2+\cF(\gamma) \cdot \|v-\kappa_*^{-1/2}w\|^2
+\cH(\gamma,1-\|v\|^2-\|w\|^2).
\end{equation}

Finally, observe that $f_n(\sigma)$ is continuous on the sphere $\{\sigma \in
\R^n:\|\sigma\|^2=n\}$, and the function
$\sigma \mapsto (v(\sigma),w(\sigma))$ is continuous, relatively open, and
maps the dense subset $U_n$ of this sphere
to $\calV$ for
every $n$. By Proposition \ref{prop:continuousextension} in \prettyref{app:aux},
$f(v,w)$ admits a continuous extension\footnote{Here, it is not hard to show that this extension to $\|v\|^2+\|w\|^2 = 1$ is given explicitly by 
$f(v,w)=\frac{2a_*}{\kappa_*^{1/2}}v^\top
w+(\lambda_*-\frac{a_*}{\kappa_*})\|w\|^2$, but this explicit form is not
needed for proving the end result in \prettyref{eq:Psi1}.}
 to the closure $\bar{\calV}=\{(v,w):\|v\|^2+\|w\|^2 \leq 1\}$, and (\ref{eq:fnapprox}) holds
uniformly over all $\sigma$ on this sphere. Thus, we have shown the almost sure
uniform convergence
\begin{equation}\label{eq:fnapprox2}
\lim_{n \to \infty} \sup_{\sigma \in \R^n:\|\sigma\|^2=n}
\left|f_n(\sigma)-f(v(\sigma),w(\sigma)) \right|=0.
\end{equation}\\

\noindent {\bf Large deviations analysis.}
We conclude the proof by applying Varadhan's Lemma and the
G\"{a}rtner-Ellis Theorem: Consider now the discrete uniform law
$\sigma \sim \Unif(\{+1,-1\}^n)$ and write $\langle \cdot \rangle$ for the
expectation over this law. For arguments $U \in \R$ and $V,W \in \R^t$,
define the limiting cumulant generating function
\begin{align*}
\lambda(U,V,W)&=\lim_{n \to \infty} \frac{1}{n}
\log \Big\langle \exp\Big[n\big(U \cdot u(\sigma)+V^\top v(\sigma)
+W^\top w(\sigma)\big)\Big]\Big\rangle\\
&=\lim_{n \to \infty} \frac{1}{n}
\log \left\langle \exp\left[U \cdot h^\top \sigma+
(V^\top,W^\top)\left[\frac{1}{n}\begin{pmatrix} X^\top X & X^\top Y \\
Y^\top X & Y^\top Y \end{pmatrix}\right]^{-1/2}
\begin{pmatrix} X^\top \sigma \\ Y^\top \sigma \end{pmatrix}
\right]\right\rangle\\
&=\lim_{n \to \infty} \frac{1}{n}\log \Big\langle \exp\Big[U \cdot h^\top \sigma
+V^\top \Delta^{-1/2}X^\top \sigma+W^\top (\kappa_*\Delta)^{-1/2}
Y^\top \sigma+n \cdot r_n(\sigma)\Big]\Big\rangle.
\end{align*}
Here $r_n(\sigma)$ is a remainder term satisfying $r_n(\sigma) \to 0$
uniformly over $\sigma \in \{+1,-1\}^n$ for any
fixed arguments $U,V,W$, and hence is negligible in the
large-$n$ limit. Evaluating the average over $\sigma$ using $\langle
e^{a\sigma_i} \rangle=\cosh a$,
and writing $h_i \in \R$ and $x_i,y_i \in \R^t$ for the entries of $h$
and rows of $X,Y$, we obtain
\begin{align*}
\lambda(U,V,W)&=\lim_{n \to \infty} \frac{1}{n}\sum_{i=1}^n \log \cosh
\Big(U \cdot h_i+V^\top \Delta^{-1/2}x_i+\kappa_*^{-1/2} W^\top \Delta^{-1/2}
y_i\Big)
\end{align*}
Then the weak convergence in law (\ref{eq:AMPconvergence})
from the AMP state evolution of Theorem \ref{thm:SE} shows that this limit
indeed exists almost surely, and is given by
\[\lambda(U,V,W)=\E\Big[\log \cosh\Big(U \cdot \H+V^\top
\Delta^{-1/2}(\X_1,\ldots,\X_t)+\kappa_*^{-1/2}W^\top \Delta^{-1/2}
(\Y_1,\ldots,\Y_t)\Big)\Big].\]
Note that the function $\lambda(U,V,W)$ is finite and differentiable at all 
$(U,V,W) \in \R^{2t+1}$. Then, denoting by
\begin{equation}\label{eq:Fencheldual}
\lambda^*(u,v,w)=\sup_{U \in \R,\;V,W \in \R^t}
U \cdot u+V^\top v+W^\top w-\lambda(U,V,W)
\end{equation}
its Fenchel-Legendre dual, the G\"{a}rtner-Ellis Theorem shows that
$(u(\sigma),v(\sigma),w(\sigma))$ satisfies a large deviations principle
with good rate function $\lambda^*(u,v,w)$ \cite[Theorem 2.3.6]{DZ98}.

The function $(u,v,w) \mapsto u+f(v,w)/2$ is continuous over
$\{u \in \R,\;v,w \in \R^t:\;\|v\|^2+\|w\|^2 \leq 1\}$.
Here $f(v,w)$ must be bounded over the compact
set $\{v,w \in \R^t:\|v\|^2+\|w\|^2 \leq 1\}$, and for any $c>0$ we have the
exponential integrability
\[\lim_{n \to \infty} \frac{1}{n}\log \Big\langle e^{cnu(\sigma)}\Big\rangle
=\lim_{n \to \infty}
\frac{1}{n}\log \Big\langle e^{c \cdot h^\top \sigma} \Big\rangle
=\lim_{n \to \infty} \frac{1}{n}\sum_{i=1}^n \log\cosh(ch_i)
=\E[\log \cosh(c\H)]<\infty.\]
Then by (\ref{eq:fnsigma}), (\ref{eq:fnapprox2}), and
Varadhan's lemma \cite[Theorem 4.3.1]{DZ98},
\begin{align*}
\lim_{n \to \infty} \frac{1}{n}\log \E[Z \mid \cG]
&=\log 2+\lim_{n \to \infty} \frac{1}{n}\log \left\langle
\exp\Big(n \cdot \Big[u(\sigma)+\frac{1}{2}f(v(\sigma),w(\sigma))\Big]\Big)
\right\rangle\\
&=\mathop{\sup_{u \in \R}}_{v,w \in \R^t:\;\|v\|^2+\|w\|^2 \leq 1}
\log 2+u+\frac{f(v,w)}{2}-\lambda^*(u,v,w).
\end{align*}
The domain $\|v\|^2+\|w\|^2 \leq 1$ in this supremum may now be restricted
to $\|v\|^2+\|w\|^2<1$, by continuity of $f(v,w)$ and lower-semicontinuity
of the rate function $\lambda^*(u,v,w)$.
Substituting the forms of $f$ and $\lambda^*$ from (\ref{eq:fvw})
and (\ref{eq:Fencheldual}) concludes the proof.
\end{proof}

\subsection{Analysis of the variational formula}

Denote by $\partial_u \Phi_{1,t} \in \R$, $\partial_v \Phi_{1,t} \in \R^t$,
etc.\ the partial derivatives of the function $\Phi_{1,t}$ in each argument.
We now consider an approximate stationary point of (\ref{eq:Psi1}), given by
\[u_*=\E[\H \cdot \tanh(\H+\sigma_*\G)], \qquad v_*=(1-q_*)\Delta_t^{1/2}e_t,
\qquad w_*=\kappa_*^{1/2}(1-q_*)\Delta_t^{1/2}e_t\]
\[\gamma_*=\bar{G}^{-1}(1-q_*)
=\bar{R}(1-q_*)+(1-q_*)^{-1}, \qquad U_*=1, \qquad V_*=0, \qquad
W_*=\kappa_*^{1/2}\Delta_t^{1/2}e_t\]
where $e_t=(0,\ldots,0,1)$
is the $t^\text{th}$ standard basis vector in $\R^t$. We check in
two steps that first, this is approximately stationary for the optimization
in (\ref{eq:Psi1}) and yields the desired value $\Psi_{\RS}$ in \prettyref{eq:PsiRS}, and second,
that it is approximately the global solution to (\ref{eq:Psi1})
when $\beta>0$ is sufficiently small.

For these steps, we require the following properties of
$\cF(\gamma)$ defined in (\ref{eq:Fgamma}).
\begin{lemma}\label{lmm:Fderivatives}	
\begin{enumerate}[(a)]
\item $\cF(\gamma)$ is monotonically decreasing and convex over
$\gamma>\bar{d}_+$.
\item Fix any $\delta>0$, open neighborhood $U \subset \R$, and twice
differentiable function $\gamma:U \to (\bar{d}_++\delta,\infty)$.
Then for some constants $C,\beta_0>0$ depending only on $\mu_D$ and $\delta$,
any $s \in U$, and all $\beta \in (0,\beta_0)$,
\begin{align*}
|\calF(\gamma(s))| &\leq C\beta^4(1-q_*)^2 \cdot
\sup_{x \in \supp(\mu_{\bar{D}})} |(\gamma(s)-x)^{-1}| \\
|\partial_s \calF(\gamma(s))| &\leq C\beta^4(1-q_*)^2 \cdot
\sup_{x \in \supp(\mu_{\bar{D}})} |\partial_s (\gamma(s)-x)^{-1}| \\
\partial_s^2 \calF(\gamma(s)) &\leq C\beta^4(1-q_*)^2 \cdot 
\sup_{x \in \supp(\mu_{\bar{D}})} |\partial_s^2 (\gamma(s)-x)^{-1}|.
\end{align*}
\end{enumerate}
\end{lemma}
\begin{proof}
Part (a) was verified in (\ref{eq:Fmonotonicity}).

For part (b),
we use the notation $O(f(\beta))$ as in Proposition \ref{prop:smallbeta}, and
allow the constant in this notation to depend also on $\delta$ throughout the
proof. We have $x=O(\beta)$ uniformly over $x \in
\supp(\mu_{\bar{D}})$. Applying this and Proposition \ref{prop:smallbeta},
\[(1-q_*)(\lambda_*-x)=1-(1-q_*)x+O(\beta^2(1-q_*)^2), \qquad
\bar{R}(1-q_*)\kappa_*^{-1}=\frac{1}{1-q_*}+O(\beta).\]
Then for $\lambda(x)$ and $\theta(x)$ defined in \prettyref{eq:LambdaThetax},
uniformly over $x \in \supp(\mu_{\bar{D}})$, we have
\[\lambda(x) = O(\beta(1-q_*)), \quad \theta(x) = O(\beta^2(1-q_*)).\]
Abbreviate $\lambda \equiv \beta(1-q_*)$ and $\theta\equiv\beta^2(1-q_*)$.
Then differentiating $\cF_{11},\cF_{12},\cF_{22}$ in $\gamma$,
this implies for $k=0,1,2$,
\begin{align}
\partial_s^k \calF_{11}(\gamma(s)) &= O\pth{\begin{pmatrix} 1 &
\lambda \\ \lambda & \lambda^2 \end{pmatrix}}
\cdot \sup_{x \in \supp(\mu_{\bar{D}})} |\partial_s^k (\gamma(s)-x)^{-1}|
\label{eq:Fall11}\\
\partial_s^k \calF_{12}(\gamma(s))&=O\pth{\begin{pmatrix} \theta  \\ \lambda
\theta \end{pmatrix}} \cdot \sup_{x \in \supp(\mu_{\bar{D}})}
|\partial_s^k (\gamma(s)-x)^{-1}|\label{eq:Fall12}\\
\partial_s^k \calF_{22}(\gamma(s)) &= O(\theta^2) \cdot
\sup_{x \in \supp(\mu_{\bar{D}})} |\partial_s^k (\gamma(s)-x)^{-1}|.
\label{eq:Fall22}
\end{align}
Here and below, $O(\cdot)$ for a matrix or vector is in the sense of
entrywise comparison.

To bound $\calF_{11}(\gamma)^{-1}$ appearing in $\cF(\gamma)$, we first bound
$\det \calF_{11}(\gamma)$ from below
in terms of the variance of $\mu_D$ as follows:
Let $D_1,D_2$ be independently drawn from $\mu_D$. Let
$X_i=\frac{1}{\gamma-\beta D_i}$ and $Y_i=\frac{1}{(1-q_*)(\lambda_*-\beta
D_i)}-1$ for $i=1,2$, where $X_i$ is positive for $\gamma>\bar{d}_+$.
Let $\bar{d}_-=\beta d_-$ denote the minimum point of support of
$\mu_{\bar{D}}$. Then
	\begin{align*}
	\det \calF_{11}(\gamma)
	= \Expect[X_1]\Expect[X_1Y_1^2]- (\Expect[X_1Y_1])^2
	= & ~ \frac{1}{2} \Expect[X_1X_2(Y_1-Y_2)^2]\\
	\geq & ~
\frac{\beta^2}{2(1-q_*)^2(\gamma-\bar{d}_-)^2(\lambda_*-\bar{d}_-)^2} \Expect[(D_1-D_2)^2]\\
	= & ~
\frac{\beta^2}{(1-q_*)^2(\gamma-\bar{d}_-)^2(\lambda_*-\bar{d}_-)^2} \Var(D_1).
	\end{align*}
Using the small-$\beta$ expansion of $\lambda_*$ in \prettyref{prop:smallbeta},
we have $(1-q_*)(\lambda_*-\bar{d}_-)=1+O(\beta)$. Then for any $\gamma \geq
\bar{d}_++\delta$ and some constant $c=c(\mu_D,\delta)>0$,
we have $\det \cF_{11}(\gamma) \geq c\beta^2/\gamma^2$.
Applying the explicit $2 \times 2$ matrix inverse of $\cF_{11}$, and
combining this with \prettyref{eq:Fall11} for $k=0$ and the bound
$|(\gamma(s)-x)^{-1}|=O(1/\gamma(s))$,
\begin{equation}\label{eq:F11invbound}
\calF_{11}(\gamma(s))^{-1} = O\pth{\beta^{-2}\gamma(s)
\begin{pmatrix}\lambda^2 & \lambda \\ \lambda & 1 \end{pmatrix}}.
\end{equation}
Then applying \prettyref{eq:F11invbound},
\prettyref{eq:Fall12} and \prettyref{eq:Fall22} for $k=0$,
and $|(\gamma(s)-x)^{-1}|=O(1/\gamma(s))$ again,
\begin{align*}
\cF(\gamma(s))&=\cF_{22}-\cF_{12}^\top \cF_{11}^{-1}\cF_{12}\\
&=O(\theta^2) \cdot \sup_{x \in \supp(\mu_{\bar{D}})} |(\gamma(s)-x)^{-1}|
+O(\theta^2 \lambda^2/\beta^2) \cdot \gamma(s)
\left(\sup_{x \in \supp(\mu_{\bar{D}})} |(\gamma(s)-x)^{-1}|\right)^2\\
&=O(\beta^4(1-q_*)^2) \cdot\sup_{x \in \supp(\mu_{\bar{D}})}
|(\gamma(s)-x)^{-1}|.
\end{align*}
Here, the second equality uses $\theta^2=\beta^4(1-q_*)^2$ and
$\lambda^2=\beta^2(1-q_*)^2=O(\beta^2)$.
This is the desired bound for $|\cF(\gamma(s))|$.

For the derivative, let us write as shorthand
$\cF_{11}'=\partial_s \cF_{11}(\gamma(s))$ and similarly for the
other terms. Now differentiating $\cF_{11}^{-1}$ and applying 
(\ref{eq:F11invbound}) and (\ref{eq:Fall11}) with $k=1$,
\[[\cF_{11}^{-1}]'=-\cF_{11}^{-1}\cF_{11}'\cF_{11}^{-1}
=O\pth{\beta^{-4}\begin{pmatrix}\lambda^4 & \lambda^3 \\ \lambda^3 & \lambda^2 
\end{pmatrix}} \cdot \gamma(s)^2
\sup_{x \in \supp(\mu_{\bar{D}})} |\partial_s (\gamma(s)-x)^{-1}|.\]
Then applying also (\ref{eq:Fall12}) and (\ref{eq:Fall22}) with $k \in \{0,1\}$
and $|(\gamma(s)-x)^{-1}| \leq O(1/\gamma(s))$,
\begin{align*}
\partial_s \cF(\gamma(s))&=\cF_{22}'-2{\cF_{12}'}^\top
\cF_{11}^{-1}\cF_{12}-\cF_{12}^\top[\cF_{11}^{-1}]'\cF_{12}\\
&=O\left(\theta^2\left(1+\frac{\lambda^2}{\beta^2}+\frac{\lambda^4}{\beta^4}
\right)\right)
\cdot \sup_{x \in \supp(\mu_{\bar{D}})} |\partial_s (\gamma(s)-x)^{-1}|.
\end{align*}
The desired bound for $|\partial_s \cF(\gamma(s))|$ follows again from
$\theta^2=\beta^4(1-q_*)^2$ and $\lambda^2=O(\beta^2)$.

Finally, differentiating $\cF_{11}^{-1}$ again and applying 
(\ref{eq:F11invbound}) and (\ref{eq:Fall11}) with $k=2$,
\begin{align*}
[\cF_{11}^{-1}]''&=-\cF_{11}^{-1}\cF_{11}''\cF_{11}^{-1}
+2\cF_{11}^{-1}\cF_{11}'\cF_{11}^{-1}\cF_{11}'\cF_{11}^{-1}\\
&=O\pth{\beta^{-4}\begin{pmatrix}\lambda^4 & \lambda^3 \\ \lambda^3 & \lambda^2 
\end{pmatrix}} \cdot \gamma(s)^2
\sup_{x \in \supp(\mu_{\bar{D}})} |\partial_s^2 (\gamma(s)-x)^{-1}|
+2\cF_{11}^{-1}\cF_{11}'\cF_{11}^{-1}\cF_{11}'\cF_{11}^{-1}.
\end{align*}
Then applying (\ref{eq:F11invbound}),
(\ref{eq:Fall12}) and (\ref{eq:Fall22}) with $k \in \{0,2\}$,
\begin{align*}
\partial_s^2 \cF(\gamma(s))&=\cF_{22}''-2{\cF_{12}''}^\top \cF_{11}^{-1}\cF_{12}
-\cF_{12}^\top[\cF_{11}^{-1}]''\cF_{12}
-2{\cF_{12}'}^\top \cF_{11}^{-1}{\cF_{12}'}
-4{\cF_{12}'}^\top [\cF_{11}^{-1}]'\cF_{12} \numberthis \label{eq:Fpp1}\\
&=O\left(\theta^2\left(1+\frac{\lambda^2}{\beta^2}+\frac{\lambda^4}{\beta^4}\right)\right)
\cdot \sup_{x \in \supp(\mu_{\bar{D}})} |\partial_s^2 (\gamma(s)-x)^{-1}|\\
&\hspace{1in}-2\Big(\cF_{12}^\top\cF_{11}^{-1}\cF_{11}'\cF_{11}^{-1}\cF_{11}'\cF_{11}^{-1}\cF_{12}+{\cF_{12}'}^\top \cF_{11}^{-1}{\cF_{12}'}
+2{\cF_{12}'}^\top [\cF_{11}^{-1}]'\cF_{12}\Big).
\end{align*}
Note that for the second term above,
\begin{align*}
&\cF_{12}^\top\cF_{11}^{-1}\cF_{11}'\cF_{11}^{-1}\cF_{11}'\cF_{11}^{-1}\cF_{12}
+{\cF_{12}'}^\top \cF_{11}^{-1}{\cF_{12}'}
+2{\cF_{12}'}^\top [\cF_{11}^{-1}]'\cF_{12}\\
&\hspace{1in}=\Big[{\cF_{12}'}-\cF_{11}'\cF_{11}^{-1}\cF_{12}\Big]^\top
\cF_{11}^{-1}\Big[{\cF_{12}'}-\cF_{11}'\cF_{11}^{-1}\cF_{12}\Big] \geq 0
\end{align*}
where this inequality holds because $\cF_{11}^{-1} \succ 0$.
Then, applying again $\theta^2=\beta^4(1-q_*)^2$ and $\lambda^2=O(\beta^2)$,
we obtain the desired upper bound for $\partial_s^2 \cF(\gamma(s))$.
\end{proof}

\begin{lemma}\label{lemma:Psi1stationary}
For all $t \geq 1$ and each $\iota \in \{u,v,w,\gamma,U,W\}$,
\[\Phi_{1,t}(u_*,v_*,w_*;\gamma_*,U_*,V_*,W_*)=\Psi_{\RS},
\qquad \partial_\iota
\Phi_{1,t}(u_*,v_*,w_*;\gamma_*,U_*,V_*,W_*)=0.\]
Furthermore, for $\iota=V$,
\[\lim_{t \to \infty} \|\partial_V
\Phi_{1,t}(u_*,v_*,w_*;\gamma_*,U_*,V_*,W_*)\|=0.\]
\end{lemma}
\begin{proof}
Let $\delta_{tt}=e_t^\top\Delta_te_t$
be the lower-right entry of $\Delta_t$, and recall that
$\delta_{tt}=\delta_*$ by Proposition \ref{prop:AMPconvergent}.
Then the function $\Phi_{1,t}$ in (\ref{eq:Phi1})
evaluated at $u_*,v_*,w_*,\gamma_*,U_*,V_*,W_*$ is
\begin{align*}
\Phi_{1,t}&=\E[\log 2 \cosh(\H+\Y_t)]
-\kappa_*(1-q_*)\delta_*+\bar{R}(1-q_*) \cdot (1-q_*)^2\delta_*
+\frac{\lambda_*\kappa_*-\bar{R}(1-q_*)}{2}(1-q_*)^2\delta_*\\
&\hspace{1in}+\frac{1}{2}\cH\Big(\bar{G}^{-1}(1-q_*),
1-\|v_*\|^2-\|w_*\|^2\Big).
\end{align*}
For the first term, by Theorem \ref{thm:SE} and
Proposition \ref{prop:AMPconvergent},
$\Y_t \sim \N(0,\kappa_*\delta_{tt})$ where $\kappa_*\delta_{tt}=\sigma_*^2$, so
\[\E[\log 2\cosh(\H+\Y_t)]=\E[\log 2\cosh(\H+\sigma_*\G)], \qquad \G \sim
\N(0,1).\]
For the second term, applying $\delta_*\kappa_*=\sigma_*^2=q_*\bar{R}'(1-q_*)$
by (\ref{eq:kappadelta}) and (\ref{eq:qstar}),
\[-\kappa_*(1-q_*)\delta_*=-q_*(1-q_*)\bar{R}'(1-q_*).\]
For the third and fourth terms, applying also
$\lambda_*=\bar{R}(1-q_*)+(1-q_*)^{-1}$ from (\ref{eq:lambdastar})
and $\delta_*=q_*(1-q_*)^{-2}-\sigma_*^2$ from (\ref{eq:deltastar}),
\begin{align*}
&\bar{R}(1-q_*) \cdot (1-q_*)^2\delta_*
+\frac{\lambda_*\kappa_*-\bar{R}(1-q_*)}{2}(1-q_*)^2\delta_*,\\
&=\frac{\bar{R}(1-q_*) \cdot (1-q_*)^2}{2}\left(\frac{q_*}{(1-q_*)^2}-
\sigma_*^2\right)+\left(\bar{R}(1-q_*)+\frac{1}{1-q_*}\right)
\frac{(1-q_*)^2}{2}\sigma_*^2,\\
&=\frac{q_*}{2}\bar{R}(1-q_*)+\frac{q_*(1-q_*)}{2}\bar{R}'(1-q_*).
\end{align*}
For the last term, observe that
$1-\|v_*\|^2-\|w_*\|^2=1-(1+\kappa_*)(1-q_*)^2\delta_*$.
Applying $\delta_*=\sigma_*^2/\kappa_*=q_*\bar{R}'(1-q_*)/\kappa_*$
and the form of $\kappa_*$ in (\ref{eq:kappastar}), this is
\begin{align*}
1-\|v_*\|^2-\|w_*\|^2&=1-(1+\kappa_*^{-1})(1-q_*)^2q_*\bar{R}'(1-q_*)=1-q_*.
\end{align*}
For some $\beta_0=\beta_0(\mu_D)>0$ and all $\beta \in (0,\beta_0)$,
we have $1-q_* \leq 1<\bar{G}(\bar{d}_+)=\beta^{-1}G(d_+)$.
Then by Proposition \ref{prop:infgamma}(a),
\[\cH\Big(\bar{G}^{-1}(1-q_*),1-q_*\Big)
=\int_0^{1-q_*} \bar{R}(z)dz.\]
Combining all of the above yields
\[\Phi_{1,t}=\E[\log 2\cosh(\H+\sigma_*\G)]
+\frac{q_*}{2}\bar{R}(1-q_*)-\frac{q_*(1-q_*)}{2}\bar{R}'(1-q_*)
+\frac{1}{2}\int_0^{1-q_*}\bar{R}(z)dz=\Psi_{\RS}.\]

To check the stationary conditions, first by the form of $\cH$ in (\ref{eq:Hgamma}), we have 
$\partial_\gamma \cH(\gamma,\alpha)=\alpha-\bar G(\gamma)$ and 
$\partial_\alpha \cH(\gamma,\alpha)=\gamma-1/\alpha$. Since 
$\gamma_*=\bar{G}^{-1}(1-q_*)=\bar{R}(1-q_*)+(1-q_*)^{-1}$ and
$\|v_*\|^2+\|w_*\|^2=q_*$, we have
\begin{equation}
\partial_\gamma \cH(\gamma_*,1-q_*)=0,
\qquad \partial_\alpha \cH(\gamma_*,1-q_*)=\gamma_*-\frac{1}{1-q_*}
=\bar{R}(1-q_*).
\label{eq:Hder}
\end{equation}
Then
evaluating at $u_*,v_*,w_*,\gamma_*,U_*,V_*,W_*$ where $v_*=\kappa_*^{-1/2}w_*$,
$V_*=0$, and $W_*=(1-q_*)^{-1}w_*$,
\begin{align}
\partial_u \Phi_{1,t}&=-U_*+1=0,\\
\partial_v \Phi_{1,t}&=-V_*+\bar{R}(1-q_*)\kappa_*^{-1/2}w_*
-\bar{R}(1-q_*)v_*=0,\label{eq:derv}\\
\partial_w \Phi_{1,t}&=-W_*+\bar{R}(1-q_*)\kappa_*^{-1/2}v_*
+(\lambda_*-\bar{R}(1-q_*)\kappa_*^{-1})w_*
-\bar{R}(1-q_*)w_*=0,\label{eq:derw}\\
\partial_\gamma \Phi_{1,t}&=\frac{1}{2}\partial_\gamma \cH(\gamma_*,1-q_*)=0. \label{eq:dergamma}
\end{align}
The third line above applies again
$\lambda_*=\bar{R}(1-q_*)+(1-q_*)^{-1}$.

For the derivatives in $U,V,W$,
observe that the derivative of $\log 2 \cosh x$ is $\tanh x$, and
\[\tanh(\H+\Y_t)=(1-q_*)(\X_{t+1}+\Y_t)\]
by the definition of the AMP state evolution (\ref{eq:XYrelation}). Hence
\begin{align}
\partial_U \Phi_{1,t}&=\E[\H \cdot (1-q_*)(\X_{t+1}+\Y_t)]-u_*, \label{eq:derU}\\
\partial_V \Phi_{1,t}&=\E[\Delta_t^{-1/2}(\X_1,\ldots,\X_t)
\cdot (1-q_*)(\X_{t+1}+\Y_t)]-v_*, \label{eq:derV}\\
\partial_W \Phi_{1,t}&=\E[\kappa_*^{-1/2}\Delta_t^{-1/2}(\Y_1,\ldots,\Y_t)
\cdot (1-q_*)(\X_{t+1}+\Y_t)]-w_*. \label{eq:derW}
\end{align}
From the joint law of $\H,\Y_1,\ldots,\Y_t,\X_1,\ldots,\X_{t+1}$ described in
Theorem \ref{thm:SE}, we have $\E[\H \cdot \Y_t]=0$ and
$\E[(1-q_*)\H \cdot \X_{t+1}]=\E[\H \cdot \tanh(\H+\sigma_*\G)]$,
so $\partial_U \Phi_{1,t}=0$. We have
$\E[(\Y_1,\ldots,\Y_t) \cdot \Y_t]=\kappa_*\Delta_te_t$ and 
$\E[(\Y_1,\ldots,\Y_t) \cdot \X_{t+1}]=0$,
so also $\partial_W \Phi_{1,t}=0$. Finally, writing a block decomposition of
$\Delta_{t+1}$ as
\begin{equation}
\Delta_{t+1}=\begin{pmatrix} \Delta_t & \delta_t \\ \delta_t^\top &
\delta_* \end{pmatrix}, \qquad
\delta_t=(\delta_{1,t+1},\ldots,\delta_{t,t+1}),
\label{eq:deltat}
\end{equation}
we have $\E[(\X_1,\ldots,\X_t) \cdot \X_{t+1}]=\delta_t$ and
$\E[(\X_1,\ldots,\X_t) \cdot \Y_t]=0$. Thus
\begin{equation}
\partial_V \Phi_{1,t}=(1-q_*)\Delta_t^{-1/2}\delta_t-v_*
=(1-q_*)\Big[\Delta_t^{-1/2}\delta_t-\Delta_t^{1/2}e_t\Big],
\label{eq:derV2}
\end{equation}
so (recalling that $\delta_{t+1,t+1}=\delta_{tt}=\delta_*$)
\[(1-q_*)^{-2}\|\partial_V \Phi_{1,t}\|^2
=\delta_t^\top \Delta_t^{-1}\delta_t-2e_t^\top\delta_t
+\delta_*=\Big(\delta_t^\top \Delta_t^{-1}\delta_t-\delta_{t+1,t+1}\Big)
-2(\delta_{t,t+1}-\delta_*).\]
By Proposition \ref{prop:AMPconvergent}, $\lim_{t \to \infty}
\delta_{t,t+1}=\delta_*$. By (\ref{eq:XX}) applied at $t+1$,
\[\delta_{t+1,t+1}-\delta_t^\top \Delta_t^{-1}\delta_t
=\inf_{\alpha \in \R^t} \E\Big[\Big(\X_{t+1}-\alpha^\top
(\X_1,\ldots,\X_t)\Big)^2\Big],\]
where the infimum is attained at the least-squares coefficients
\[\alpha=\E\Big[(\X_1,\ldots,\X_t)(\X_1,\ldots,\X_t)^\top\Big]^{-1}
\E\Big[(\X_1,\ldots,\X_t) \cdot \X_{t+1}\Big]=\Delta_t^{-1}\delta_t.\]
Then
\[0 \leq \delta_{t+1,t+1}-\delta_t^\top \Delta_t^{-1}\delta_t
\leq \E[(\X_{t+1}-\X_t)^2]=2\delta_*-2\delta_{t,t+1},\]
so also $\lim_{t \to \infty} \delta_{t+1,t+1}-\delta_t^\top
\Delta_t^{-1}\delta_t=0$. Thus $\lim_{t \to \infty} \|\partial_V
\Phi_{1,t}\|=0$.
\end{proof}

\begin{lemma}\label{lemma:Psi1optimal}
For a constant $\beta_0=\beta_0(\mu_D)>0$ and any $\beta \in (0,\beta_0)$,
\[\lim_{t \to \infty} \Psi_{1,t}=\Psi_{\RS}.\]
\end{lemma}
\begin{proof}
We will establish separately
\begin{align}
\liminf_{t \to \infty} \Psi_{1,t} &\geq \Psi_{\RS},\label{eq:Psi1lower}\\
\limsup_{t \to \infty} \Psi_{1,t} &\leq \Psi_{\RS}.\label{eq:Psi1upper}
\end{align}
We write $o_t(1)$ for any scalar, vector, or matrix error (with dimension
depending on $t$) that satisfies $\lim_{t \to \infty} \|o_t(1)\|=0$, 
where $\|\cdot\|$ is the Euclidean norm for vectors and operator norm for
matrices.
Note that $\Psi_{1,t}$ takes the max-min form $\Psi_{1,t}=\sup_{u,v,w} \inf_{\gamma,U,V,W}\Phi_{1,t}$. Here the supremum and the infimum cannot be interchanged due to the non-concavity in the $(u,v,w)$ parameter. Our strategy is as follows: 
\begin{itemize}
	\item For the lower bound, we specialize the outer supremum to a fixed choice of $(u,v,w)$ near $(u_*,v_*,w_*)$ and minimize the resulting (convex) function over $(\gamma,U,V,W)$. This minimizer is shown to be approximately $(\gamma_*,U_*,V_*,W_*)$.
	
	\item For the upper bound, we specialize the inner infimum to a choice of $(\gamma,U,V,W)$ depending on $(u,v,w)$ in such a way that the resulting function is globally concave for sufficiently small $\beta$. This concave function is then shown to be approximately maximized at $(u_*,v_*,w_*)$.

\end{itemize}

To show the lower bound (\ref{eq:Psi1lower}), we specialize $\Phi_{1,t}$
to $(u,v,w)=(u_*,\tilde{v}_*,w_*)$ where
\[\tilde{v}_*=v_*+(1-q_*)[\Delta_t^{-1/2}\delta_t-\Delta_t^{1/2}e_t]=v_*+o_t(1),\]
and $\delta_t$ was defined in \prettyref{eq:deltat}. Here 
the second equality has been verified in the preceding proof of
Lemma \ref{lemma:Psi1stationary}. As defined in \prettyref{eq:Phi1}, 
$\Phi_{1,t}(u_*,\tilde{v}_*,w_*;\gamma,U,V,W)$ decomposes as $X(U,V,W)+Y(\gamma)$, where $X$ and $Y$ are both convex functions; specifically,
$Y(\gamma) = \frac{1}{2}\cF(\gamma)
\|\tilde v_*-v_*\|^2 +\frac{1}{2}\cH(\gamma,1-\|\tilde v_*\|^2-\|w_*\|^2)$
which is convex applying Lemma \ref{lmm:Fderivatives}(a). Then
\begin{align*}
\Psi_{1,t} &\geq \inf_{\gamma>\bar d_+}
\inf_{U \in \R,\;V,W \in \R^t} \Phi_{1,t}(u_*,\tilde{v}_*,w_*;\gamma,U,V,W) = 
\inf_{U \in \R,\;V,W \in \R^t} X(U,V,W)+\inf_{\gamma>\bar d_+} Y(\gamma).
\end{align*}
In view of \prettyref{eq:derU}, \prettyref{eq:derW}, and \prettyref{eq:derV2}, note that $\tilde{v}_*$ is chosen so that $(U_*,V_*,W_*)$ is now an exact
stationary point of $X$. Hence by the convexity of $X$, $\inf_{U \in \R,\;V,W \in \R^t} X(U,V,W) = X(U_*,V_*,W_*)$.
For the infimum over $\gamma$, recall from \prettyref{eq:dergamma} that $\partial_\gamma \calH(\gamma_*,1-\|v_*\|^2-\|w_*\|^2)=1-\|v_*\|^2-\|w_*\|^2-\bar G(\gamma_*)=0$.
Since $\|\tilde v_*-v_*\|=o_t(1)$, we have that $Y'(\gamma_*)=\frac{1}{2}\cF'(\gamma_*)\|\tilde v_*-v_*\|^2 +\frac{1}{2}(\|v_*\|^2-\|\tilde v_*\|^2) = o_t(1)$.
Furthermore,
$Y''(\gamma) = \frac{1}{2}\cF''(\gamma)\|\tilde v_*-v_*\|^2 +\frac{1}{2} \bar G'(\gamma)$.
Thus there exist some constants $c,\delta>0$ independent of $t$, such that 
$Y''(\gamma) \geq c$ whenever $|\gamma-\gamma_*|<\delta$. Applying
\prettyref{prop:approxmin}, we conclude that 
$\inf_{\gamma>\bar d_+} Y(\gamma) \geq Y(\gamma_*)+o_t(1)$.
Combining these two bounds,
\begin{align*}
\Psi_{1,t} \geq 
X(U_*,V_*,W_*) + Y(\gamma_*)+o_t(1)
=\Phi_{1,t}(u_*,\tilde v_*,w_*;\gamma_*,U_*,V_*,W_*)+o_t(1)=\Psi_{\RS}+o_t(1).
\end{align*}
Here the last step follows from $\|\partial_v \Phi_{1,t}(u_*,v,w_*;\gamma_*,U_*,V_*,W_*)\| \leq C$
for all $\|v-v_*\|\leq \delta$, where $C,\delta$ are constants independent of $t$.
This shows the lower bound (\ref{eq:Psi1lower}).

For the upper bound (\ref{eq:Psi1upper}), we now specialize $\Phi_{1,t}$ to
\[\gamma=\gamma(v,w)=\bar{G}^{-1}(1-\|v\|^2-\|w\|^2),\]
\[U=U_*=1, \qquad V=V(v)=\beta^{1/2}(v-v_*), \qquad
W=W(w)=\beta^{1/2}(w-w_*)+W_*.\]
Here $\gamma(v,w)$ is well-defined for any $(v,w)$ such that $\|v\|^2+\|w\|^2<1$, since $\bar G(\bar d_+)>1$ in view of \prettyref{eq:Glarge}.
Note that at $(v,w)=(v_*,w_*)$, this
gives $(\gamma(v_*,w_*),U,V(v_*),W(w_*))=(\gamma_*,U_*,V_*,W_*)$. Furthermore,
\begin{align*}
\Psi_{1,t} &\leq \mathop{\sup_{u \in \R}}_{v,w \in \R^t:\|v\|^2+\|w\|^2<1}
\Phi_{1,t}\Big(u,v,w;\gamma(v,w),1,V(v),W(w)\Big).
\end{align*}
Due to the choice $U=1$, the function on the right no longer depends
on $u$. We denote it by
\[\tilde{\Phi}_{1,t}(v,w)=
\Phi_{1,t}\Big(u,v,w;\gamma(v,w),1,V(v),W(w)\Big)=\I+\II+\III+\IV\]
where
\begin{align*}
\I&=\E\Big[\log 2 \cosh \Big(\H+V(v)^\top\Delta_t^{-1/2}
(\X_1,\ldots,\X_t)+\kappa_*^{-1/2}W(w)^\top
\Delta_t^{-1/2}(\Y_1,\ldots,\Y_t)\Big)\Big]\\
\II&=-v^\top V(v)-w^\top W(w)+\bar{R}(1-q_*)\kappa_*^{-1/2}v^\top w
+\frac{\lambda_*-\bar{R}(1-q_*)\kappa_*^{-1}}{2}\|w\|^2\\
\III&=\frac{1}{2}\cF(\gamma(v,w))\|v-\kappa_*^{-1/2}w\|^2\\
\IV&=\frac{1}{2}\cH\Big(\gamma(v,w),1-\|v\|^2-\|w\|^2\Big).
\end{align*}

We claim that for some $\beta_0=\beta_0(\mu_D)>0$ and all $\beta \in
(0,\beta_0)$, this function $\tilde{\Phi}_{1,t}(v,w)$ is concave over the
domain $\{v,w \in \R^t:\|v\|^2+\|w\|^2<1\}$. To show this claim, we analyze the
Hessian of each term $\I,\II,\III,\IV$ using the small-$\beta$ approximations of
Proposition \ref{prop:smallbeta}---the desired concavity will arise from the
first two terms of $\II$. We write $O(\beta^k)$ for a scalar, vector,
or matrix whose (Euclidean or operator) norm is at most $C\beta^k$ uniformly
over $\{v,w \in \R^t:\|v\|^2+\|w\|^2<1\}$, for a constant
$C=C(\mu_D)>0$ depending only on $\mu_D$.

For $\I$, we have
\[\nabla^2_{v,w} \I=\beta \cdot \E\Big[Z_tZ_t^\top \cdot
\Big(1-\tanh^2\Big(\H+V(v)^\top\Delta_t^{-1/2}
(\X_1,\ldots,\X_t)+\kappa_*^{-1/2}W(w)^\top
\Delta_t^{-1/2}(\Y_1,\ldots,\Y_t)\Big)\Big]\]
where
\begin{equation}
Z_t \triangleq \Big(\Delta_t^{-1/2}(\X_1,\ldots,\X_t),
\kappa_*^{-1/2} \Delta_t^{-1/2}(\Y_1,\ldots,\Y_t)\Big) \in \R^{2t}.
\label{eq:Zt}
\end{equation}
Then $0 \preceq \nabla_{v,w}^2 \I \preceq \beta \E[Z_tZ_t^\top]=\beta
I_{2t \times 2t}$,
the last equality applying (\ref{eq:XX}--\ref{eq:XY}). 

For $\II$, observe that
Proposition \ref{prop:smallbeta} implies
\[\bar{R}(1-q_*)\kappa_*^{-1/2}=O(\beta), \qquad
\lambda_*-\bar{R}(1-q_*)\kappa_*^{-1}=O(\beta).\]
Then $\nabla_{v,w}^2 \II=-2\beta^{1/2}I_{2t \times 2t}+O(\beta)$.

For $\III$, consider any scalar linear parametrization
\[(v(s),w(s))_{s \in \R}=(v,w)+s \cdot (v',w')\]
where $\|(v',w')\|=1$. Write as shorthand
\[A(s) \triangleq 1-\|v(s)\|^2-\|w(s)\|^2,
\qquad B(s) \triangleq \|v(s)-\kappa_*^{-1/2}w(s)\|^2.\]
Applying $\|v\|,\|w\|,\|v'\|,\|w'\| \leq 1$, it is easily checked that
\begin{equation}\label{eq:derAs}
|A(s)|,|\partial_s A(s)|,|\partial_s^2 A(s)|=O(1) \text{ at } s=0.
\end{equation}
Applying also $\kappa_*^{-1}=O(\beta^{-2}(1-q_*)^{-2})$ by
Proposition \ref{prop:smallbeta}, we have
\begin{equation}\label{eq:derBs}
|B(s)|,|\partial_s B(s)|,|\partial_s^2 B(s)|=O(\beta^{-2}(1-q_*)^{-2})
\text{ at } s=0.
\end{equation}
Now write also as shorthand
\[\cF(s) \triangleq \cF(\gamma(v(s),w(s)))=\cF\Big(\bar{G}^{-1}(A(s))\Big).\]
Then
\begin{align}
&(v',w')^\top \cdot \nabla_{v,w}^2 \III \cdot (v',w')\nonumber\\
&=\partial_s^2 \III\Big|_{s=0}=\frac{1}{2}
\Big(\partial_s^2 \cF(s) \cdot B(s)+2\partial_s \cF(s) \cdot \partial_s B(s)
+\cF(s) \cdot \partial_s^2 B(s)\Big)\Big|_{s=0}.\label{eq:derIII}
\end{align}

Observe that since $\|v\|^2+\|w\|^2<1$, we have $A(s) \in (0,1]$ at
$s=0$. Let $\bar{d}_-=\beta d_-$ be the smallest point of support of
$\mu_{\bar{D}}$. For any $x>\bar{d}_+$, since $\bar{G}(x) \geq 1/(x-\bar{d}_-)$
and $\bar{G}$ is decreasing, we have $x \leq \bar{G}^{-1}(1/(x-\bar{d}_-))$.
Thus
\[\bar{G}^{-1}(A(s)) \geq \bar{G}^{-1}(1) \geq 1+\bar{d}_->\bar{d}_++0.1,\]
where the last inequality holds for all sufficiently small
$\beta$. Then \prettyref{lmm:Fderivatives}(b) implies
\begin{align}
|\cF(s)| &\leq O(\beta^4(1-q_*)^2) \cdot \sup_{x \in \supp(\mu_{\bar{D}})}
\Big|(\bar{G}^{-1}(A(s))-x)^{-1}\Big|,\label{eq:Fs}\\
|\partial_s \cF(s)| &\leq O(\beta^4(1-q_*)^2) \cdot \sup_{x \in \supp(\mu_{\bar{D}})}
\Big|\partial_s (\bar{G}^{-1}(A(s))-x)^{-1}\Big|,\label{eq:derFs}\\
\partial_s^2 \cF(s) & \leq O(\beta^4(1-q_*)^2) \cdot
\sup_{x \in \supp(\mu_{\bar{D}})}
\Big|\partial_s^2 (\bar{G}^{-1}(A(s))-x)^{-1}\Big|\label{eq:der2Fs}
\end{align}
where this third inequality (\ref{eq:der2Fs}) is a one-sided bound without
absolute value on the left side.
Here $\bar{G}^{-1}(A(s))=\bar{R}(A(s))+A(s)^{-1}$,
where $A(s) \in (0,1]$ at $s=0$. To further bound
(\ref{eq:Fs}--\ref{eq:der2Fs}),\footnote{One may apply more explicit
bounds for $\bar{G}^{-1}$ and its derivatives here, such as $|\bar{G}^{-1}(z)-\frac{1}{z}|\leq \beta\|D\|_\infty$, but the current argument allows an easier generalization to the 
second moment computation (cf.~\prettyref{lemma:Psi2optimal}).} we may apply the series expansion
for $\bar{R}(z)$ from (\ref{eq:barRseries}), recalling $\bar{\kappa}_1=0$, to
write
\begin{align}
\big(\bar{G}^{-1}(z)-x\big)^{-1}
=\big(\bar{R}(z)+z^{-1}-x\big)^{-1}
&=z\left(1-xz+\sum_{k \geq 2} \bar{\kappa}_k z^k\right)^{-1}\nonumber\\
&=z \cdot \sum_{j \geq 0} \left(xz-\sum_{k \geq 2} \bar{\kappa}_k
z^k\right)^j \triangleq \sum_{k \geq 0} c_k(x) z^{k+1}.\label{eq:Fseries}
\end{align}
Applying $|x| \leq C\beta$ and $|\bar{\kappa}_k| \leq (C\beta)^k$
for a constant $C=C(\mu_D)>0$ and all $k$, we have
\[|c_k(x)| \leq 2^{k-1} \cdot (C\beta)^k,\]
where $2^{k-1}$ is the number of ordered partitions of $k$ into positive
integers. Then for sufficiently small $\beta_0(\mu_D)>0$ and any $\beta \in
(0,\beta_0)$ and $z \in (0,1]$,
all summations of (\ref{eq:Fseries}) are absolutely
convergent, and the right side is an analytic power series for the
function $(\bar{G}^{-1}(z)-x)^{-1}$ on the left. The derivatives in $z$
may be computed term-by-term, to yield
\[\left|\big(\bar{G}^{-1}(z)-x\big)^{-1}\right|,
\left|\partial_z \big(\bar{G}^{-1}(z)-x\big)^{-1}\right|,
\left|\partial_z^2 \big(\bar{G}^{-1}(z)-x\big)^{-1}\right|=O(1).\]
Combining with (\ref{eq:derAs}) and applying this to
(\ref{eq:Fs}--\ref{eq:der2Fs}) using the chain rule, we obtain that at $s=0$,
$|\cF(s)|,|\partial_s \cF(s)|=O(\beta^4(1-q_*)^2)$
and $\partial_s^2 \cF(s) \leq C\beta^4(1-q_*)^2$,
for a constant $C=C(\mu_D)>0$.
Note that $B(s) \geq 0$, so this last inequality implies also
$\partial_s^2 \cF(s) \cdot B(s) \leq C\beta^4(1-q_*)^2 \cdot B(s)$.
 Then combining with (\ref{eq:derBs}) and applying this to
(\ref{eq:derIII}), we obtain the upper bound
$\nabla_{v,w}^2 \III \prec C'\beta^2$ for a constant $C'=C'(\mu_D)>0$.

Finally, for $\IV$, observe that by Proposition \ref{prop:infgamma}(a),
\[\IV=\frac{1}{2}\cH\Big(\gamma(v,w),1-\|v\|^2-\|w\|^2\Big)
=\frac{1}{2}\int_0^{1-\|v\|^2-\|w\|^2} \bar{R}(z)dz.\]
Writing as shorthand $f(s)=\int_0^{A(s)} \bar{R}(z)dz$ with $A(s)=1-\|v(s)\|^2-\|w(s)\|^2$ previously defined,
we have similarly
\[(v',w')^\top \nabla_{v,w}^2 \IV \cdot (v',w')
=\partial_s^2 \IV\Big|_{s=0}=\frac{1}{2}\partial_s^2 f(s)\Big|_{s=0}.\]
Applying again (\ref{eq:derAs}) and the bounds
$\bar{R}(z),\bar{R}'(z)=O(\beta^2)$ over $z \in (0,1)$ from Proposition
\ref{prop:smallbeta}, we obtain $\nabla_{v,w}^2 \IV=O(\beta^2)$.
Combining $\I$--$\IV$, we conclude
\[\nabla_{v,w}^2 \tilde{\Phi}_{1,t}(v,w) \prec
-2\beta^{1/2}I_{2t \times 2t}+O(\beta).\]
Then for some sufficiently small $\beta_0=\beta_0(\mu_D)>0$,
all $\beta \in (0,\beta_0)$, and any $t$, we have
\begin{equation}\label{eq:PhiHessupperbound}
\nabla_{v,w}^2 \tilde{\Phi}_{1,t}(v,w) \prec -\beta^{1/2}I_{2t \times 2t}
\end{equation}
over the whole domain $\{v,w \in \R^t:\|v\|^2+\|w\|^2<1\}$. In particular,
$\tilde{\Phi}_{1,t}$ is concave as claimed.

Finally, we argue that $(v,w)=(v_*,w_*)$ is an approximate maximizer for $\tilde{\Phi}_{1,t}(v,w)$. Indeed,
\[\partial_v \tilde{\Phi}_{1,t}(v_*,w_*)
=\partial_v \Phi_{1,t}+\partial_\gamma \Phi_{1,t}
\cdot \partial_v \gamma(v_*,w_*)+\partial_V \Phi_{1,t} \cdot \partial_v V(v_*)\]
where the derivatives of $\Phi_{1,t}$ are evaluated at
$(u_*,v_*,w_*;\gamma_*,U_*,V_*,W_*)$. Applying Lemma \ref{lemma:Psi1stationary}, we have 
$\partial_v \tilde{\Phi}_{1,t}(v_*,w_*)=o_t(1)$. Similarly,
$\partial_w \tilde{\Phi}_{1,t}(v_*,w_*)=0$. 
In view of (\ref{eq:PhiHessupperbound}), applying \prettyref{prop:approxmin} yields 
$\sup_{\|v\|^2+\|w\|^2<1} \tilde{\Phi}_{1,t}(v,w) = \tilde{\Phi}_{1,t}(v_*,w_*) +o_t(1)$.
Thus
\begin{align*}
\Psi_{1,t} &\leq 
\tilde{\Phi}_{1,t}(v_*,w_*) +o_t(1)= 
\Phi_{1,t}\Big(u_*,{v}_*,w_*;
\gamma_*,U_*,V_*,W_*\Big)+o_t(1)=\Psi_{\RS}+o_t(1),
\end{align*}
which is the desired (\ref{eq:Psi1upper}).
\end{proof}

Lemma \ref{lemma:firstmoment} follows immediately from Lemmas
\ref{lemma:Psi1correct} and \ref{lemma:Psi1optimal}.

\section{Conditional second moment}\label{sec:secondmoment}

We now provide a similar computation for the conditional second moment.

\begin{lemma}\label{lemma:secondmoment}
In the setting of Theorem \ref{thm:replicasymmetric},
\[\lim_{t \to \infty} \lim_{n \to \infty} \frac{1}{n}\log \E[Z^2 \mid \cG_t]
=2\Psi_{\RS}\]
where the inner limit as $n \to \infty$ exists almost surely for each fixed $t$.
\end{lemma}

\subsection{Derivation of the variational formula}

Define the domain
\begin{equation}\label{eq:barDplus}
\cD_+=\left\{(\gamma,\nu,\rho) \in \R^3:
\begin{pmatrix} \gamma & \nu \\ \nu & \rho \end{pmatrix}
\succ \bar{d}_+ \cdot I_{2 \times 2}\right\}.
\end{equation}
For scalar arguments
$(\gamma,\nu,\rho) \in \cD_+$ and $u,k,U,K,P \in \R$ and $p \in [-1,1]$,
and vector arguments
$v,w,\ell,m,V,W,L,M \in \R^t$ satisfying
\begin{equation}\label{eq:A2}
A(p,v,w,\ell,m) \triangleq \begin{pmatrix}
1-\|v\|^2-\|w\|^2 & p-v^\top \ell-w^\top m \\ p-v^\top \ell-w^\top m &
1-\|\ell\|^2-\|m\|^2 \end{pmatrix} \succ 0,
\end{equation}
we define 
\begin{align}
&\Phi_{2,t}(u,v,w,k,\ell,m,p;\gamma,\nu,\rho,U,V,W,K,L,M,P)\nonumber\\
&=\E\Big[\cL\Big(P,\;U \cdot \H+V^\top \Delta_t^{-1/2}(\X_1,\ldots,\X_t)
+\kappa_*^{-1/2}W^\top \Delta_t^{-1/2}(\Y_1,\ldots,\Y_t),\nonumber\\
&\hspace{1in}K \cdot \H+L^\top \Delta_t^{-1/2}(\X_1,\ldots,\X_t)
+\kappa_*^{-1/2}M^\top \Delta_t^{-1/2}(\Y_1,\ldots,\Y_t)\Big)\Big]\nonumber\\
&\hspace{0.2in}-u \cdot U-k \cdot K-v^\top V-w^\top W-\ell^\top L-m^\top M
-p \cdot P\nonumber\\
&\hspace{0.2in}+u+k+\bar{R}(1-q_*)\kappa_*^{-1/2}\Big(v^\top w+\ell^\top m\Big)
+\frac{\lambda_*-\bar{R}(1-q_*)\kappa_*^{-1}}{2}\Big(\|w\|^2+\|m\|^2\Big)
\nonumber\\
&\hspace{0.2in}
+\frac{1}{2} \Tr \cF(\gamma,\nu,\rho) \cdot B(v,w,\ell,m)
\nonumber\\
&\hspace{0.2in}
+\frac{1}{2}\cH\Big(\gamma,\nu,\rho;1-\|v\|^2-\|w\|^2,
p-v^\top \ell-w^\top m,1-\|\ell\|^2-\|m\|^2 \Big).\label{eq:Phi2}
\end{align}
Here, $\cL$ is a multivariate analogue of $\log 2 \cosh$ defined as
\begin{equation}\label{eq:L2}
\cL(x,y,z)=\log[e^{x+y+z}+e^{x-y-z}+e^{-x+y-z}+e^{-x-y+z}],
\end{equation}
$\cF(\gamma,\nu,\rho)$ denotes the univariate function
$\cF$ from (\ref{eq:Fgamma}) applied spectrally to
$(\begin{smallmatrix} \gamma & \nu \\ \nu & \rho \end{smallmatrix})$ via the
functional calculus,
$B$ is the $2 \times 2$-matrix-valued function
\begin{equation}\label{eq:B2}
B(v,w,\ell,m)=\begin{pmatrix} \|v-\kappa_*^{-1/2}w\|^2 &
(v-\kappa_*^{-1/2}w)^\top (\ell-\kappa_*^{-1/2}m) \\
(v-\kappa_*^{-1/2}w)^\top (\ell-\kappa_*^{-1/2}m) 
& \|\ell-\kappa_*^{-1/2}m\|^2 \end{pmatrix},
\end{equation}
and $\cH$ is the scalar-valued function
\begin{equation}\label{eq:H2}
\cH(\gamma,\nu,\rho;a,b,c)
=\Tr \begin{pmatrix} \gamma & \nu \\ \nu & \rho \end{pmatrix}
\begin{pmatrix} a & b \\ b & c \end{pmatrix}
-\int \log\det \begin{pmatrix} \gamma-x & \nu \\ \nu & \rho-x
\end{pmatrix}\mu_{\bar{D}}(dx)-\left(2+\log \det
\begin{pmatrix} a & b \\ b & c \end{pmatrix}\right).
\end{equation}

Define the variational formula
\begin{equation}\label{eq:Psi2}
\Psi_{2,t}=\mathop{\mathop{\sup_{u,k \in \R,\;p \in [-1,1]}}_{v,w,\ell,m \in
\R^t:A(p,v,w,\ell,m) \succ 0}} \inf_{(\gamma,\nu,\rho) \in \cD_+}
\mathop{\inf_{U,K,P \in \R}}_{V,W,L,M \in \R^t}
\Phi_{2,t}(u,v,w,k,\ell,m,p;\gamma,\nu,\rho,U,V,W,K,L,M,P).
\end{equation}

\begin{lemma}\label{lemma:Psi2correct}
In the setting of Theorem \ref{thm:replicasymmetric}, for any fixed $t \geq 1$,
almost surely
\[\lim_{n \to \infty} \frac{1}{n}\log \E[Z^2 \mid \cG_t]=\Psi_{2,t}.\]
\end{lemma}
\begin{proof}
The proof is analogous to that of Lemma \ref{lemma:Psi1correct}, and we will
omit details where the arguments are the same. We again fix $t$ and
write $\cG,X,Y,S,\Delta$ for $\cG_t,X_t,Y_t,S_t,\Delta_t$. We have
\[\E[Z^2 \mid \cG]=\sum_{\sigma,\tau \in \{+1,-1\}^n}
\exp\Big(h^\top \sigma+h^\top \tau+\frac{n}{2} \cdot f_n(\sigma,\tau)\Big),\]
where we define
\[f_n(\sigma,\tau)=\frac{2}{n}\log \E\left[\exp\left(\frac{1}{2}
\sigma^\top O^\top \bar{D}O\sigma
+\frac{1}{2}\tau^\top O^\top \bar{D}O\tau\right)\;\bigg|\;\cG\right].\]
We will approximate this function $f_n(\sigma,\tau)$ on the spheres
$\|\sigma\|^2=n$ and $\|\tau\|^2=n$.\\

\noindent {\bf Conditional law of $O$.}
Recall the shorthand $V=V_{(S,\Lambda S)^\perp}$ and
$\sigma_\perp,\sigma_\parallel$ from
(\ref{eq:sigmadecomp}), and define similarly
\[\tau_\perp=V_{(X,Y)^\perp}^\top \tau, \qquad \tau_\parallel=(S,\Lambda S)
\begin{pmatrix} X^\top X & X^\top Y \\ Y^\top X & Y^\top Y \end{pmatrix}^{-1}
(X,Y)^\top \tau.\]
Then similarly to (\ref{eq:firstmomentform}), an application of \prettyref{prop:Haarconditioning} yields
\begin{align*}
f_n(\sigma,\tau)&=\frac{1}{n}\sigma_\parallel^\top \bar{D}\sigma_\parallel
+\frac{1}{n}\tau_\parallel^\top \bar{D}\tau_\parallel\\
&\hspace{0.2in}+\frac{2}{n}\log
\E\Bigg[\exp\Bigg(\frac{1}{2}\sigma_\perp^\top \tilde{O}^\top V^\top
\bar{D}V\tilde{O}\sigma_\perp+\frac{1}{2}\tau_\perp^\top \tilde{O}^\top
V^\top \bar{D} V\tilde{O}\tau_\perp
+\sigma_\parallel^\top\bar{D}V\tilde{O}\sigma_\perp
+\tau_\parallel^\top\bar{D}V\tilde{O}\tau_\perp\Bigg)\Bigg].
\end{align*}\\

\noindent {\bf Expectation over $\tilde{O}$.} We first restrict to the domain
\[U_n=\left\{(\sigma,\tau) \in \R^n \times \R^n:
\|\sigma\|^2=n,\;\|\tau\|^2=n,\;\sigma_\perp \text{ and }
\tau_\perp \text{ are (non-zero and) linearly independent}\right\}.\]
In particular, $\sigma$ and $\tau$ must be different on this domain.
We evaluate the expectation over $\tilde{O}$ using Proposition
\ref{prop:Ointegralrank2}: Taking $a=\sigma_\perp$, $c=\tau_\perp$,
$b=V^\top\bar{D}\sigma_\parallel$, $d=V^\top\bar{D}\tau_\parallel$, and defining
$\Omega_n$ by some constants $\eps,C>0$ depending only on $\mu_D$ and $\beta$, 
for sufficiently small $\beta$ and all large $n$, we have
$(a,b,c,d) \in \Omega_n$. Then
\begin{equation}\label{eq:fnsigmatau}
f_n(\sigma,\tau)=\frac{1}{n}\sigma_\parallel^\top \bar{D}\sigma_\parallel
+\frac{1}{n}\tau_\parallel^\top \bar{D}\tau_\parallel
+E_n(\sigma,\tau)+r_n(\sigma,\tau)
\end{equation}
where
\begin{align*}
E_n(\sigma,\tau)&=\inf_{(\gamma,\nu,\rho) \in \cD_\eps}
\Bigg\{\frac{1}{n}\Tr \begin{pmatrix} \gamma & \nu \\
\nu & \rho \end{pmatrix}\begin{pmatrix} \|\sigma_\perp\|^2 &
\sigma_\perp^\top \tau_\perp \\ \sigma_\perp^\top \tau_\perp & \|\tau_\perp\|^2
\end{pmatrix}\\
&\hspace{1in}+\frac{1}{n}\begin{pmatrix} V^\top\bar{D}\sigma_\parallel \\
V^\top\bar{D}\tau_\parallel \end{pmatrix}^\top
\begin{pmatrix} \gamma I-V^\top \bar{D}V & \nu I \\ \nu I & \rho
I-V^\top \bar{D}V
\end{pmatrix}^{-1}\begin{pmatrix} V^\top\bar{D}\sigma_\parallel \\
V^\top\bar{D}\tau_\parallel \end{pmatrix}\\
&\hspace{1in}-\frac{1}{n} \log \det
\begin{pmatrix} \gamma I-V^\top \bar{D}V & \nu I \\ \nu I & \rho
I-V^\top \bar{D}V
\end{pmatrix}-\left(2+\log\det \frac{1}{n}\begin{pmatrix} \|\sigma_\perp\|^2 &
\sigma_\perp^\top \tau_\perp \\ \sigma_\perp^\top \tau_\perp &
\|\tau_\perp\|^2 \end{pmatrix}\right)\Bigg\}
\end{align*}
and
\begin{equation}\label{eq:Depsdef}
\cD_\eps=\left\{(\gamma,\nu,\rho) \in \R^3:
\begin{pmatrix} \gamma & \nu \\ \nu & \rho \end{pmatrix} \succeq
(\bar{d}_++\eps) \cdot I\right\}.
\end{equation}
We use $r_n(\sigma,\tau)$ to denote any remainder satisfying
\[\lim_{n \to \infty} \sup_{(\sigma,\tau) \in U_n} \|r_n(\sigma,\tau)\| \to 0\]
almost surely, and changing from instance to instance.\\

\noindent {\bf Approximation by $v,w,\ell,m,p$.}
Define the functionals
\begin{align*}
u(\sigma)=\frac{1}{n}h^\top\sigma,& \qquad
\begin{pmatrix} v(\sigma) \\ w(\sigma) \end{pmatrix}
=\left[\frac{1}{n}\begin{pmatrix} X^\top X & X^\top Y \\
Y^\top X & Y^\top Y \end{pmatrix}\right]^{-1/2}
\cdot \frac{1}{n}(X,Y)^\top \sigma\\
k(\tau)=\frac{1}{n}h^\top\tau,& \qquad
\begin{pmatrix} \ell(\tau) \\ m(\tau) \end{pmatrix}
=\left[\frac{1}{n}\begin{pmatrix} X^\top X & X^\top Y \\
Y^\top X & Y^\top Y \end{pmatrix}\right]^{-1/2}
\cdot \frac{1}{n}(X,Y)^\top \tau, \qquad
p(\sigma,\tau)=\frac{1}{n}\sigma^\top \tau.
\end{align*}
Then
\[\frac{\|\sigma_\perp\|^2}{n}=1-\|v(\sigma)\|^2-\|w(\sigma)\|^2,
\qquad \frac{\|\tau_\perp\|^2}{n}=1-\|\ell(\tau)\|^2-\|m(\tau)\|^2,\]
\[\frac{\sigma_\perp^\top \tau_\perp}{n}=p(\sigma,\tau)-v(\sigma)^\top \ell(\tau)
-w(\sigma)^\top m(\tau),\]
and
\begin{align*}
\sigma_\parallel&=S \cdot \Delta^{-1/2}v(\sigma)
+\Lambda S \cdot (\kappa_*\Delta)^{-1/2}w(\sigma)
+(S,\Lambda S) \cdot r_n(\sigma,\tau)\\
\tau_\parallel&=S \cdot \Delta^{-1/2}\ell(\tau)
+\Lambda S \cdot (\kappa_*\Delta)^{-1/2}m(\tau)
+(S,\Lambda S) \cdot r_n(\sigma,\tau).
\end{align*}

We approximate the terms of (\ref{eq:fnsigmatau}) using the low-dimensional parameters $v,w,\ell,m,p$:
Setting $a_*=\bar{R}(1-q_*)$ and
following arguments similar to (\ref{eq:Ifirstterm}),
\[\frac{\sigma_\parallel^\top \bar{D} \sigma_\parallel}{n}
+\frac{\tau_\parallel^\top \bar{D} \tau_\parallel}{n}
=\frac{2a_*}{\kappa_*^{1/2}}\Big(v(\sigma)^\top w(\sigma)
+\ell(\tau)^\top m(\tau)\Big)+\left(\lambda_*-\frac{a_*}{\kappa_*}\right)
\Big(\|w(\sigma)\|^2+\|m(\tau)\|^2\Big)+r_n(\sigma,\tau).\]
For approximating $E_n(\sigma,\tau)$, we will refer to the
eigen-decomposition
\begin{equation}\label{eq:gammanurhoeig}
\begin{pmatrix} \gamma & \nu \\ \nu & \rho \end{pmatrix}
=\begin{pmatrix} y_1 & y_2 \end{pmatrix}
\begin{pmatrix} \alpha_1 & 0 \\ 0 & \alpha_2 \end{pmatrix}
\begin{pmatrix} y_1^\top \\ y_2^\top \end{pmatrix}
\end{equation}
for $(\gamma,\nu,\rho) \in \cD_\eps$. Here $\alpha_1,\alpha_2 \geq
\bar{d}_++\eps$ are the eigenvalues, and $y_1 \in \R^2$ and $y_2 \in \R^2$
are the two corresponding eigenvectors. We write also
\begin{equation}\label{eq:VDVeig}
V^\top \bar{D}V={V'}^\top \bar{D}'V'=
{V'}^\top\diag(\bar{d}_1',\ldots,\bar{d}_{n-2t}') V'
\end{equation}
as the eigendecomposition of $V^\top \bar{D}V \in \R^{(n-2t) \times (n-2t)}$,
where $V' \in \R^{(n-2t) \times (n-2t)}$ is orthogonal and
$\bar{d}_i'$ are the eigenvalues.
Then as $n \to \infty$,
\begin{align*}
\frac{1}{n}\log \det \begin{pmatrix}\gamma I-V^\top \bar{D}V & \nu I \\ 
\nu I & \rho I-V^\top \bar{D}V \end{pmatrix}
&=\frac{1}{n}\sum_{i=1}^{n-2t} \log \det \begin{pmatrix} \gamma-\bar{d}_i' & \nu \\
\nu & \rho-\bar{d}_i' \end{pmatrix}\\
&\to \int \log \det \begin{pmatrix} \gamma-x & \nu \\
\nu & \rho-x \end{pmatrix} \mu_{\bar{D}}(dx).
\end{align*}
This convergence is uniform over $(\gamma,\nu,\rho) \in \cD_\eps$, because
the left side is
\[\frac{1}{n}\sum_{i=1}^{n-2t} \log(\alpha_1-\bar{d}_i')
+\frac{1}{n}\sum_{i=1}^{n-2t} \log(\alpha_2-\bar{d}_i'),\]
and the uniform convergence of each sum over
$\alpha_1,\alpha_2 \geq \bar{d}_++\eps$ was verified in the
first-moment calculation of Lemma \ref{lemma:firstmoment}. Thus, for
any $(\sigma,\tau) \in U_n$,
\begin{align*}
&\frac{1}{n}\Tr \begin{pmatrix} \gamma & \nu \\
\nu & \rho \end{pmatrix}\begin{pmatrix} \|\sigma_\perp\|^2 &
\sigma_\perp^\top \tau_\perp \\ \sigma_\perp^\top \tau_\perp & \|\tau_\perp\|^2
\end{pmatrix}-\frac{1}{n} \log \det
\begin{pmatrix} \gamma I-V^\top \bar{D}V & \nu I \\ \nu I & \rho
I-V^\top \bar{D}V
\end{pmatrix}\\
&\hspace{1in}-\left(2+\log\det \frac{1}{n}\begin{pmatrix} \|\sigma_\perp\|^2 &
\sigma_\perp^\top \tau_\perp \\ \sigma_\perp^\top \tau_\perp &
\|\tau_\perp\|^2 \end{pmatrix}\right)\\
&=\cH\Big(\gamma,\nu,\rho;
1-\|v(\sigma)\|^2-\|w(\sigma)\|^2,\; p(\sigma,\tau)-v(\sigma)^\top
\ell(\tau)-w(\sigma)^\top m(\tau),\;
1-\|\ell(\tau)\|^2-\|m(\tau)\|^2 \Big)\\
&\hspace{2in}+r_n(\gamma,\nu,\rho)
\end{align*}
where $r_n(\gamma,\nu,\rho) \to 0$ uniformly over
$(\gamma,\nu,\rho) \in \cD_\eps$.

For the remaining second term of $E_n(\sigma,\tau)$, let us write
\begin{equation}\label{eq:VDVblockinv}
\begin{pmatrix} \gamma I-V^\top\bar{D}V & \nu I \\ \nu I & \rho
I-V^\top\bar{D}V \end{pmatrix}^{-1}
=\begin{pmatrix} V' & 0 \\ 0 & V' \end{pmatrix}^\top
\begin{pmatrix} \gamma I-\bar{D}' & \nu I \\ \nu I & \rho I-\bar{D}'
\end{pmatrix}^{-1}\begin{pmatrix} V' & 0 \\ 0 & V' \end{pmatrix}.
\end{equation}
We may invert the matrix on the right by inverting separately the non-zero
$2 \times 2$ blocks,
\[\begin{pmatrix} \gamma-\bar{d}_i' & \nu \\ \nu & \rho-\bar{d}_i'
\end{pmatrix}^{-1}
=\frac{1}{\alpha_1-\bar{d}_i'}y_1y_1^\top+\frac{1}{\alpha_2-\bar{d}_i'}
y_2y_2^\top.\]
Then for each $j,k \in \{1,2\}$, the $(j,k)$ block of (\ref{eq:VDVblockinv}) is
\begin{align*}
\begin{pmatrix} \gamma I-V^\top\bar{D}V & \nu I \\ \nu I & \rho
I-V^\top\bar{D}V \end{pmatrix}_{jk}^{-1}
&=y_{1j}y_{1k}{V'}^\top\diag\left(\frac{1}{\alpha_1-\bar{d}_i'}\right)V'
+y_{2j}y_{2k}{V'}^\top\diag\left(\frac{1}{\alpha_2-\bar{d}_i'}\right)V'\\
&=y_{1j}y_{1k}(\alpha_1 I-V^\top \bar{D} V)^{-1}
+y_{2j}y_{2k}(\alpha_2 I-V^\top \bar{D} V)^{-1}.
\end{align*}
Let us consider first $j=k=1$. Then by (\ref{eq:Fconvergencefixed}) from the
first-moment calculation of Lemma \ref{lemma:firstmoment},
\begin{align*}
&y_{11}^2 \cdot \frac{\sigma_\parallel^\top \bar{D}V
(\alpha_1 I-V^\top \bar{D}V)^{-1}V^\top\bar{D}\sigma_\parallel}{n}
+y_{21}^2 \cdot \frac{\sigma_\parallel^\top \bar{D}V
(\alpha_2 I-V^\top \bar{D}V)^{-1}V^\top\bar{D}\sigma_\parallel}{n}\\
&=\Big(y_{11}^2\cF(\alpha_1)+y_{21}^2\cF(\alpha_2)\Big)
\cdot
\|v(\sigma)-\kappa_*^{-1/2}w(\sigma)\|^2+r_n(\sigma,\alpha_1,\alpha_2,y_{11},y_{21})
\end{align*}
where $r_n(\sigma,\alpha_1,\alpha_2,y_{11},y_{21}) \to 0$ uniformly over
$\alpha_1,\alpha_2 \geq \bar{d}_++\eps$, $y_{11},y_{21} \in [-1,1]$, and
$\{\sigma:\|\sigma\|^2=n,\sigma_\perp \neq 0\}$.
Similarly, for the other blocks $j,k \in \{1,2\}$,
\begin{align*}
&y_{11}y_{12} \cdot \frac{\sigma_\parallel^\top \bar{D}V
(\alpha_1-V^\top \bar{D}V)^{-1}V^\top\bar{D}\tau_\parallel}{n}
+y_{21}y_{22} \cdot \frac{\sigma_\parallel^\top \bar{D}V
(\alpha_2-V^\top \bar{D}V)^{-1}V^\top\bar{D}\tau_\parallel}{n}\\
&=\Big(y_{11}y_{12}\cF(\alpha_1)+y_{21}y_{22} \cF(\alpha_2)\Big)
\cdot (v(\sigma)-\kappa_*^{-1/2}w(\sigma))^\top
(\ell(\tau)-\kappa_*^{-1/2}m(\tau))+r_n(\sigma,\tau,\alpha_1,\alpha_2,
y_1,y_2),\\
&y_{12}^2 \cdot \frac{\tau_\parallel^\top \bar{D}V
(\alpha_1-V^\top \bar{D}V)^{-1}V^\top\bar{D}\tau_\parallel}{n}
+y_{22}^2 \cdot \frac{\tau_\parallel^\top \bar{D}V
(\alpha_2-V^\top \bar{D}V)^{-1}V^\top\bar{D}\tau_\parallel}{n}\\
&=\Big(y_{12}^2\cF(\alpha_1)+y_{22}^2\cF(\alpha_2)\Big)
\cdot
\|\ell(\tau)-\kappa_*^{-1/2}m(\tau)\|^2+r_n(\tau,\alpha_1,\alpha_2,y_{12},y_{22})
\end{align*}
where these remainders converge to 0 uniformly over $(\sigma,\tau) \in U_n$,
$\alpha_1,\alpha_2 \geq \bar{d}_++\eps$, and $y_{11},y_{12},y_{21},y_{22} \in
[-1,1]$. Combining these statements, we have for the second term of
$E_n(\sigma,\tau)$ that
\begin{align}
&\frac{1}{n}\begin{pmatrix} V^\top\bar{D}\sigma_\parallel \\
V^\top\bar{D}\tau_\parallel \end{pmatrix}^\top
\begin{pmatrix} \gamma I-V^\top\bar{D}V & \nu I \\ \nu I & \rho I-V^\top\bar{D}V
\end{pmatrix}^{-1}\begin{pmatrix} V^\top\bar{D}\sigma_\parallel \\
V^\top\bar{D}\tau_\parallel \end{pmatrix}\nonumber\\
&=\Tr \begin{pmatrix} y_{11} & y_{21} \\ y_{12} & y_{22} \end{pmatrix}
\begin{pmatrix} \cF(\alpha_1) & \\ & \cF(\alpha_2) \end{pmatrix}
\begin{pmatrix} y_{11} & y_{12} \\ y_{21} & y_{22} \end{pmatrix}
\cdot B(v(\sigma),w(\sigma),\ell(\tau),m(\tau))+r_n(\sigma,\tau,\gamma,\nu,\rho)
\nonumber\\
&=\Tr \cF(\gamma,\nu,\rho) \cdot
B(v(\sigma),w(\sigma),\ell(\tau),m(\tau))+r_n(\sigma,\tau,\gamma,\nu,\rho)
\label{eq:VDV2}
\end{align}
where $\cF(\gamma,\nu,\rho)$ is the function $\cF$ applied to
$(\begin{smallmatrix} \gamma & \nu \\ \nu & \rho \end{smallmatrix})$ spectrally,
and $r_n(\sigma,\tau,\gamma,\nu,\rho) \to 0$ uniformly over $(\sigma,\tau) \in
U_n$ and $(\gamma,\nu,\rho) \in \cD_\eps$.

Observe that this also implies, for any fixed vector $z \in \R^2$, with respect
to the positive-definite ordering for $(\begin{smallmatrix} \gamma & \nu \\
\nu & \rho\end{smallmatrix})$,
\begin{equation}\label{eq:F2monotonicity}
z^\top \cF(\gamma,\nu,\rho) z \text{ is non-increasing and convex over } 
(\gamma,\nu,\rho) \in \cD_+.
\end{equation}
Indeed, it suffices to show this for unit vectors $z=(z_1,z_2) \in \R^2$.
Fixing any $(\gamma,\nu,\rho) \in \cD_+$, we may take $\eps$ above small enough
such that $(\gamma,\nu,\rho)
\in \cD_\eps$. For each $n$, we may then take $(\sigma,\tau) \in U_n$ such that
$\|v(\sigma)\|^2 \to z_1^2$, $\|\ell(\tau)\|^2 \to z_2^2$,
$v(\sigma)^\top \ell(\tau) \to z_1z_2$, $\|w(\sigma)\|^2 \to 0$, and
$\|m(\tau)\|^2 \to 0$. (For example, we may choose
\[\sigma=\sqrt{n}\frac{z_1 x+(\sqrt{1-z_1^2}+\delta_n) r_1}{
\|z_1 x+(\sqrt{1-z_1^2}+\delta_n) r_1\|},
\qquad \tau=\sqrt{n}\frac{z_2 x+(\sqrt{1-z_2^2}+\delta_n) r_2}{
\|z_2 x+(\sqrt{1-z_2^2}+\delta_n) r_2\|}\]
where $x$
is the first column of $X$, $r_1$ and $r_2$ are vectors with
$\|r_1\|=\|r_2\|=\|x\|$ that are orthogonal to each
other and to the column span of $(X,Y)$, and $\delta_n \to 0$ as $n \to
\infty$.) Then as $n \to \infty$, the right side of (\ref{eq:VDV2}) converges to
$\Tr \cF(\gamma,\nu,\rho) \cdot (\begin{smallmatrix} z_1^2 & z_1z_2 \\ z_1z_2
& z_2^2 \end{smallmatrix})=z^\top \cF(\gamma,\nu,\rho)z$. The left side is
non-increasing with respect to the positive-definite ordering and convex
at $(\gamma,\nu,\rho)$, so the same properties hold for the limit
$z^\top \cF(\gamma,\nu,\rho)z$, showing (\ref{eq:F2monotonicity}).

Combining the above, we obtain the uniform approximation
over $(\sigma,\tau) \in U_n$
\begin{align}
f_n(\sigma,\tau)&=\inf_{(\gamma,\nu,\rho) \in \cD_\eps} \Bigg(
\frac{2a_*}{\kappa_*^{1/2}}(v(\sigma)^\top w(\sigma)+\ell(\tau)^\top m(\tau))
+\left(\lambda_*-\frac{a_*}{\kappa_*}\right)(\|w(\sigma)\|^2+\|m(\tau)\|^2)
\nonumber\\
&\hspace{0.5in}+\Tr \cF(\gamma,\nu,\rho) \cdot B(v(\sigma),w(\sigma),\ell(\tau),m(\tau))
+\cH\Big(\gamma,\nu,\rho;\;1-\|v(\sigma)\|^2-\|w(\sigma)\|^2,\nonumber\\
&\hspace{0.5in}p(\sigma,\tau)-v(\sigma)^\top \ell(\tau)-w(\sigma)^\top
m(\tau),\;1-\|\ell(\tau)\|^2-\|m(\tau)\|^2 \Big)\Bigg)+r_n(\sigma,\tau).
\label{eq:fnsigmatauapprox}
\end{align}
We now show that, for small $\beta$ and $\eps$, the above infimum over
$\cD_\eps$ is the same as that over the large domain $\cD_+$ in
(\ref{eq:barDplus}). Indeed, for any fixed $(\sigma,\tau) \in U_n$, denote by $S(\gamma,\nu,\rho)$
the quantity inside this infimum. Recall the eigendecomposition $(\begin{smallmatrix}\gamma & \nu\\\nu & \rho\end{smallmatrix})=\alpha_1 y_1y_1^\top+\alpha_2 y_2y_2^\top$ in 
(\ref{eq:gammanurhoeig}). For any $(\gamma,\nu,\rho) \in
\calD_+\backslash\calD_\eps$, we compare $S(\gamma,\nu,\rho)$ with
$S(\gamma',\nu'\rho')$, where
$(\begin{smallmatrix}\gamma' & \nu'\\\nu' &
\rho'\end{smallmatrix})=\max\{\alpha_1,\bar{d}_++\eps\}
y_1y_1^\top+\max\{\alpha_2,\bar{d}_++\eps\} y_2y_2^\top$ and
$(\gamma',\nu',\rho') \in \calD_\eps$. Note first that since $B(v,w,\ell,m)
\succeq 0$, (\ref{eq:F2monotonicity}) implies
\[\Tr \cF(\gamma',\nu',\rho')
\cdot B(v(\sigma),w(\sigma),\ell(\tau),m(\tau))
\leq \Tr \cF(\gamma,\nu,\rho) \cdot B(v(\sigma),w(\sigma),\ell(\tau),m(\tau)).\]
Next, let $\Delta$ denote the matrix derivative of the term
$\cH(\gamma,\nu,\rho;\cdot)$ of (\ref{eq:fnsigmatauapprox}),
\[\Delta \triangleq \begin{pmatrix} \partial_\gamma \cH(\gamma,\nu,\rho,\cdot)
& \frac{1}{2}\partial_\nu \cH(\gamma,\nu,\rho,\cdot) \\
\frac{1}{2}\partial_\nu \cH(\gamma,\nu,\rho,\cdot) &
\partial_\rho \cH(\gamma,\nu,\rho,\cdot) \end{pmatrix}\]
which has the explicit form
\[\Delta = A(p(\sigma,\tau),v(\sigma),w(\sigma),\ell(\tau),m(\tau))
-\int \begin{pmatrix} \gamma-x & \nu \\ \nu & \rho-x \end{pmatrix}^{-1}
\mu_{\bar{D}}(dx).\]
Since $(\gamma,\nu,\rho) \in \cD_+ \setminus \cD_\eps$, there is at least one
eigenvalue, say $\alpha_1$, which is less than $\bar{d}_++\eps$.
Then by the monotonicity of $\bar G$, for the corresponding eigenvector $y_1$, we have
\[y_1^\top \left(\int \begin{pmatrix} \gamma-x & \nu \\
\nu & \rho-x \end{pmatrix}^{-1} \mu_{\bar{D}}(dx)\right)y_1
=\bar{G}(\alpha_1)\geq \bar{G}(\bar{d}_++\eps).\]
So
\begin{align*}
\Tr \big[\Delta \cdot y_1y_1^\top\big]
\leq & ~ \Tr \big[A(p(\sigma,\tau),v(\sigma),w(\sigma),\ell(\tau),m(\tau)) \cdot
y_1 y_1^\top\big]-\bar{G}(\bar{d}_++\eps) \\
\leq & ~  4-\bar{G}(\bar{d}_++\eps) < 0
\end{align*}
where the second inequality is by Cauchy-Schwarz and the fact that all entries of $A$ are in $[-2,2]$, 
and the last inequality holds for $\beta \in (0,\beta_0)$ and sufficiently small
$\beta_0=\beta_0(\mu_D)>0$ and all sufficiently small $\eps$.
Integrating this bound from $\alpha_1$ to $\bar{d}_++\eps$, and also
from $\alpha_2$ to $\bar{d}_++\eps$ if $\alpha_2<\bar{d}_++\eps$, we obtain
$\cH(\gamma',\nu',\rho';\cdot)<\cH(\gamma,\nu,\rho;\cdot)$.
Combining the above, $S(\gamma',\nu',\rho') < S(\gamma,\nu,\rho)$.
This shows that $\inf_{\calD_\eps} S(\gamma,\nu,\rho) = \inf_{\calD_+}
S(\gamma,\nu,\rho)$.

Finally, observe that $(\sigma,\tau) \mapsto
(p(\sigma,\tau),v(\sigma),w(\sigma),\ell(\tau),m(\tau))$ is continuous,
relatively open, and maps $U_n$ onto the fixed domain
\begin{equation}
\calV\triangleq \{p \in [-1,1],v,w,\ell,m \in \R^t:
A(p,v,w,\ell,m) \succ 0\},
\label{eq:domain2}
\end{equation}
where $A(p,v,w,\ell,m)$ is as defined in \prettyref{eq:A2}.
Then,
applying Proposition \ref{prop:continuousextension} as in the proof of
Lemma \ref{lemma:Psi1correct} to extend the uniform approximation from $U_n$ to
its closure $\{\sigma,\tau \in \R^n:\|\sigma\|^2=\|\tau\|^2=n\}$, we obtain
\[\lim_{n \to \infty} \sup_{\sigma,\tau \in \R^n:\|\sigma\|^2=\|\tau\|^2=n}
\left|f_n(\sigma,\tau)
-f(p(\sigma,\tau),v(\sigma),w(\sigma),\ell(\tau),m(\tau))\right|=0\]
where we define for $(p,v,w,\ell,m) \in \calV$ the function
\begin{align*}
f(p,v,w,\ell,m) &\triangleq \inf_{(\gamma,\nu,\rho) \in \cD_+}
\frac{2a_*}{\kappa_*^{1/2}}(v^\top w+\ell^\top m)
+\left(\lambda_*-\frac{a_*}{\kappa_*}\right)(\|w\|^2+\|m\|^2)
+\Tr \cF(\gamma,\nu,\rho) \times\\
&\hspace{0.1in} B(v,w,\ell,m)
+\cH\Big(\gamma,\nu,\rho;\;1-\|v\|^2-\|w\|^2,\;
p-v^\top \ell-w^\top m,\;1-\|\ell\|^2-\|m\|^2 \Big),
\end{align*}
and extend this definition by continuity to the closure $\bar{\calV}$.\\

\noindent {\bf Large deviations analysis.}
Finally, writing $\langle \cdot \rangle$ for the expectation over
the independent discrete uniform laws $\sigma \sim
\Unif(\{+1,-1\}^n)$ and $\tau \sim \Unif(\{+1,-1\}^n)$, we may define
the limiting cumulant generating function
\begin{align*}
&\lambda(U,V,W,K,L,M,P)\\
&=\lim_{n \to \infty} \frac{1}{n}\log \Big\langle
\exp\Big[n(U \cdot u(\sigma)+V^\top v(\sigma)+W^\top w(\sigma)
+K \cdot k(\tau)+L^\top \ell(\tau)+M^\top m(\tau)+P \cdot p(\sigma,\tau)\Big]
\Big\rangle\\
&=\lim_{n \to \infty} \frac{1}{n}\log \Big\langle \exp\Big[
U \cdot h^\top \sigma+V^\top \Delta^{-1/2} X^\top \sigma
+W^\top(\kappa_*\Delta)^{-1/2}Y^\top \sigma\\
&\hspace{1.5in}+K \cdot h^\top \tau+L^\top \Delta^{-1/2} X^\top \tau
+M^\top(\kappa_*\Delta)^{-1/2}Y^\top \tau
+P \cdot \sigma^\top \tau+n \cdot r_n(\sigma,\tau)\Big]\Big\rangle
\end{align*}
where $r_n(\sigma,\tau) \to 0$ uniformly over $\sigma,\tau \in \{+1,-1\}^n$.
Evaluating the average over $(\sigma,\tau)$ using
\begin{equation}
\langle e^{x\sigma_i\tau_i+y\sigma_i+z\tau_i} \rangle
=e^{\cL(x,y,z)}/4,
\label{eq:Lxyz}
\end{equation}
 where $\cL(x,y,z)$ is as defined in (\ref{eq:L2}),
and applying the AMP convergence (\ref{eq:AMPconvergence}), this limit exists
and is given by
\begin{align*}
\lambda(U,V,W,K,L,M,P)&=
\E\Big[\cL\Big(P,\;U \cdot \H+V^\top \Delta_t^{-1/2}(\X_1,\ldots,\X_t)
+\kappa_*^{-1/2}W^\top \Delta_t^{-1/2}(\Y_1,\ldots,\Y_t),\\
&\hspace{0.2in}K \cdot \H+L^\top \Delta_t^{-1/2}(\X_1,\ldots,\X_t)
+\kappa_*^{-1/2}M^\top \Delta_t^{-1/2}(\Y_1,\ldots,\Y_t)\Big)\Big]
-\log 4.
\end{align*}
The proof is then concluded by the same argument as in the first-moment
calculation of Lemma \ref{lemma:firstmoment}, using the G\"{a}rtner-Ellis
Theorem and Varadhan's Lemma.
\end{proof}

\subsection{Analysis of the variational formula}

We now consider the approximate stationary point of (\ref{eq:Psi2}) given by
\[u_*=k_*=\E[\H \cdot \tanh(\H+\sigma_*\G)], \quad v_*=\ell_*=(1-q_*)\Delta_t^{1/2}e_t,
\quad w_*=m_*=\kappa_*^{1/2}(1-q_*)\Delta_t^{1/2}e_t,\]
\[\gamma_*=\rho_*=\bar{G}^{-1}(1-q_*), \quad \nu_*=0, \quad
U_*=K_*=1, \quad V_*=L_*=0, \quad W_*=M_*=\kappa_*^{1/2}\Delta_t^{1/2}e_t,\]
\[p_*=q_*, \quad P_*=0.\]
We write $\Phi_{2,t}(u_*,\ldots,P_*)$ for the evaluation of $\Phi_{2,t}$ at
this point.
We again verify in two steps that this approximately solves (\ref{eq:Psi2})
for $\beta>0$ sufficiently small.

For these steps, we require the following properties of
$\cF(\gamma,\nu,\rho)$ analogous to Lemma \ref{lmm:Fderivatives}.

\begin{lemma}\label{lmm:F2derivatives}
\begin{enumerate}[(a)]
\item For any fixed vector $z \in \R^2$, $z^\top \cF(\gamma,\nu,\rho)z$ is
non-increasing (with respect to the positive-definite ordering) and convex over
$(\gamma,\nu,\rho) \in \cD_+$.
\item Fix any $\delta>0$, open neighborhood $U \subset \R$, and twice
differentiable function $(\gamma,\nu,\rho):U \to \cD_\delta$ where $\cD_\delta$
is as defined in (\ref{eq:Depsdef}). Then for some
constants $C,\beta_0>0$ depending only on $\mu_D$ and $\delta$, any $s \in U$,
and all $\beta \in (0,\beta_0)$,
\begin{align*}
\|\cF(\gamma(s),\nu(s),\rho(s))\| &\leq C\beta^4(1-q_*)^2
\sup_{x \in \supp(\mu_{\bar{D}})} \big\|\big((\begin{smallmatrix} \gamma(s) &
\nu(s) \\ \nu(s) & \rho(s)\end{smallmatrix})-xI\big)^{-1}\big\|\\
\|\partial_s \cF(\gamma(s),\nu(s),\rho(s))\| &\leq C\beta^4(1-q_*)^2
\sup_{x \in \supp(\mu_{\bar{D}})} \big\|\partial_s
\big((\begin{smallmatrix} \gamma(s) &
\nu(s) \\ \nu(s) & \rho(s)\end{smallmatrix})-xI\big)^{-1} \big\|\\
\partial_s^2 \cF(\gamma(s),\nu(s),\rho(s)) &\preceq C\beta^4(1-q_*)^2
\sup_{x \in \supp(\mu_{\bar{D}})} \big\|\partial_s^2
\big((\begin{smallmatrix} \gamma(s) &
\nu(s) \\ \nu(s) & \rho(s)\end{smallmatrix})-xI\big)^{-1} \big\|
\cdot I_{2 \times 2}.
\end{align*}
\end{enumerate}
\end{lemma}
\begin{proof}
Part (a) was verified in (\ref{eq:F2monotonicity}).

For part (b), as in \prettyref{lmm:Fderivatives}, let us write $O(f(\beta))$
for a quantity bounded in magnitude by $C|f(\beta)|$ for a constant
$C=C(\mu_D,\delta)>0$, and interpret this entrywise for vectors and matrices. 
We again diagonalize
\[\begin{pmatrix} \gamma & \nu \\ \nu & \rho \end{pmatrix}
=\begin{pmatrix} y_1 & y_2 \end{pmatrix}
\begin{pmatrix} \alpha_1 & \\ & \alpha_2 \end{pmatrix}
\begin{pmatrix} y_1^\top \\ y_2^\top \end{pmatrix},\]
where $(y_1,y_2)$ are the two unit eigenvectors. Then by definition,
\[\cF(\gamma,\nu,\rho)=y_1y_1^\top \cdot \cF(\alpha_1)
+y_2y_2^\top \cdot \cF(\alpha_2)\]
where $\cF(\alpha)$ is the univariate function defined in
(\ref{eq:Fgamma}). Then $\|\cF(\gamma,\nu,\rho)\|
=\max(|\cF(\alpha_1)|,|\cF(\alpha_2)|)$, 
and also $\|((\begin{smallmatrix} \gamma(s) & \nu(s) \\ \nu(s) & \rho(s)
\end{smallmatrix})-xI)^{-1}\|
=\max(|\alpha_1-x|^{-1},|\alpha_2-x|^{-1})$, so the bound for
$\|\cF(\gamma(s),\nu(s),\rho(s))\|$ follows directly from Lemma
\ref{lmm:Fderivatives}.

To bound the derivatives, let us write $\cF(\gamma,\nu,\rho)$
in a more explicit form that parallels (\ref{eq:Fgamma}):
\begin{equation}\label{eq:F2}
\cF(\gamma,\nu,\rho)=\cF_{22}(\gamma,\nu,\rho)
-\cF_{12}(\gamma,\nu,\rho)^\top \cF_{11}(\gamma,\nu,\rho)^{-1} \cF_{12}(\gamma,\nu,\rho)
\end{equation}
where 
\begin{align}
\cF_{11}(\gamma,\nu,\rho)&=\int \begin{pmatrix} \gamma-x & \nu \\ \nu & \rho-x
\end{pmatrix}^{-1} \otimes \begin{pmatrix} 1 & \lambda(x) \\ \lambda(x)
& \lambda(x)^2 \end{pmatrix} \mu_{\bar{D}}(dx) \in \R^{4 \times 4},
\label{eq:F11secondmoment}\\
\cF_{12}(\gamma,\nu,\rho)&=\int \begin{pmatrix} \gamma-x & \nu \\ \nu & \rho-x
\end{pmatrix}^{-1} \otimes \begin{pmatrix} \theta(x) \\ \lambda(x)\theta(x)
\end{pmatrix} \mu_{\bar{D}}(dx) \in \R^{4 \times 2},\label{eq:F12secondmoment}\\
\cF_{22}(\gamma,\nu,\rho)&=
\int \begin{pmatrix} \gamma-x & \nu \\ \nu & \rho-x \end{pmatrix}^{-1}
\theta(x)^2\mu_{\bar{D}}(dx) \in \R^{2 \times 2}.\label{eq:F22secondmoment}
\end{align}
and $\lambda(x)$ and $\theta(x)$ were defined in (\ref{eq:LambdaThetax}).
To verify this form, 
recall the univariate $\calF_{11}(\gamma),\calF_{12}(\gamma),\calF_{22}(\gamma)$ defined in \prettyref{eq:F11}--\prettyref{eq:F22} and
observe that
\begin{align*}
\cF_{11}(\gamma,\nu,\rho)
&=\int \left(y_1y_1^\top \otimes \frac{1}{\alpha_1-x} 
\begin{pmatrix} 1 & \lambda(x) \\ \lambda(x)
& \lambda(x)^2 \end{pmatrix}
+y_2y_2^\top \otimes \frac{1}{\alpha_2-x} 
\begin{pmatrix} 1 & \lambda(x) \\ \lambda(x)
& \lambda(x)^2 \end{pmatrix}\right) \mu_{\bar{D}}(dx) \\
&=y_1y_1^\top \otimes \cF_{11}(\alpha_1)+y_2y_2^\top \otimes \cF_{11}(\alpha_2).
\end{align*}
Then, using $y_1y_1^\top \cdot y_2y_2^\top=y_2y_2^\top \cdot y_1y_1^\top=0$,
\begin{equation}
\cF_{11}(\gamma,\nu,\rho)^{-1}
=y_1y_1^\top \otimes \cF_{11}(\alpha_1)^{-1}
+y_2y_2^\top \otimes \cF_{11}(\alpha_2)^{-1}.
\label{eq:F11inv-spec}
\end{equation}
Similarly
\begin{align}
\cF_{12}(\gamma,\nu,\rho)= & ~ y_1y_1^\top \otimes \cF_{12}(\alpha_1)+y_2y_2^\top \otimes \cF_{12}(\alpha_2),\label{eq:F12-spec}	 \\
 \cF_{22}(\gamma,\nu,\rho)= & ~ y_1y_1^\top \otimes \cF_{22}(\alpha_1)
+y_2y_2^\top \otimes \cF_{22}(\alpha_2). \label{eq:F22-spec}	
\end{align}
Combining these yields the identity \prettyref{eq:F2}.

As in the proof of Lemma \ref{lmm:Fderivatives}, we use again the abbreviations $\lambda \equiv \beta(1-q_*)$ and
$\theta\equiv\beta^2(1-q_*)$. Then, from the forms
(\ref{eq:F11secondmoment}--\ref{eq:F22secondmoment}), for $k=1,2$,
\begin{align}
\partial_s^k \cF_{11}(\gamma(s))&=
O\left(\begin{pmatrix} 1 & 1 \\ 1 & 1 \end{pmatrix} \otimes
\begin{pmatrix} 1 & \lambda \\ \lambda & \lambda^2 \end{pmatrix}\right)
\cdot \sup_{x \in \supp(\mu_{\bar{D}})}
\big\|\partial_s^k \big((\begin{smallmatrix} \gamma(s) &
\nu(s) \\ \nu(s) & \rho(s)\end{smallmatrix})-xI\big)^{-1}\big\|
\label{eq:F2all11}\\
\partial_s^k \cF_{12}(\gamma(s))&=
O\left(\begin{pmatrix} 1 & 1 \\ 1 & 1 \end{pmatrix} \otimes
\begin{pmatrix} \theta \\ \lambda \theta \end{pmatrix}\right)
\cdot \sup_{x \in \supp(\mu_{\bar{D}})}
\big\|\partial_s^k \big((\begin{smallmatrix} \gamma(s) &
\nu(s) \\ \nu(s) & \rho(s)\end{smallmatrix})-xI\big)^{-1}\big\|
\label{eq:F2all12}\\
\partial_s^k \cF_{22}(\gamma(s))&=
O\left(\begin{pmatrix} 1 & 1 \\ 1 & 1 \end{pmatrix} \cdot \theta^2\right)
\cdot \sup_{x \in \supp(\mu_{\bar{D}})}
\big\|\partial_s^k \big((\begin{smallmatrix} \gamma(s) &
\nu(s) \\ \nu(s) & \rho(s)\end{smallmatrix})-xI\big)^{-1}\big\|.
\label{eq:F2all22}
\end{align}
Writing $\cF_{11}'=\partial_s \cF_{11}(\gamma(s),\nu(s),\rho(s))$ and similarly
for the other terms,
\[\cF'=\cF_{22}'-{\cF_{12}'}^\top \cF_{11}^{-1}\cF_{12}
-\cF_{12}^\top \cF_{11}^{-1}\cF_{12}'
+\cF_{12}^\top \cF_{11}^{-1}\cF_{11}'\cF_{11}^{-1}\cF_{12}.\]
Taking the product of (\ref{eq:F11inv-spec}) and (\ref{eq:F12-spec}),
\begin{align}
\cF_{11}^{-1}\cF_{12}
&=y_1y_1^\top \otimes [\cF_{11}(\alpha_1)^{-1}\cF_{12}(\alpha_1)]
+y_2y_2^\top \otimes [\cF_{11}(\alpha_2)^{-1}\cF_{12}(\alpha_2)]\nonumber\\
&=O\left(\begin{pmatrix} 1 & 1 \\ 1 & 1 \end{pmatrix}
\otimes \beta^{-2} \begin{pmatrix} \lambda^2 \theta \\ \lambda \theta
\end{pmatrix}\right),\label{eq:F11invF12}
\end{align}
where the second equality applies (\ref{eq:F11invbound}) and (\ref{eq:Fall12})
from Lemma \ref{lmm:Fderivatives}. Then, applying also
(\ref{eq:F2all11}--\ref{eq:F2all22}) for $k=1$, and $\lambda^2=O(\beta^2)$
and $\theta^2=\beta^4(1-q_*)^2$, we obtain
\[\|\cF'\|=O(\beta^4(1-q_*)^2) \cdot \big\|\partial_s
\big((\begin{smallmatrix} \gamma(s) & \nu(s) \\ \nu(s) &
\rho(s)\end{smallmatrix})-xI\big)^{-1}\big\|\]
which is the desired bound for $\|\partial_s \cF(\gamma(s),\nu(s),\rho(s))\|$.

For the second derivative, similar to \prettyref{eq:Fpp1}, we have
\begin{align*}
\cF''&=\cF_{22}''-{\cF_{12}''}^\top \cF_{11}^{-1} \cF_{12}
-\cF_{12}^\top \cF_{11}^{-1} \cF_{12}''
+\cF_{12}^\top \cF_{11}^{-1}\cF_{11}''\cF_{11}^{-1}\cF_{12}\\
&\hspace{0.5in}
-2\Big(\cF_{12}^\top\cF_{11}^{-1}\cF_{11}'\cF_{11}^{-1}\cF_{11}'\cF_{11}^{-1}\cF_{12}+{\cF_{12}'}^\top \cF_{11}^{-1}\cF_{12}'
+{\cF_{12}'}^\top [\cF_{11}^{-1}]'\cF_{12}
+\cF_{12}^\top [\cF_{11}^{-1}]'\cF_{12}'\Big).
\end{align*}
Bounding the terms on the first line using (\ref{eq:F11invF12}) and
(\ref{eq:F2all11}--\ref{eq:F2all22}) for $k=2$, and applying for the second line
\begin{align*}
&\cF_{12}^\top\cF_{11}^{-1}\cF_{11}'\cF_{11}^{-1}\cF_{11}'\cF_{11}^{-1}\cF_{12}+{\cF_{12}'}^\top \cF_{11}^{-1}\cF_{12}'
+{\cF_{12}'}^\top [\cF_{11}^{-1}]'\cF_{12}
+\cF_{12}^\top [\cF_{11}^{-1}]'\cF_{12}'\\
&\hspace{1in}=\Big[{\cF_{12}'}-\cF_{11}'\cF_{11}^{-1}\cF_{12}\Big]^\top
\cF_{11}^{-1}\Big[{\cF_{12}'}-\cF_{11}'\cF_{11}^{-1}\cF_{12}\Big] \succeq 0,
\end{align*}
we obtain
\[\cF'' \preceq O(\beta^4(1-q_*)^2) \cdot \sup_{x \in \supp(\mu_{\bar{D}})}
\big\|\partial_s^2
\big((\begin{smallmatrix} \gamma(s) & \nu(s) \\ \nu(s) &
\rho(s)\end{smallmatrix})-xI\big)^{-1}\big\|\]
which is the desired upper bound for
$\partial_s^2 \cF(\gamma(s),\nu(s),\rho(s))$.
\end{proof}

\begin{lemma}\label{lemma:Psi2stationary}
For each $\iota \in \{u,v,w,k,\ell,m,p,\gamma,\nu,\rho,U,W,K,M,P\}$, we have
\[\Phi_{2,t}(u_*,\ldots,P_*)=2\Psi_{\RS}, \qquad
\partial_\iota \Phi_{2,t}(u_*,\ldots,P_*)=0.\]
For $\iota \in \{V,L\}$, we have
\[\lim_{t \to \infty} \|\partial_\iota\Phi_{2,t}(u_*,\ldots,P_*)\|=0.\]
\end{lemma}
\begin{proof}
At $P_*=0$, we have $\cL(0,y,z)=\log (e^y+e^{-y})(e^z+e^{-z})
=\log 2\cosh y+\log 2\cosh z$. Recalling the definition of
$\calF(\gamma,\nu,\rho)$ by the spectral calculus, at $\nu_*=0$, we have
$\cF(\gamma,0,\rho) = \diag(\calF(\gamma),\calF(\rho))$, where 
$\cF(\cdot)$ is the function defined in (\ref{eq:Fgamma}). Hence
\[\Tr \cF(\gamma,0,\rho) \cdot B(v,w,\ell,m)
=\cF(\gamma) \cdot \|v-\kappa_*^{-1/2}w\|^2
+\cF(\rho) \cdot \|\ell-\kappa_*^{-1/2}m\|^2.\]
 At the above
$v_*,\ell_*,w_*,m_*,p_*$, from the computation in Lemma
\ref{lemma:Psi1stationary}, we also have
\begin{align*}
1-\|v_*\|^2-\|w_*\|^2&=1-(1+\kappa_*)(1-q_*)^2\delta_*=1-q_*,\\
1-\|\ell_*\|^2-\|m_*\|^2&=1-(1+\kappa_*)(1-q_*)^2\delta_*=1-q_*,\\
p_*-v_*^\top \ell_*-w_*^\top m_*&=q_*-(1+\kappa_*)(1-q_*)^2\delta_*=0,
\end{align*}
and
\[\cH(\gamma,0,\rho;\,1-q_*,0,1-q_*)=\cH(\gamma,1-q_*)+\cH(\rho,1-q_*)\]
where $\cH(\cdot,\cdot)$ on the right is the function (\ref{eq:Hgamma}). Thus,
\[\Phi_{2,t}(u_*,\ldots,P_*)=\Phi_{1,t}(u_*,v_*,w_*;\gamma_*,U_*,V_*,W_*)
+\Phi_{1,t}(k_*,\ell_*,m_*;\rho_*,K_*,L_*,M_*)=2\Psi_{\RS},\]
the second equality applying Lemma \ref{lemma:Psi1stationary}. Also, in view of \prettyref{eq:Lxyz}, 
\begin{equation}
\partial_x \calL(x,y,z)\big|_{x=0} = \tanh(y)\tanh(z), \quad
\partial_y \calL(x,y,z)\big|_{x=0} = \tanh(y), \quad \partial_z \calL(x,y,z)\big|_{x=0} = \tanh(z).
\label{eq:Lder}
\end{equation}
Furthermore,
\[\partial_\gamma \calF(\gamma,\nu,\rho)\big|_{\nu=0} = \partial_\gamma
\calF(\gamma)e_1e_1^\top, \qquad \partial_\rho
\calF(\gamma,\nu,\rho)\big|_{\nu=0} = \partial_\rho \calF(\rho)e_2e_2^\top,\]
\[\partial_\gamma \calH(\gamma,\nu,\rho;a,b,c)\big|_{\nu=0}=\partial_\gamma
\calH(\gamma,a), \qquad \partial_\rho
\calH(\gamma,\nu,\rho;a,b,c)\big|_{\nu=0}=\partial_\rho \calH(\rho,c).\]
Using these identities and applying Lemma \ref{lemma:Psi1stationary}, we obtain
\begin{align*}
\partial_\iota \Phi_{2,t}(u_*,\ldots,P_*)
&=\partial_\iota \Phi_{1,t}(u_*,v_*,w_*;\gamma_*,U_*,V_*,W_*)=0 \text{ for }
\iota \in \{u,\gamma,U,W\}\\
\partial_V \Phi_{2,t}(u_*,\ldots,P_*)
&=\partial_V \Phi_{1,t}(u_*,v_*,w_*;\gamma_*,U_*,V_*,W_*)=o_t(1),\\
\partial_\iota \Phi_{2,t}(u_*,\ldots,P_*)
&=\partial_\iota \Phi_{1,t}(k_*,\ell_*,m_*;\rho_*,K_*,L_*,M_*)=0 \text{ for }
\iota \in \{k,\rho,K,M\}\\
\partial_L \Phi_{2,t}(u_*,\ldots,P_*)
&=\partial_L \Phi_{1,t}(k_*,\ell_*,m_*;\rho_*,K_*,L_*,M_*)=o_t(1),
\end{align*}
where $o_t(1)$ denotes a length-$t$ vector satisfying 
$\lim_{t \to \infty} \|o_t(1)\|=0$.

It remains to check the derivatives in $\iota \in \{v,w,\ell,m,p,\nu,P\}$.
Since $v_*=\kappa_*^{-1/2}w_*$ and $\ell_*=\kappa_*^{-1/2}m_*$, we have
$B(v_*,w_*,\ell_*,m_*)=0$ and $\partial_\iota
B(v_*,w_*,\ell_*,m_*)=0$ for each $\iota \in \{v,w,\ell,m\}$. 
Writing $a_*=c_*=1-q_*$ and $b_*=0$, we have
\[\partial_b \cH(\gamma_*,\nu_*,\rho_*;a_*,b_*,c_*)
=\partial_\nu \cH(\gamma_*,\nu_*,\rho_*;a_*,b_*,c_*)=0\]
by the identities $\nu_*=b_*=0$ and
\begin{equation}\label{eq:detvanishes}
\partial_y \log \det \begin{pmatrix} x & y \\ y & z
\end{pmatrix}\Bigg|_{y=0}=0.
\end{equation}
Then it follows directly that
\[\partial_p \Phi_{2,t}(u_*,\ldots,P_*)=0,
\qquad \partial_\nu \Phi_{2,t}(u_*,\ldots,P_*)=0.\]
Furthermore, 
\begin{align*}
\partial_a \cH(\gamma_*,\nu_*,\rho_*;a_*,b_*,c_*)
&=\partial_a \cH(\gamma_*,a_*)=\bar{R}(1-q_*),\\
\partial_c \cH(\gamma_*,\nu_*,\rho_*;a_*,b_*,c_*)
&=\partial_c \cH(\rho_*,c_*)=\bar{R}(1-q_*),
\end{align*}
where the latter two equalities follow from \prettyref{eq:Hder}.
Applying also (\ref{eq:derv}--\ref{eq:derw}) and the identity 
$\lambda_*=\bar{R}(1-q_*)+(1-q_*)^{-1}$, we have
\begin{align*}
\partial_v \Phi_{2,t}(u_*,\ldots,P_*)&=-V_*
+\bar{R}(1-q_*)\kappa_*^{-1/2}w_*-\bar{R}(1-q_*)v_*=0,\\
\partial_w \Phi_{2,t}(u_*,\ldots,P_*)&=-W_*
+\bar{R}(1-q_*)\kappa_*^{-1/2}v_*+(\lambda_*-\bar{R}(1-q_*)\kappa_*^{-1})w_*
-\bar{R}(1-q_*)w_*=0;
\end{align*}
similarly $\partial_\ell \Phi_{2,t}(u_*,\ldots,P_*)=0$
and $\partial_m \Phi_{2,t}(u_*,\ldots,P_*)=0$.

Finally, for the derivative in $P$, 
applying \prettyref{eq:Lder}, together with
\begin{align*}
&U_* \cdot \H+V_*^\top \Delta_t^{-1/2}(\X_1,\ldots,\X_t)
+\kappa_*^{-1/2}W_*^\top \Delta_t^{-1/2}(\Y_1,\ldots,\Y_t)\\
&=K_* \cdot \H+L_*^\top \Delta_t^{-1/2}(\X_1,\ldots,\X_t)
+\kappa_*^{-1/2}M_*^\top \Delta_t^{-1/2}(\Y_1,\ldots,\Y_t)=\H+\Y_t,
\end{align*}
$p_*=q_*$, and the definition of $q_*$ from (\ref{eq:qstar}), we obtain
\[\partial_P \Phi_{2,t}(u_*,\ldots,P_*)=\E[\tanh(\H+\Y_t)^2]-p_*
=\E[\tanh(\H+\sigma_* \G)^2]-q_*=0.\]
\end{proof}

\begin{lemma}\label{lemma:Psi2optimal}
For a constant $\beta_0=\beta_0(\mu_D)>0$ and any $\beta \in (0,\beta_0)$,
\[\lim_{t \to \infty} \Psi_{2,t}=2\Psi_{\RS}.\]
\end{lemma}
\begin{proof}
The proof is analogous to that of Lemma \ref{lemma:Psi1optimal}.
We establish separately
\begin{align}
\liminf_{t \to \infty} \Psi_{2,t} &\geq 2\Psi_{\RS},\label{eq:Psi2lower}\\
\limsup_{t \to \infty} \Psi_{2,t} &\leq 2\Psi_{\RS}.\label{eq:Psi2upper}
\end{align}

Recall the max-min form of $\Psi_{2,t}$ in \prettyref{eq:Psi2}.
For the lower bound (\ref{eq:Psi2lower}), we specialize the outer supremum of $\Phi_{2,t}$ to
$(u,v,w,k,\ell,m,p)=(u_*,\tilde{v}_*,w_*,k_*,\tilde{\ell}_*,m_*,p_*)$ where
\begin{equation}\label{eq:vellapprox}
\begin{aligned}
\tilde{v}_*&=v_*+(1-q_*)[\Delta_t^{-1/2}\delta_t-\Delta_t^{1/2}e_t]
=v_*+o_t(1)\\
\tilde{\ell}_*&=\ell_*+(1-q_*)[\Delta_t^{-1/2}\delta_t-\Delta_t^{1/2}e_t]
=\ell_*+o_t(1)
\end{aligned}
\end{equation}
and $\delta_t=(\delta_{1,t+1},\ldots,\delta_{t,t+1})$
is as defined in the proof of Lemma \ref{lemma:Psi1optimal}. 
Note that 
\[
\Phi_{2,t}(u_*,\tilde{v}_*,w_*,k_*,\tilde{\ell}_*,m_*,p_*;\gamma,\ldots,P)=X(U,V,W,K,L,M,P)+Y(\gamma,\nu,\rho),
\] where 
both $X$ and $Y$ are convex functions. (Convexity of $Y$ holds by Lemma
\ref{lmm:F2derivatives}(a).)
Then
\begin{align}
\Psi_{2,t} &\geq 
\inf_{U,K,P \in \R,\;V,W,L,M \in \R^t} X(U,V,W,K,L,M,P)+\inf_{(\gamma,\rho,\nu)\in \calD_+} Y(\gamma,\nu,\rho).
\label{eq:Psi2lb}
\end{align}

Under the above definitions of $\tilde{v}_*$ and
$\tilde{\ell}_*$, the point $(U_*,V_*,W_*,K_*,L_*,M_*,P_*)$ is an exact
stationary point of $X$ hence its minimizer.
For the minimum of $Y$, note that 
$Y(\gamma,\nu,\rho) = \frac{1}{2} \Tr \cF(\gamma,\nu,\rho) \tilde B+
\frac{1}{2}\cH(\gamma,\nu,\rho;\tilde a_*,\tilde b_*,\tilde c_*)$, where we denote
$\tilde B = B(\tilde{v}_*,w_*,\tilde{\ell}_*,m_*) = o_t(1)$,
$\tilde a_*=\tilde c_*=1-\|\tilde v_*\|^2-\|w_*\|^2=1-\|\tilde \ell_*\|^2-\|m_*\|^2$, and
$\tilde b_*=p_* -\tilde v_*^\top\tilde \ell_*-w_*^\top m_*=q_* -\|\tilde v_*\|^2-\|w_*\|^2$.
Recalling the identity $\bar{G}(\gamma_*)=1-q_*=1-\|v_*\|^2-\|w_*\|^2$, we have for each $\iota \in \{\gamma,\nu,\rho\}$,
\[
\partial_\iota \calH(\gamma_*,\nu_*,\rho_*;\tilde a_*,\tilde b_*,\tilde c_*)=\|v_*\|^2-\|\tilde v_*\|^2=o_t(1).
\]
Therefore $\nabla Y(\gamma^*,\nu^*,\rho^*)=o_t(1)$.
Furthermore, there exist some constants $c,\delta>0$ independent of $t$, such that 
$\nabla Y(\gamma,\nu,\rho)\succeq c I$ whenever
$\|(\gamma,\nu,\rho)-(\gamma_*,\nu_*,\rho_*)\|\leq \delta$. 
Applying \prettyref{prop:approxmin} yields
$\inf_{(\gamma,\nu,\rho)\in\calD_+} Y(\gamma,\nu,\rho) \geq Y(\gamma_*,\nu_*,\rho_*)+o_t(1)$.
Note that 
$\|\nabla_{v,\ell} \Phi_{2,t}(u_*,v,w_*,k_*,\ell,m_*;\gamma_*,\ldots,P_*)\| \leq C$ 
for all $\|v-v_*\|\leq \delta$ and $\|\ell-\ell_*\|\leq \delta$, where $C,\delta$ are constants independent of $t$.
In view of \prettyref{eq:vellapprox} and \prettyref{eq:Psi2lb}, we then have
\[\Psi_{2,t} \geq \Phi_{2,t}(u_*,\ldots,P_*)+o_t(1)=2\Psi_{\RS}+o_t(1),\]
implying the lower bound (\ref{eq:Psi2lower}).

For the upper bound (\ref{eq:Psi2upper}), let $A \in \R^{2 \times 2}$ be a
symmetric matrix satisfying $0 \prec A \prec \bar{G}(\bar{d}_+)I$. We define
by the spectral calculus
\begin{equation}\label{eq:gammanurhochoice}
\begin{pmatrix} \gamma(A) & \nu(A) \\ \nu(A) & \rho(A) \end{pmatrix}
=\bar{G}^{-1}(A)=\bar{R}(A)+A^{-1} \succ \bar{d}_+ I.
\end{equation}
We then specialize the inner infimum of $\Phi_{2,t}$ to
\[(\gamma,\nu,\rho)=\big(\gamma(A(p,v,w,\ell,m)),\nu(A(p,v,w,\ell,m)),
\rho(A(p,v,w,\ell,m))\big),\]
\[U=U_*=1, \quad K=K_*=1, \quad
V=V(v)=\beta^{1/2}(v-v_*), \quad L=L(\ell)=\beta^{1/2}(\ell-\ell_*),\]
\[W=W(w)=\beta^{1/2}(w-w_*)+W_*, \quad M=M(m)=\beta^{1/2}(m-m_*)+M_*,\]
\[P=P(p)=\beta^{1/2}(p-p_*),\]
where $A(p,v,w,\ell,m)$ is in \prettyref{eq:A2}. Note that the above 
$(\gamma,\nu,\rho)$ is well defined for any $(p,v,w,\ell,m)$ in the domain
$\calV$ defined in \prettyref{eq:domain2}, provided that
$\beta<G(d_+)/2$, in which case $\bar{G}(\bar{d}_+) >2$. Indeed, since $p\in[-1,1]$, we have
$A(p,v,w,\ell,m) \preceq (\begin{smallmatrix} 1&p\\p&1\end{smallmatrix}) \preceq
2 I \prec \bar{G}(\bar{d}_+) I$.

At $(p,v,w,\ell,m)=(p_*,v_*,w_*,\ell_*,m_*)$, this specialization
gives $(\gamma,\nu,\rho)=(\gamma_*,\nu_*,\rho_*)$, because
$A(p_*,v_*,w_*,\ell_*,m_*)=(1-q_*)I_{2\times 2}$,
$(V,L,W,M)=(V_*,L_*,W_*,M_*)$, and $P=P_*=0$. Now
write the function $\Phi_{2,t}$ under this specialization (which no longer
depends on $u$ or $k$, thanks to the choice of $U$ and $K$) as
\[\tilde{\Phi}_{2,t}(p,v,w,\ell,m)=\I+\II+\III+\IV\]
where
\begin{align*}
\I&=\E\Big[\cL\Big(P(p),\H+V(v)^\top \Delta_t^{-1/2}(\X_1,\ldots,\X_t)
+\kappa_*^{-1/2}W(w)^\top\Delta_t^{-1/2}(\Y_1,\ldots,\Y_t),\\
&\hspace{2in}\H+L(\ell)^\top \Delta_t^{-1/2}(\X_1,\ldots,\X_t)
+\kappa_*^{-1/2}M(m)^\top\Delta_t^{-1/2}(\Y_1,\ldots,\Y_t)\Big)\Big]\\
\II&=-v^\top V(v)-w^\top W(w)-\ell^\top L(\ell)-m^\top M(m)-p \cdot P(p)\\
&\hspace{1in}+\bar{R}(1-q_*)\kappa_*^{-1/2}(v^\top w+\ell^\top m)
+\frac{\lambda_*-\bar{R}(1-q_*)\kappa_*^{-1}}{2}
\Big(\|w\|^2+\|m\|^2\Big)\\
\III&=\frac{1}{2}\Tr\cF\Big(\gamma(A(p,v,w,\ell,m)),\nu(A(p,v,w,\ell,m)),
\rho(A(p,v,w,\ell,m))\Big)
\cdot B(v,w,\ell,m)\\
\IV&=\frac{1}{2}\cH\Big(\gamma(A(p,v,w,\ell,m)),\nu(A(p,v,w,\ell,m)),
\rho(A(p,v,w,\ell,m));\\
&\hspace{1in}1-\|v\|^2-\|w\|^2,p-v^\top \ell-w^\top m,1-\|\ell\|^2-\|m\|^2\Big).
\end{align*}
Recalling the definition of $\Psi_{2,t}$ in \prettyref{eq:Psi2} and the domain $\calV$ in \prettyref{eq:domain2}, we have
\begin{align}
\Psi_{2,t} &\leq 
\sup_{(p,v,w,\ell,m)\in \calV} \tilde{\Phi}_{2,t}(p,v,w,\ell,m).
\label{eq:Psi2ub}
\end{align}
Note that $\calV$ is a convex set, since $A(p,v,w,\ell,m) \succ 0$ is equivalent
to $(\begin{smallmatrix}1&p&\\p&1\end{smallmatrix}) \prec
(\begin{smallmatrix}v&\ell&\\w&m\end{smallmatrix})^\top(\begin{smallmatrix}v&\ell&\\w&m\end{smallmatrix})$.

For all $\beta \in (0,\beta_0)$ and sufficiently small
$\beta_0=\beta_0(\mu_D)>0$, we claim that $\tilde{\Phi}_{2,t}$ is globally
concave on $\calV$. As in Lemma \ref{lemma:Psi1optimal}, we write $O(\beta^k)$
for a scalar, vector, or matrix of norm at most $C\beta^k$, uniformly over
$\calV$, for a constant $C=C(\mu_D)>0$ depending only on $\mu_D$.

For $\I$,
recall from \prettyref{eq:Lxyz} that $\cL(x,y,z)=\log (4 \langle
e^{x\sigma\tau+y\sigma+z\tau} \rangle )$, where $\langle\cdot\rangle$ is the
mean with respect to $\sigma$ and $\tau$ independent and uniform on $\{+1,-1\}$.
Thus its Hessian coincides with the covariance matrix of the random vector $r=(\sigma\tau,\sigma,\tau)$ 
\[
\nabla^2 \cL(x,y,z) = \langle rr^\top\rangle' - \langle r\rangle' \langle r\rangle'^\top,
\]
under the tilted distribution of $(\sigma,\tau)$ defined by $\langle
f(\sigma,\tau)\rangle'=\frac{\langle f(\sigma,\tau)
e^{x\sigma\tau+y\sigma+z\tau}\rangle}{\langle
e^{x\sigma\tau+y\sigma+z\tau}\rangle}$. So
\begin{equation}
0 \preceq \nabla^2 \cL(x,y,z)  \preceq 3 I.
\label{eq:HessL}
\end{equation}
Recall the random vector $Z_t\in\reals^{2t}$ defined in \prettyref{eq:Zt}, which satisfies $\Expect[Z_tZ_t^\top]=I$.
For any unit vector $q=(a,b,c) \in \R^{4t+1}$ where
$a\in\reals$ and $b,c\in\reals^{2t}$, 
define $\eta = (a, b^\top Z_t, c^\top Z_t) \in \R^3$. Then 
$q^\top (\nabla_{p,v,w,\ell,m}^2 \I) q = 
\beta \cdot \Expect[\eta^\top (\nabla^2\calL) \eta]$, where 
$\nabla^2\calL$ is evaluated in the same point as in the definition of $\I$. 
Applying \prettyref{eq:HessL}, we have 
$0 \leq q^\top (\nabla_{p,v,w,\ell,m}^2 \I) q \leq 3 \beta \cdot
\Expect[\|\eta\|^2] =3\beta$, and thus
$0 \preceq \nabla_{p,v,w,\ell,m}^2 \I \preceq 3 \beta \cdot I$.

For $\II$, by the same arguments as in Lemma \ref{lemma:Psi1optimal}, we have
$\nabla_{p,v,w,\ell,m}^2 \II=-2\beta^{1/2}I+O(\beta)$.

For $\III$, consider any scalar linear parametrization
\[(p(s),v(s),w(s),\ell(s),m(s))_{s \in \R}
=(p,v,w,\ell,m)+s \cdot (p',v',w',\ell',m')\]
where $\|(p',v',w',\ell',m')\|=1$. Write as shorthand the following $2\times 2$ matrices
\[A(s)=A\big(p(s),v(s),w(s),\ell(s),m(s)\big), \qquad
B(s)=B\big(v(s),w(s),\ell(s),m(s)\big),\]
and $\cF(s)=\cF(\gamma(A(s)),\nu(A(s)),\rho(A(s)))$.
As in Lemma \ref{lemma:Psi1optimal}, it is easily checked from the definitions
(\ref{eq:A2}) and (\ref{eq:B2}) and the bound
$\kappa_*=O(\beta^{-2}(1-q_*)^{-2})$ in Proposition \ref{prop:smallbeta}
that at $s=0$, we have
\begin{equation}
\|A(s)\|,\|\partial_s A(s)\|,\|\partial_s^2 A(s)\|=O(1), \qquad
\|B(s)\|,\|\partial_s B(s)\|,\|\partial_s^2 B(s)\|=O(\beta^{-2}(1-q_*)^{-2}).
\label{eq:derAs2}
\end{equation}
We may write
\begin{align*}
&(p',v',w',\ell',m')^\top \cdot \partial_{p,v,w,\ell,m}^2 \III
\cdot (p',v',w',\ell',m')\\
&=\partial_s^2 \III\Big|_{s=0}
=\frac{1}{2}\Tr \Big(\partial_s^2 \cF(s)
\cdot B(s)+2\partial_s \cF(s) \cdot \partial_s B(s)+\cF(s)
\cdot \partial_s^2 B(s)\Big)\Big|_{s=0}.
\end{align*}
Applying (\ref{eq:gammanurhochoice}) and Lemma \ref{lmm:F2derivatives},
we have analogous to Lemma \ref{lemma:Psi1optimal} that
\begin{align*}
\Big|\Tr \cF(s) \cdot \partial_s^2 B(s)\Big|
&\leq O(\beta^4(1-q_*)^2) \cdot \sup_{x \in \supp(\mu_{\bar{D}})}
\Big\|\Big(\bar{G}^{-1}(A(s))-xI\Big)^{-1}\Big\|
\cdot \Big\|\partial_s^2 B(s)\Big\|\\
\Big|\Tr \partial_s \cF(s) \cdot \partial_s B(s)\Big|
&\leq O(\beta^4(1-q_*)^2) \cdot \sup_{x \in \supp(\mu_{\bar{D}})}
\Big\|\partial_s \Big(\bar{G}^{-1}(A(s))-xI\Big)^{-1}\Big\|
\cdot \Big\|\partial_s B(s)\Big\|\\
\Tr \partial_s^2 \cF(s) \cdot B(s)
&\leq O(\beta^4(1-q_*)^2) \cdot \sup_{x \in \supp(\mu_{\bar{D}})}
\Big\|\partial_s^2 \Big(\bar{G}^{-1}(A(s))-xI\Big)^{-1}\Big\|
\cdot \Big\|B(s)\Big\|
\end{align*}
where the last inequality above applies $B(s) \succeq 0$ and holds
without absolute value on the left side.
Applying the series expansion (\ref{eq:Fseries}), for some
$\beta_0=\beta_0(\mu_D)>0$ and all $\beta \in (0,\beta_0)$, we have
\[\Big(\bar{G}^{-1}(A(s))-xI\Big)^{-1}
=\sum_{k \geq 0} c_k(x) A(s)^{k+1}\]
as a convergent matrix series. Then, differentiating in $s$ term-by-term,
\[\Big\|\Big(\bar{G}^{-1}(A(s))-xI\Big)^{-1}\Big\|,
\Big\|\partial_s \Big(\bar{G}^{-1}(A(s))-xI\Big)^{-1}\Big\|,
\Big\|\partial_s^2 \Big(\bar{G}^{-1}(A(s))-xI\Big)^{-1}\Big\|=O(1),\]
so $\nabla_{p,v,w,\ell,m}^2 \III \prec C\beta^2$ for a constant
$C=C(\mu_D)>0$.

For $\IV$, observe that by Proposition \ref{prop:infgamma}(b), we have
\[\IV=\frac{1}{2}\Tr f(A(p,v,w,\ell,m)), \qquad f(\alpha) \triangleq
\int_0^\alpha \bar{R}(z)dz.\]
Then similarly
\[(p',v',w',\ell',m')^\top \nabla_{p,v,w,\ell,m}^2 \IV \cdot
(p',v',w',\ell',m')=\partial_s^2 \IV\Big|_{s=0}
=\frac{1}{2} \Tr \partial_s^2 f(A(s))\Big|_{s=0}.\]
For all $\beta \in (0,\beta_0)$, we may integrate (\ref{eq:barRseries}) term by
term to write $f(A(s))$ as the convergent matrix series
\[f(A(s))=\sum_{k \geq 2} \frac{\bar{\kappa}_k}{k} A(s)^k,\]
where $|\bar{\kappa}_k|\leq (C\beta)^k$ for some $C=C(\mu_D)$.
Differentiating in $s$ at $s=0$ and using \prettyref{eq:derAs2}, we have for some constant $C'$ independent of $t$,
\[\left\|\partial_s^2 f(A(s))\Big|_{s=0}\right\| 
\leq \sum_{k \geq 2} (C'\beta)^k=O(\beta^2).\]
Then also $\nabla_{p,v,w,\ell,m}^2 \IV=O(\beta^2)$.

Combining the above, $\nabla_{p,v,w,\ell,m}^2
\tilde{\Phi}_{2,t}\prec -2\beta^{1/2}I_{(4t+1) \times (4t+1)}+O(\beta)$.
Applying \prettyref{lemma:Psi2stationary}, we have that 
$\nabla_{\iota} \tilde{\Phi}_{2,t}(p_*,v_*,w_*,\ell_*,m_*)=0$
for $\iota=p,w,m$ and $o_t(1)$ for $\iota=v,\ell$. Thus, recalling that
$\calV$ is convex and applying
\prettyref{prop:approxmin},
\[\sup_{(p,v,w,\ell,m) \in \calV} \tilde{\Phi}_{2,t}(v,w) =
\tilde{\Phi}_{2,t}(p_*,v_*,w_*,\ell_*,m_*)
+o_t(1)=\Phi_{2,t}(u_*,\ldots,P_*)+o_t(1)=2\Psi_{\RS}+o_t(1).\]
Then $\Psi_{2,t} \leq 2\Psi_{\RS}+o_t(1)$ in view of \prettyref{eq:Psi2ub}, proving the upper bound (\ref{eq:Psi2upper}).
\end{proof}

Lemma \ref{lemma:secondmoment} follows immediately from Lemmas
\ref{lemma:Psi2correct} and \ref{lemma:Psi2optimal}.

\section{Proof of Theorem \ref{thm:replicasymmetric}}
\label{sec:replicasymmetric-pf}
Finally, using Lemmas \ref{lemma:firstmoment} and \ref{lemma:secondmoment},
we conclude the proof of Theorem \ref{thm:replicasymmetric}.

\begin{proof}
We first show concentration of $n^{-1} \log Z$ around its mean: Writing
$\sigma^\top J\sigma=\Tr \sigma\sigma^\top O^\top DO$ and viewing $Z=Z(O)$ as
a function of $O \in \R^{n \times n}$, we have
\[\partial_O \log Z(O)=\frac{1}{Z}\sum_{\sigma \in \{+1,-1\}^n}
\beta \sigma\sigma^\top O^\top D \cdot
\exp\left(\frac{\beta}{2}\sigma^\top J\sigma+h^\top \sigma\right).\]
Then the Frobenius norm of this derivative
(for any $O \in \R^{n \times n}$) is bounded as
\[\Fnorm{\partial_O \log Z(O)} \leq \max_{\sigma \in \{+1,-1\}^n}
\Fnorm{\beta \sigma\sigma^\top O^\top D}
=\sqrt{n\beta^2 \cdot \sigma^\top O^\top D^2 O\sigma} \leq n\beta
\Opnorm{D}\Opnorm{O}.\]
So for any $O,O' \in \O(n)$,
integrating along a linear path from $O$ to $O'$ in $\R^{n \times n}$,
\[|\log Z(O)-\log Z(O')| \leq \|O-O'\|_F \cdot n\beta \Opnorm{D}.\]
We apply Gromov's concentration inequality in the form of
\cite[Corollary 4.4.28]{AGZ10}: Let $Q \sim \SO(n)$ and $O \sim \O(n)$ be
independent. Then for any $\eps>0$,
\begin{equation}\label{eq:concentrationprelim}
\P\left[\left|\frac{1}{n}\log Z(O)-\frac{1}{n}\E[\log Z(OQ) \mid O]
\right|>\eps\right] \leq 2\exp\left(
-\frac{\left(\frac{n}{4}-\frac{1}{2} \right)\eps^2}
{2\beta^2\Opnorm{D}^2}\right).
\end{equation}
For any diagonal sign matrix $P$ with diagonal entries $\{+1,-1\}$, note that
$O^\top DO=O^\top P^\top DPO$, so that $Z(O)=Z(PO)$. Then for any fixed
$O \in \O(n)$, the conditional expectation $\E[\log Z(OQ) \mid O]$ over
$Q \sim \Haar(\SO(n))$ coincides with that over $Q \sim \Haar(\O(n))$, which in
turn equals the unconditional expectation
$\E[\log Z(O)]$ over $O \sim \Haar(\O(n))$ by the invariance of the Haar measure.
Thus under Assumption \ref{assump:main}(b), for any $\eps>0$ and a
constant $c=c(\eps,\beta,\mu_D)>0$,
\begin{equation}\label{eq:concentration}
\P\left[\left|\frac{1}{n}\log Z-\frac{1}{n}\E \log Z\right| \leq \eps\right]
\geq 1-e^{-cn}.
\end{equation}

The remainder of the argument is the same as in \cite{bolthausen2018morita},
but for convenience we reproduce it here. Fix any $\eps>0$. First observe that
by Lemma \ref{lemma:firstmoment}, for a large enough iteration $t=t(\eps)$,
almost surely
\[\lim_{n \to \infty} \frac{1}{n} \log \E[Z \mid \cG_t]
\leq \Psi_{\RS}+\eps.\]
Since
\[\frac{1}{n} \log \E[Z \mid \cG_t] \leq
\log 2+\max_{\sigma \in \{+1,-1\}^n}
\frac{1}{n}\left(\frac{\beta}{2}\sigma^\top J\sigma+h^\top \sigma\right)
\leq \log 2+\frac{\beta}{2}\Opnorm{D}+\frac{1}{n}\sum_{i=1}^n |h_i|,\]
and the right side
has a constant upper bound under Assumption \ref{assump:main}, this
and Jensen's inequality yields
\[\frac{1}{n}\E \log Z \leq \E \frac{1}{n}\log \E[Z \mid \cG_t]
\leq \Psi_{\RS}+2\eps \text{ for all large } n.\]
For the complementary lower bound, for any $t \geq 1$,
Markov's inequality gives
\begin{align*}
\P\left[\frac{1}{n}\log Z \geq \Psi_{\RS}-\eps\right]
&=\E\left[\P\left[\frac{1}{n}\log Z \geq
\Psi_{\RS}-\eps\bigg|\cG_t\right]\right]\\
&\geq \P\left[\P\left[\frac{1}{n}\log Z \geq \Psi_{\RS}-\eps\bigg|\cG_t\right] \geq
e^{-cn/2}\right] \cdot e^{-cn/2},
\end{align*}
where we take $c>0$ to be the constant in (\ref{eq:concentration}) for this
$\eps$. Taking $t=t(\eps)$ large enough and applying
Lemma \ref{lemma:firstmoment} again, almost surely
\[\Psi_\RS-\eps \leq \lim_{n \to \infty}
\frac{1}{n}\log \frac{\E[Z \mid \cG_t]}{2}.\]
Then applying also the Paley-Zygmund inequality and
Lemma \ref{lemma:secondmoment}, for $t=t(\eps,c)$ large enough,
almost surely for all large $n$,
\begin{align*}
\P\left[\frac{1}{n}\log Z \geq \Psi_{\RS}-\eps\bigg|\cG_t\right]
&\geq \P\left[\frac{1}{n}\log Z \geq \frac{1}{n}\log \frac{\E[Z \mid
\cG_t]}{2}\bigg|\cG_t\right]\\
&=\P\left[Z \geq \frac{\E[Z \mid \cG_t]}{2}\bigg|\cG_t\right]
\geq \frac{\E[Z \mid \cG_t]^2}{4\E[Z^2 \mid \cG_t]} \geq e^{-cn/2}.
\end{align*}
Then for all large $n$,
\begin{equation}\label{eq:raredeviation}
\P\left[\frac{1}{n}\log Z \geq \Psi_{\RS}-\eps\right] \geq 0.99e^{-cn/2}
>e^{-cn}.
\end{equation}
Together (\ref{eq:concentration}) and (\ref{eq:raredeviation}) imply
\[\frac{1}{n}\E\log Z \geq \Psi_{\RS}-2\eps \text{ for all large } n.\]
Thus $n^{-1}\E\log Z \to \Psi_{\RS}$, and applying again the
concentration (\ref{eq:concentration}) finishes the proof of almost sure convergence by Borel-Cantelli.
\end{proof}

\appendix

\section{Analysis of AMP}\label{appendix:AMP}

We prove Propositions \ref{prop:qstarunique}, \ref{prop:AMPconvergent}, and
\ref{prop:smallbeta}, followed by Theorem \ref{thm:SE} and
Propositions \ref{prop:AMPconvergent} and \ref{prop:AMPfree}.

\begin{proof}[Proof of Proposition \ref{prop:qstarunique}]
Recall $\bar{R}(z)=\beta R(\beta z)$ from (\ref{eq:barGR}).
We note that the statement
(\ref{eq:Rsmallbeta}) in Proposition \ref{prop:smallbeta} immediately follows
from the series expansion (\ref{eq:barRseries}) for $\bar{R}(z)$,
where $\bar{\kappa}_1=0$, $\bar{\kappa}_2=\beta^2$,
and $\bar{\kappa}_k=O(\beta^k)$.

Set $t(x)=\tanh(x)^2$. The fixed-point equation (\ref{eq:qstarunscaled})
is equivalently given
in (\ref{eq:qstar}) by $f(q_*)=q_*$, where $f:[0,1] \to [0,1)$ is the function
\[f(q)=\E\left[t\left(\H+\sqrt{q\bar{R}'(1-q)} \cdot \G\right)\right].\]
Applying Gaussian integration by parts,
\begin{align*}
f'(q)&=\E\left[t'\left(\H+\sqrt{q\bar{R}'(1-q)} \cdot \G\right)
\cdot \frac{-q\bar{R}''(1-q)+\bar{R}'(1-q)}{2\sqrt{q\bar{R}'(1-q)}}\cdot
\G\right]\\
&=\E\left[t''\left(\H+\sqrt{q\bar{R}'(1-q)} \cdot \G\right)
\cdot\frac{-q\bar{R}''(1-q)+\bar{R}'(1-q)}{2}\right].
\end{align*}
We have $|t''(x)| \leq 2$. By (\ref{eq:Rsmallbeta}),
we have $|\bar{R}'(1-q)| \leq C\beta^2$ and
$|\bar{R}''(1-q)| \leq C\beta^3$ for all $q \in [0,1]$,
$\beta \in (0,\beta_0)$, and some
constants $C,\beta_0>0$ depending only on $\mu_D$.
So $|f'(q)|<1$ for any such $\beta$ and sufficiently small $\beta_0$. Then
$f:[0,1] \to [0,1)$ is contractive and has a unique fixed point $q_* \in [0,1)$.
\end{proof}

\begin{proof}[Proof of Proposition \ref{prop:kappadelta}]
Note that by \prettyref{assump:main}(b),
\[\bar{G}(z)=\lim_{n \to \infty} n^{-1}\Tr (zI-\bar{J})^{-1},
\qquad -\bar{G}'(z)=\lim_{n \to \infty} n^{-1}\Tr (zI-\bar{J})^{-2}.\]
Recall, by definition of $\lambda_*$ in (\ref{eq:lambdastar}),
that $\bar{G}(\lambda_*)=1-q_*$. Then by the definitions of $\kappa_*$ and
$\Gamma$ in (\ref{eq:kappadelta}) and (\ref{eq:Gamma}),
\begin{align*}
\kappa_*&=\lim_{n \to \infty} \Tr \left(\frac{1}{1-q_*}
(\lambda_* I-\bar{D})^{-1}-I\right)^2\\
&=\frac{1}{(1-q_*)^2}\Big({-}\bar{G}'(\lambda_*)\Big)
-\frac{2}{1-q_*}\bar{G}(\lambda_*)+1=-\frac{1}{(1-q_*)^2}\bar{G}'(\lambda_*)-1.
\end{align*}
We have $\bar{R}(z)=\bar{G}^{-1}(z)-1/z$, so that
$\bar{G}'(z)=1/[(\bar{G}^{-1})'(\bar{G}(z))]=1/[\bar{R}'(\bar{G}(z))-1/\bar{G}(z)^2)]$.
Then
\[\kappa_*=-\frac{1}{(1-q_*)^2} \cdot \frac{1}{\bar{R}'(1-q_*)
-(1-q_*)^{-2}}-1=\frac{1}{1-(1-q_*)^2\bar{R}'(1-q_*)}-1.\]
Substituting $\bar{R}'(1-q_*)=\sigma_*^2/q_*$ from the definition of
$\sigma_*^2$ in (\ref{eq:qstar}), this yields
\[\delta_*=\frac{\sigma_*^2}{\kappa_*}
=\frac{\sigma_*^2[1-(1-q_*)^2\sigma_*^2/q_*]}{(1-q_*)^2\sigma_*^2/q_*}
=\frac{q_*}{(1-q_*)^2}-\sigma_*^2.\]
The second equality of (\ref{eq:deltastar}) may be checked by expanding the
square on the right side, and applying the definition of $q_*$ in
(\ref{eq:qstar}) and Gaussian integration by parts.
\end{proof}

\begin{proof}[Proof of Proposition \ref{prop:smallbeta}]
As already argued, the statement (\ref{eq:Rsmallbeta})
follows from (\ref{eq:barRseries}).
This implies $\sigma_*^2=O(\beta^2)$ by its definition in (\ref{eq:qstar}).
Setting $t(x)=\tanh(x)^2$, we have
\[q_*=\E\big[t(\H)+t'(\H)\cdot\sigma_*\G+t''(\H')\cdot(\sigma_*^2\G^2/2)\big]\]
for some random variable $\H'$ between $\H$ and $\H+\sigma_*\G$.
Here $|t''(x)| \leq 2$ and $\E[t'(\H) \cdot \G]=0$, so
$q_*=\E[\tanh(\H)^2]+O(\beta^2)$. The remaining statements follow immediately
from (\ref{eq:Rsmallbeta}) and the forms of
$\sigma_*^2,\lambda_*,\kappa_*,\delta_*$ in (\ref{eq:qstar}),
(\ref{eq:lambdastar}), (\ref{eq:kappastar}), and (\ref{eq:deltastar}).
\end{proof}

\begin{proof}[Proof of Theorem \ref{thm:SE}]
The AMP algorithm (\ref{eq:AMPsmall1}--\ref{eq:AMPsmall2}) is a particular
instance of the more general algorithm studied in \cite[Eqs.~(4.2--4.3)]{fan2020approximate}, whose state evolution is obtained in \cite[Theorem 4.3]{fan2020approximate}. 
We apply this result with the notational
identifications $\mathbf{u}_t \leftrightarrow x^t$, $\mathbf{z}_t
\leftrightarrow y^t$, $\mathbf{W}\leftrightarrow \Gamma, \mathbf{\Lambda}\leftrightarrow \Lambda$, $\mathbf{E} \leftrightarrow h$,
$(Z_1,\ldots,Z_t,E) \leftrightarrow (\Y_1,\ldots,\Y_t,\H)$, and
\[u_{t+1}(Z_1,\ldots,Z_t,E) \leftrightarrow f(\H,\Y_t)
\triangleq (1-q_*)^{-1}\tanh(\H+\Y_t)-\Y_t.\]
Applying the property (\ref{eq:divergencefree}) for this function $f$,
the matrix $\mathbf{\Phi}_t$ of \cite[Eq.\ (4.4)]{fan2020approximate} satisfies
\[\lim_{n \to \infty} \mathbf{\Phi}_t=0.\]
Furthermore, by the definitions of $\lambda_*$ and $\kappa_*$ in
(\ref{eq:lambdastar}) and (\ref{eq:kappadelta}),
\[\frac{1}{n}\Tr \Lambda \to 0, \qquad \frac{1}{n}\Tr \Lambda^2 \to \kappa_*,\]
so that the second free cumulant of the empirical spectral distribution of $\Lambda$ converges to $\kappa_*$.
Then the matrices $\mathbf{\Theta}_t^{(j)}$, $\mathbf{B}_t$, and
$\mathbf{\Sigma}_t$ of \cite[Eqs.\ (4.5) and (4.7)]{fan2020approximate} satisfy
\begin{equation}\label{eq:Thetaform}
\lim_{n \to \infty} \mathbf{\Theta}_t^{(j)}=\begin{cases} \Delta_t & \text{ if
} j=0 \\ 0 & \text{ otherwise},\end{cases}
\qquad \lim_{n \to \infty} \mathbf{B}_t=0,
\qquad \lim_{n \to \infty} \mathbf{\Sigma}_t=\kappa_*\Delta_t,
\end{equation}
where we define
\[\Delta_t=\lim_{n \to \infty} n^{-1}X_t^\top X_t\]
provided that this limit exists.
Thus, (\ref{eq:AMPsmall1}--\ref{eq:AMPsmall2}) is a special case of the general
AMP algorithm of \cite[Section 4]{fan2020approximate}, replacing the debiasing
coefficients $b_{ts}$ therein by their large-$n$ limits $b_{ts}^\infty=0$.

From the initialization $y^0 \sim
\N(0,\sigma_*^2I)$, \cite[Proposition E.1]{fan2020approximate} ensures that the
empirical distribution of rows of $(h,y^0)$ converges almost surely
in the Wasserstein space $W_p$ to $(\H,\Y_0)$, for every $p \geq 1$. Since $f$
is Lipschitz,
the distribution of entries of $x^1=f(h,y^0)$ then converges in $W_p$ to $\X_1$.
By definition, $\lambda_*>\max(x:x \in \supp(\mu_{\bar{D}}))$, so
Assumption \ref{assump:main}(b) implies that the empirical eigenvalue
distribution of $\Lambda$ also converges in $W_p$ to a compactly supported
limit. The remaining conditions
of \cite[Assumption 4.2]{fan2020approximate} are easily checked from Assumption
\ref{assump:main}. Thus,
\cite[Theorem 4.3]{fan2020approximate} shows the distributional convergence
(\ref{eq:AMPconvergence}) in $W_p$, for any fixed $p \geq 1$. In particular, the
above matrix $\Delta_t$ is well-defined and non-singular for every $t \geq 1$,
and coincides with the definition of \prettyref{eq:Deltadef}. Thus (\ref{eq:XX}) holds. 

The limit (\ref{eq:YY}) then immediately follows from the distributional
convergence (\ref{eq:AMPconvergence}) and the specification of the law
$(\Y_1,\ldots,\Y_t) \sim \N(0,\kappa_*\Delta_t)$. The limit (\ref{eq:XY})
follows from writing each $\X_s$ as a function of $\Y_{s-1}$ according to
(\ref{eq:XYrelation}), and applying the divergence-free condition
(\ref{eq:divergencefree}) and Gaussian integration by parts for a
multivariate Gaussian vector---see \cite[Proposition E.5]{fan2020approximate}.
\end{proof}

\begin{proof}[Proof of Proposition \ref{prop:AMPconvergent}]
Since $\Y_0 \sim \N(0,\sigma_*^2)$, we have $\delta_{11}=\E[\X_1^2]=\delta_*$ by
(\ref{eq:XYrelation}) and the second equality of (\ref{eq:deltastar}).
Then $\kappa_*\delta_{11}=\sigma_*^2$ by definition of $\delta_*$ in
(\ref{eq:kappadelta}), so $\Y_1 \sim \N(0,\sigma_*^2)$ by the characterization
of its law in Theorem \ref{thm:SE}. The statements $\delta_{tt}=\delta_*$
and $\kappa_*\delta_{tt}=\sigma_*^2$ then hold for all $t \geq 1$ by induction.

To show the convergence $\delta_{st} \to \delta_*$ as $\min(s,t) \to \infty$, 
let us set $\delta_{0t}=\delta_{t0}=0$ for all $t \geq 0$.
We first show that $0 \leq \delta_{st} \leq \delta_*$ for all $s,t$. 
Observe that
\[\delta_{s+1,t+1}=\E[\X_{s+1}\X_{t+1}]=\E[f(\H,\Y_s)f(\H,\Y_t)],\]
where $f(h,y)=(1-q_*)^{-1}\tanh(h+y)-y$.  
By induction on $\min(s,t)$, it suffices to show that $\delta_{st} \in
[0,\delta_*]$ implies that $\delta_{s+1,t+1} \in [0,\delta_*]$. Represent the bivariate Gaussian law of $(\Y_s,\Y_t)$ as
\[(\Y_s,\Y_t)=\Big(\sqrt{\kappa_*\delta_{st}}\G
+\sqrt{\sigma_*^2-\kappa_*\delta_{st}}\G',\;
\sqrt{\kappa_*\delta_{st}}\G+\sqrt{\sigma_*^2-\kappa_*\delta_{st}}\G''\Big),\]
where $\G,\G',\G''$ are independent $\N(0,1)$ variables. Note that this
representation holds also when $s=0$ and/or $t=0$, because $\Y_0$ is independent
of $\Y_t$ for $t \neq 0$. Then $\delta_{s+1,t+1}=g(\delta_{st})$,
where $g(\delta)$ is the map defined on $[0,\delta_*]$ by
\[g(\delta)\triangleq\E\Big[f\Big(\H,\sqrt{\kappa_*\delta} \cdot \G
+\sqrt{\sigma_*^2-\kappa_*\delta} \cdot \G'\Big)
f\Big(\H,\sqrt{\kappa_*\delta} \cdot \G
+\sqrt{\sigma_*^2-\kappa_*\delta} \cdot \G''\Big)\Big].\]
Denote $\Y'=\sqrt{\kappa_*\delta} \cdot \G+\sqrt{\sigma_*^2-\kappa_*\delta}
\cdot \G'$, and define $\Y''$ similarly with $\G''$ in place of $\G'$.
By Cauchy-Schwarz, $|g(\delta)| \leq \E[f(\H,\Y')^2] = \delta_*$ by (\ref{eq:deltastar}).
At $\delta=\delta_*$, we have $\Y'=\Y''=\sigma_*\G$ and hence
$g(\delta_*)=\E[f(\H,\Y')^2]=\delta_*$.
Furthermore, taking the expectation first over $\G'$ and $\G''$, for any
$\delta \in [0,\delta_*]$ we have
\[g(\delta)=\E\Big[\E[f(\H,\Y') \mid \H,\G]^2\Big] \in [0,\delta_*]\]
as claimed.

Next, applying symmetry with respect to $(\Y',\Y'')$ and
Gaussian integration by parts, 
\begin{align*}
g'(\delta)&=2\E\left[\partial_y f(\H,\Y') \cdot f(\H,\Y'') \cdot
\left(\frac{\kappa_*}{2\sqrt{\kappa_*\delta}}\G-\frac{\kappa_*}
{2\sqrt{\sigma_*^2-\kappa_*\delta}}\G'\right)\right]\\
&=2\E\left[\partial_y^2 f(\H,\Y') \cdot f(\H,\Y'') \cdot \frac{\kappa_*}{2}
-\partial_y^2 f(\H,\Y') \cdot f(\H,\Y'') \cdot \frac{\kappa_*}{2}
+\partial_y f(\H,\Y') \cdot \partial_y f(\H,\Y'') \cdot
\frac{\kappa_*}{2}\right]\\
&=\kappa_*\E\left[\partial_y f(\H,\Y')\partial_y f(\H,\Y'')\right].
\end{align*}
Here $|\partial_y f(h,y)| \leq 2/(1-q_*)$. Then, applying
$\kappa_*=O(\beta^2(1-q_*)^2)$ by Proposition \ref{prop:smallbeta}, we have
$|g'(\delta)| \leq 1/2$ for any $\beta \in (0,\beta_0)$
and some constant $\beta_0>0$ depending only on $\mu_D$. 
So $g:[0,\delta_*] \to [0,\delta_*]$ is contractive, and $\delta_*$ is the
unique fixed point. We then have
\[|\delta_{st}-\delta_*| \leq
(1/2)^{\min(s,t)}|\delta_{s-\min(s,t),t-\min(s,t)}-\delta_*|
=(1/2)^{\min(s,t)}\delta_* \leq (1/2)^{\min(s,t)},\]
so $\lim_{\min(s,t) \to \infty} \delta_{st}=\delta_*$ as desired. Finally,
$\lim_{\min(s,t) \to \infty} \kappa_*\delta_{st} \to \sigma_*^2$
follows from $\sigma_*^2=\kappa_*\delta_*$.
\end{proof}

\begin{proof}[Proof of Proposition \ref{prop:AMPfree}]
Since $\bar{J}=O^\top \bar{D} O$, we have $f(\bar{J})=O^\top f(\bar{D})O$
by the functional calculus. Then applying $S_t=OX_t$ yields
$n^{-1}X_t^\top f(\bar{J})X_t=n^{-1}S_t^\top f(\bar{D})S_t$.

Let $\Lambda$ be as defined in (\ref{eq:Gamma}).
Applying \cite[Lemma A.4(b)]{fan2020approximate} with the notational
identification $\mathbf{r}_t \leftrightarrow s^t$, for each fixed integer $k
\geq 0$, almost surely
\[\lim_{n \to \infty} n^{-1} S_t^\top \Lambda^k S_t=\mathbf{L}_t^{(k,\infty)}.\]
This limit matrix $\mathbf{L}_t^{(k,\infty)}$ is defined by
\cite[Eq.\ (A.6) and Lemma A.1]{fan2020approximate}. Under the divergence-free
condition (\ref{eq:divergencefree}), applying (\ref{eq:Thetaform}), we have
simply
\[\mathbf{L}_t^{(k,\infty)}=m_k \cdot \Delta_t,
\qquad m_k=\lim_{n \to \infty} n^{-1}\Tr \Lambda^k
=\int \left(\frac{1}{1-q_*}(\lambda_*-x)^{-1}-1\right)^k \mu_{\bar{D}}(dx).\]
Define the increasing map $g:(-\infty,\lambda_*) \to (-1,\infty)$ by
\[g(x)=\frac{1}{1-q_*}(\lambda_*-x)^{-1}-1,\]
so that $\Lambda=g(\bar{D})$.
Then, for any fixed polynomial $p:\R \to \R$, this shows
\[\lim_{n \to \infty} n^{-1}S_t^\top p(\Lambda) S_t
=\Delta_t \cdot \int p(g(x)) \mu_{\bar{D}}(dx).\]

We apply Weierstrass polynomial approximation to extend the above to general continuous
functions: Let $g^{-1}:(-1,\infty) \to (-\infty,\lambda_*)$ be the functional
inverse of $g$.
Then, for any $f:\R \to \R$ which is continuous and bounded on a neighborhood of
$\supp(\mu_{\bar{D}})$, the function $f \circ g^{-1}$ is continuous and
bounded on some compact neighborhood $\calK$ of $g(\supp(\mu_{\bar{D}}))$.
Applying the Weierstrass approximation, for any $\eps>0$, there is a
polynomial $p$ for which
\[\max_{x \in \calK} |p(x)-f \circ g^{-1}(x)|<\eps.\]
Then
\begin{align*}
&\left\|\lim_{n \to \infty} n^{-1}S_t^\top f(\bar{D})S_t
-\Delta_t \cdot \int f(x)\mu_{\bar{D}}(dx)\right\|\\
&=\left\|\lim_{n \to \infty} n^{-1}S_t^\top (f \circ g^{-1}(\Lambda))S_t
-\Delta_t \cdot \int f \circ g^{-1}(g(x))\mu_{\bar{D}}(dx)\right\|\\
&\leq \limsup_{n \to \infty} \eps \cdot n^{-1}\|S_t\|^2
+\eps \cdot \|\Delta_t\| \leq \eps \cdot \Tr \Delta_t+\eps \cdot \|\Delta_t\|.
\end{align*}
This holds for any $\eps>0$, so
\[\lim_{n \to \infty} n^{-1}S_t^\top f(\bar{D})S_t
=\Delta_t \cdot \int f(x)\mu_{\bar{D}}(dx).\]
\end{proof}

\section{Large deviations for integrals over the orthogonal
group}\label{appendix:HCIZ}

\subsection{Proof of \prettyref{prop:Ointegralrank1}}
By applying a transformation $D \mapsto QDQ^\top$ and $b \mapsto Qb$ for an
orthogonal matrix $Q$, we may assume
without loss of generality that $D=\diag(d_1,\ldots,d_n)$ is diagonal.
Let $\mu_n=\frac{1}{n}\sum_{i=1}^n \delta_{d_i}$,
$d_{n,+} = \max d_i$, $d_{n,-} = \min d_i$, and $\Opnorm{D} = \max |d_i|$. 
Let 
\begin{equation}
G_n(\gamma) = \frac{1}{n} \Tr (\gamma I - D)^{-1}
= \frac{1}{n} \sum_{i=1}^n \frac{1}{\gamma-d_i}.
\label{eq:Gn}
\end{equation}

\begin{lemma}
\label{lmm:gammaloc}	
In the setting of Proposition \ref{prop:Ointegralrank1}, there exists
$n_0>0$ such that for any $n \geq n_0$ and any $(a,b) \in \Omega_n$,
the following holds: Set $\alpha=\|a\|^2/n$ and
\[F_n(\gamma)=G_n(\gamma)+\frac{b^\top (\gamma I-D)^{-2}b}{n}.\]
Then the equation
\begin{equation}
F_n(\gamma)=\alpha
\label{eq:gammanstar}
\end{equation}
has a unique solution $\gamma_n^* \in (d_++\eps,\infty)$, and
$|\gamma_n^*  - \alpha^{-1}| \leq C + \Opnorm{D}$.
\end{lemma}
\begin{proof}
By Assumption \ref{assump:main}(b), $\mu_n \to \mu_D$ weakly,
$d_{n,+}\to d_+$ as $n\to\infty$, and $\Opnorm{D}$ is bounded. 
Then $G_n(\gamma) $ converges to $G(\gamma)$ pointwise for each
$\gamma>d_+$. So for some $n_0>0$ and all $n \geq n_0$, we have
$d_{n,+}<d_++\eps$ and $G_n(d_++\eps)>G(d_++\eps)-\eps$. Then
	\[
	F_n(d_++\eps)\geq G_n(d_++\eps)>G(d_++\eps)-\eps \geq \alpha.
	\]
	Since $F_n(\gamma) \to 0$ monotonically as $\gamma\to\infty$, this
shows \prettyref{eq:gammanstar} has a unique solution $\gamma_n^*>d_++\eps$.
	Next, since $\|b\|^2 \leq Cn$, for any $\gamma>d_{n,+}$ we have
	$b^\top (\gamma I-D)^{-2} b/n \leq C/(\gamma-{d_{n,+}})^2$, and hence
	\[
	\frac{1}{\gamma-{d_{n,-}}} \leq F_n(\gamma) \leq \frac{1}{\gamma-{d_{n,+}}} + \frac{C}{(\gamma-{d_{n,+}})^2}.
	\]
Applying this at $\gamma=\gamma_n^*$ and $F_n(\gamma)=\alpha$, and rearranging,
	\[
	d_{n,-}  + \frac{1}{\alpha}  \leq \gamma_n^* \leq d_{n,+}  + \frac{1}{\alpha'},
	\]
	where $\alpha'=\frac{\sqrt{1+4C\alpha}-1}{2C}$.
	We conclude the proof by noting that $\frac{1}{\alpha'}-\frac{1}{\alpha} = \frac{2C}{\sqrt{1+4C\alpha}+1} \in [0,C]$.	
\end{proof}

\begin{proof}[Proof of \prettyref{prop:Ointegralrank1}]
We now bound the expectation in \prettyref{eq:Ointegralrank1} for any $(a,b) \in
\Omega_n$.
Let $g\sim \N(0,I_n)$ be a standard Gaussian vector. Then $\frac{g}{\|g\|}$ is
uniformly distributed over the sphere and 
$Oa \overset{L}{=}  \frac{g\|a\|}{\|g\|}$. Then
\[
\E\left[
\exp\left(b^\top Oa+\frac{a^\top O^\top DOa}{2}\right)\right]
=\E\left[\exp\left(\frac{\|a\|}{\|g\|} b^\top g +\frac{\|a\|^2}{2\|g\|^2} g^\top D g\right)\right].
\]

Let $\calE = \{g:|\|g\|^2/n-1| \leq \delta\}$ for some small $\delta$ to be specified. 
Since $\|g\|^2 \sim \chi^2_n$, by the $\chi^2$-tail bound (see e.g.~\cite[Lemma
1]{LM00}), we have for all $\delta \in (0,1)$,
\begin{equation}
\prob{g \in \calE} \geq 1 - 2 e^{-\delta^2 n/16}.
\label{eq:chitail}
\end{equation}
By the independence of $\|g\|$ and $\frac{g}{\|g\|}$, we have
\[
1 \leq \frac{\E\left[\exp\left(\frac{\|a\|}{\|g\|} b^\top g
+\frac{\|a\|^2}{2\|g\|^2} g^\top D
g\right)\right]}{\E\left[\exp\left(\frac{\|a\|}{\|g\|} b^\top g
+\frac{\|a\|^2}{2\|g\|^2} g^\top D g\right)\Indc_{g \in \calE}\right ]} =
\frac{1}{\prob{g \in \calE}} \leq \frac{1}{1 - 2 e^{-\delta^2 n/16}}.
\]

Set $\alpha=\|a\|^2/n$, and fix $\nu \in \R$ such that $\nu > d_{n,+} -
\frac{1}{\alpha}$. Then
\begin{align*}
& ~ \E\left[\exp\left(\frac{\|a\|}{\|g\|} b^\top g +\frac{\|a\|^2}{2\|g\|^2}
g^\top D g\right)\Indc_{g \in \calE}\right ] \\
\leq & ~ \E\left[\exp\left(\sqrt{\alpha} b^\top g +\frac{\alpha}{2} g^\top D g 
+ \frac{\alpha \nu}{2}(n-\|g\|^2) \right)\Indc_{g \in \calE}\right ]
\exp\underbrace{\pth{\delta \|a\|\|b\|+ \frac{1}{2}\delta  \|a\|^2 \Opnorm{D} + \frac{\alpha}{2} \delta|\nu| n }}_{\triangleq \tau} \\
\leq & ~ \prod_{i=1}^n 
\E\left[\exp\left(\sqrt{\alpha} b_i g_i - \frac{\alpha (\nu-d_i)}{2} g_i^2  \right) \right ] \exp\pth{\frac{\alpha \nu n}{2} + \tau} \\
= & ~ \prod_{i=1}^n \frac{1}{\sqrt{1 + \alpha (\nu-d_i)}}
\exp\pth{\frac{\alpha b_i^2}{2(1+\alpha (\nu-d_i))}}
\exp\pth{\frac{\alpha \nu n}{2} + \tau} \\
= & ~ \exp\pth{\sum_{i=1}^n \frac{\alpha \nu}{2} +
\frac{b_i^2}{2(\frac{1}{\alpha}+\nu-d_i)} - \frac{1}{2} \log(1 + \alpha (\nu-d_i)) }
\exp\pth{ \tau} \\
= & ~ \exp\sth{\frac{n}{2} \pth{\alpha \nu  + \frac{1}{n}b^\top\Big(
(\alpha^{-1}+\nu)I-D\Big)^{-1}b - \frac{1}{n}\log \det (I+\alpha (\nu I-D))}} \exp\pth{ \tau}.
\end{align*}
Next we minimize the leading term over $\nu>d_{n,+}-\frac{1}{\alpha}$.
Write $\nu=\gamma-\frac{1}{\alpha}$. Since the exponent is convex in $\nu$,
for all large $n$ the minimum is achieved at $\nu^*_n=\gamma_n^* - \frac{1}{\alpha}$, where $\gamma_n^*$ is previously defined as the unique solution on 
$(d_++\eps,\infty)$ to \prettyref{eq:gammanstar}, and this minimum is exactly $E_n(a,b)$ defined in \prettyref{eq:Ealpha}.
By \prettyref{lmm:gammaloc}, we have $|\nu^*_n| \leq C + \Opnorm{D}$.
Choosing $\delta=n^{-1/4}$ yields $\tau \leq C_1 n^{3/4}$ for some constant
$C_1$ depending on $\eps,C,\Opnorm{D}$ and $G(d_++\eps)$ only.
This proves
\[
\E\left[
\exp\left(b^\top Oa+\frac{a^\top O^\top DOa}{2}\right)\right] \leq
\frac{1}{1-2e^{-\sqrt{n}/16}} \exp\pth{\frac{n}{2} E_n(a,b) + C_1n^{3/4}} 
\]

For the lower bound,
\begin{align*}
& ~ \E\left[\exp\left(\frac{\|a\|}{\|g\|} b^\top g +\frac{\|a\|^2}{2\|g\|^2}
g^\top D g\right)\Indc_{g \in \calE}\right ] \\
\geq & ~ \E\left[\exp\left(\sqrt{\alpha} b^\top g +\frac{\alpha}{2} g^\top D g 
+ \frac{\alpha \nu}{2}(n-\|g\|^2) \right)\Indc_{g \in \calE}\right ] \exp(-\tau) \\
=& ~ \E\left[ \prod_{i=1}^n 
\exp\left(\sqrt{\alpha} b_i g_i - \frac{\alpha (\nu-d_i)}{2} g_i^2  \right)
\Indc_{g \in \calE} \right ] \exp\pth{\frac{\alpha \nu n}{2} - \tau} \\
= & ~ \exp\sth{\frac{n}{2} \pth{\alpha \nu  + \frac{1}{n}
b^\top\Big((\alpha^{-1}+\nu)I-D\Big)^{-1}b - \frac{1}{n} \log
\det (I + \alpha (\nu I-D))}} \prob{\tilde g \in \calE}
\exp\pth{ - \tau}
\end{align*}
where the last step follows from a change of measure from $g$ to $\tilde
g=(\tilde g_1,\ldots,\tilde g_n)$, whose coordinates are drawn independently as 
$g_i \sim \N(\mu_i,\sigma_i^2)$, with 
\[
\mu_i = \frac{\sqrt{\alpha} b_i}{1+\alpha (\nu-d_i)}, \quad \sigma_i^2 =
\frac{1}{1+\alpha (\nu-d_i)}.
\]
Note that 
\begin{align*}
\Expect[\|\tilde g\|^2] 
= & ~ \sum_{i=1}^n  (\mu_i^2+\sigma_i^2) = \frac{1}{\alpha} \sum_{i=1}^n \frac{b_i^2}{(\nu+1/\alpha - d_i)^2  } + \frac{1}{\nu +1/\alpha-d_i}	
= \frac{n}{\alpha} F_n\pth{\nu+\frac{1}{\alpha}}
\end{align*}
where $F_n$ is as defined in Lemma \ref{lmm:gammaloc}
As before, choose $\nu = \nu_n^*=\gamma_n^* - \frac{1}{\alpha}$, where
$\gamma_n^*$ is the solution to \prettyref{eq:gammanstar}. Then we have
$\Expect[\|\tilde g\|^2]=n$. Moreover,
\begin{align*}
\Var(\|\tilde g\|^2) = \sum_{i=1}^n (2 \sigma_i^4+4 \mu_i^2 \sigma_i^2)
= & ~ \sum_{i=1}^n \frac{2}{(1+\alpha (\nu-d_i))^2}+ \frac{4 \alpha
b_i^2}{(1+\alpha (\nu-d_i))^3} 	\\
= & ~ \frac{1}{\alpha^2} \pth{\sum_{i=1}^n
\frac{2}{(\gamma_n^*-d_i)^2}+\sum_{i=1}^n \frac{4 b_i^2}{(\gamma_n^*-d_i)^3}}.
\end{align*}
If $\frac{1}{\alpha} \leq 4(C+\Opnorm{D})$, we may apply
$\gamma_n^*>d_++\eps$ in \prettyref{lmm:gammaloc} and $\|b\|^2 \leq Cn$
to obtain $\Var(\|\tilde g\|^2) 
\leq \frac{n}{\alpha^2}\pth{\frac{2}{ \eps^2} + \frac{4C}{ \eps^3}}$.
If $\frac{1}{\alpha} \geq 4(C+\Opnorm{D})$, then we apply
$\gamma_n^* \geq \frac{1}{\alpha}-C-\Opnorm{D}$ from
\prettyref{lmm:gammaloc} to obtain $\gamma_n^*-d_i
\geq \frac{1}{\alpha}-2(C+\Opnorm{D}) \geq \frac{1}{2\alpha}$ and hence
$\Var(\|\tilde g\|^2) 
\leq n\pth{8 + 32 C\alpha}$. In both cases, we conclude that
\[
\Var(\|\tilde g\|^2) \leq C_2 n
\]
for some constant $C_2$ depending on $(C,\Opnorm{D},\eps)$.
By Chebyshev's inequality,
\[
\prob{\tilde g \not\in \calE} \leq \frac{\Var(\|\tilde g\|^2)}{(\delta n)^2} \leq \frac{C_2}{\sqrt{n}}.
\]
This shows
\[
\E\left[
\exp\left(b^\top Oa+\frac{a^\top O^\top DOa}{2}\right)\right] \geq
\frac{1}{1-\frac{C_2}{\sqrt{n}}} \exp\pth{\frac{n}{2} E_n(a,b) - C_1n^{3/4}}.
\]
Combining these upper and lower bounds completes the proof.
\end{proof}

\subsection{Proof of \prettyref{prop:Ointegralrank2}}

Let $s_1,\ldots,s_n \in \R^2$ be the rows of $(b,d) \in \R^{n \times 2}$.
We again assume without loss of generality that $D=\diag(d_1,\ldots,d_n)$ is
diagonal, and write $\mu_n,d_{n,+},d_{n,-},\Opnorm{D}$ as in the preceding
section. (Here $d_i$ are the diagonal entries of $D$, not the entries of the
vector $d$.)

Define
\begin{equation}
M\triangleq\begin{pmatrix} \frac{\|a\|^2}{n}
& \frac{a^\top c}{n} \\ \frac{a^\top c}{n} & \frac{\|c\|^2}{n} \end{pmatrix}.
\label{eq:covM}
\end{equation}
Define
\begin{equation}
\calF_n(\Lambda) \triangleq \Tr \Lambda M +\frac{1}{n} \sum_{i=1}^n \pth{s_i^\top (\Lambda-d_i I)^{-1} s_i - \log\det(\Lambda-d_i I)} -2-\log\det(M)
\label{eq:FLambda}
\end{equation}
so that $E_n$ defined in \prettyref{eq:En2} is given by
\[
E_n(a,b,c,d) = \inf_{\Lambda \succeq (d_++\eps) I} \calF_n(\Lambda).
\]

We have the following lemma that parallels \prettyref{lmm:gammaloc}:
\begin{lemma}
\label{lmm:Lambdaloc}	
Under the assumption of \prettyref{prop:Ointegralrank2}, there exists $n_0$ such that for all $n\geq n_0$ and all $(a,b,c,d)\in\Omega_n$, 
\[
\inf_{\Lambda \succeq (d_++\eps) I} \calF_n(\Lambda)
\]
is achieved at a unique minimizer $\Lambda^*$ such that 
$\Lambda^*\succ (d_++\eps) I$ and $\|\Lambda^*-M^{-1}\| \leq 2C + \Opnorm{D}$.
Furthermore, $\Lambda^*$ satisfies the equation
\begin{equation}
F_n(\Lambda) = M
\label{eq:Lambdastar}
\end{equation}
where
\begin{equation}
F_n(\Lambda) \triangleq \frac{1}{n} \sum_{i=1}^n (\Lambda-d_i I)^{-1} + \frac{1}{n} \sum_{i=1}^n  (\Lambda-d_i I)^{-1} s_is_i^\top (\Lambda-d_i I)^{-1}.
\label{eq:FnLambda}
\end{equation}
\end{lemma}

\begin{proof}[Proof of \prettyref{lmm:Lambdaloc}]
Let $n_0$ be sufficiently large such that 
$d_{n,+}<d_++\eps$ and $G_n(d_++\eps)>G(d_++\eps)-\eps$, where 
$G$ and $G_n$ are the Cauchy transform of $\mu_D$ and its empirical version, defined in \prettyref{eq:GR} and \prettyref{eq:Gn}. 
Write the gradient $\nabla \calF_n \triangleq \begin{pmatrix} \partial_{11} \calF_n& \frac{1}{2}\partial_{12} \calF_n \\ \frac{1}{2}\partial_{12} \calF_n  & \partial_{22} \calF_n\end{pmatrix}$ as a $2\times 2$ symmetric matrix. Then one can verify that
\[
\nabla \calF_n(\Lambda) = M - \underbrace{\frac{1}{n} \sum_{i=1}^n (\Lambda-d_i I)^{-1}}_{\triangleq G_n(\Lambda)} - \underbrace{\frac{1}{n} \sum_{i=1}^n  (\Lambda-d_i I)^{-1} s_is_i^\top (\Lambda-d_i I)^{-1}}_{\triangleq g_n(\Lambda)} = M - F_n(\Lambda).
\]

We first claim that $\inf_{\Lambda \succeq (d_++\eps) I} \calF_n(\Lambda)$
is attained at a unique minimizer $\Lambda^*$ satisfying $RI \succ \Lambda^*
\succ (d_++\eps)I$, for some $R>0$ depending only on $M,\mu_D,\eps$.
To this end, suppose $\Lambda$ has an eigenvalue $\lambda \geq R$ with unit-norm eigenvector $u$. Then
\begin{align*}
u^\top \nabla \calF_n(\Lambda) u
= & ~ u^\top M u - \frac{1}{n} \sum_{i=1}^n (\lambda-d_i)^{-1} - \frac{1}{n} \sum_{i=1}^n  (\lambda-d_i)^{-2} (s_i^\top u)^2 \\
\geq & ~ \lambda_{\min}(M) - (R-d_{n,+})^{-1} -  2 C (R-d_{n,+})^{-2},
\end{align*}
where the last inequality follows from Cauchy-Schwarz and the assumption that $\frac{1}{n}\sum \|s_i\|^2 = \frac{1}{n}(\|b\|^2+\|d\|^2) \leq 2C$.
Since $\lambda_{\min}(M)>0$ by assumption and $d_{n,+}<d_++\eps$,
for sufficiently large $R$ depending only on $M,\mu_D,\eps$,
we have $u^\top \nabla \calF_n(\Lambda) u>0$, and hence 
$\calF_n(\Lambda - \delta uu^\top) < \calF_n(\Lambda)$ for sufficiently small
$\delta$. Now suppose that $\Lambda$ has an eigenvalue equal to $d_++\eps$ with unit-norm eigenvector $u$. Then
\[u^\top \nabla \calF_n(\Lambda) u
\leq \lambda_{\max}(M)  - \frac{1}{n} \sum_{i=1}^n (d_++\eps-d_i)^{-1}
\leq G(d_++\eps)-\eps - G_n(d_++\eps)<0,\]
where we used the assumption that $M \preceq (G(d_++\eps)-\eps)I$ and
$G_n(d_++\eps)>G(d_++\eps)-\eps$.
Thus $\calF_n(\Lambda+\delta uu^\top) < \calF_n(\Lambda)$ for sufficiently
small $\delta$. In view of the strict convexity of $\calF_n$, this verifies our
claim. Furthermore, the unique minimizer $\Lambda^*$
must be a critical point of $\calF_n$, satisfying the gradient equation \prettyref{eq:Lambdastar}.

Finally, we show that 
$\|\Lambda^*-M^{-1}\| \leq 2C + \Opnorm{D}$ by showing that
\begin{align}
\Lambda^*\succeq & ~ 	M^{-1} + d_{n,-} I, \label{eq:Lambdaloc1}  \\
\Lambda^*\preceq & ~ 	M^{-1} + (d_{n,+}+2C) I. \label{eq:Lambdaloc2} 
\end{align}
Since $g_n(\Lambda) \succeq 0$, \prettyref{eq:Lambdaloc1} simply follows from
\[
M = F_n(\Lambda^*) \succeq G_n(\Lambda^*) \succeq (\Lambda^*-d_{n,-} I)^{-1}.
\]
To show \prettyref{eq:Lambdaloc2}, note that for any $x \in \R^n$, by
Cauchy-Schwarz and the bound $\frac{1}{n} \sum \|s_i\|^2 \leq 2C$, we have
\[
x^\top g_n(\Lambda) x = \frac{1}{n} \sum_{i=1}^n (s_i^\top (\Lambda-d_i I)^{-1}
x)^2 \leq 2 C x^\top(\Lambda-d_{n,+} I)^{-2} x.
\]
In other words, $g_n(\Lambda) \preceq 2 C (\Lambda-d_{n,+} I)^{-2}$. Writing
$Y=\Lambda^*-d_{n,+} I$, this shows
\[
M=F_n(\Lambda^*) \preceq Y^{-1} + 2 C Y^{-2}.
\]
Then
\begin{align*}
M^{-1} 
\succeq & ~  (Y^{-1} + 2 C Y^{-2})^{-1} = (Y^{-1/2}(I + 2 C Y^{-1})Y^{-1/2})^{-1} = Y^{1/2}(I + 2 C Y^{-1})^{-1} Y^{1/2} \\
 \succeq  & ~ Y^{1/2}(I - 2 C Y^{-1}) Y^{1/2}	= Y - 2 C I,
\end{align*}
where the second line applies $(I+X)^{-1} \succeq I-X$. Then
$Y\preceq M^{-1} + 2 CI$, which implies \prettyref{eq:Lambdaloc2}.
\end{proof}

\begin{proof}[Proof of \prettyref{prop:Ointegralrank2}]
	Let $g_1,g_2\sim \N(0,I_n)$ be independent standard Gaussian vectors.
	Let $g'_1,g'_2$ be their Gram-Schmidt orthogonalized versions
\[g'_1	= \frac{g_1}{\|g_1\|}, \qquad g'_2	= \frac{1}{\sin\theta}\left(
\frac{g_2}{\|g_2\|} - \cos\theta \frac{g_1}{\|g_1\|}\right)\]
where $\cos \theta=\frac{g_1^\top g_2}{\|g_1\|\|g_2\|}$ and $\theta \in
[0,\pi]$. Let
\[x_1=\|a\|g'_1, \qquad x_2 = \|c\|(g_1' \cos\phi + g_2' \sin\phi)\]
where $\cos\phi=\frac{a^\top c}{\|a\|\|c\|}$ and $\phi \in [0,\pi]$.
Then $(Oa,Oc) \overset{L}{=} (x_1,x_2)$ and 
\begin{align*}
& ~ \E\left[\exp\left(b^\top Oa+d^\top Oc
+\frac{a^\top O^\top DOa}{2}+\frac{c^\top O^\top DOc}{2}\right) \right] \\
= & ~ 
\E\left[\exp\left(b^\top x_1+ d^\top x_2+\frac{x_1^\top Dx_1}{2}+\frac{ x_2^\top Dx_2}{2} \right) \right]
\end{align*}
	
	Define the event
	\begin{equation}
	\calE = \sth{(g_1,g_2): |\|g_i\|^2/n-1| \leq \delta \text{ for } i =1,2, \text{ and } |\cos \theta| \leq \delta}
	\label{eq:eventE2}
	\end{equation}
	for some small $\delta \in (0,\frac{1}{2})$ to be specified.
Note that $\Expect[\exp(\lambda g_1^\top g_2)] = (1-\lambda^2)^{-n/2}$ for all
$|\lambda|<1$. Thus for $\lambda \in (0,1)$,
$\log \Expect[\exp(\lambda g_1^\top g_2)] = -\frac{n}{2}\log(1-\lambda^2) \leq \frac{n\lambda^2}{2(1-\lambda)}$.
By \cite[Theorem 2.3]{boucheron2013concentration}, we have 
$\prob{|g_1^\top g_2|\geq \sqrt{2nt}+t} \leq 2e^{-t}$.
Taking $t=\frac{\delta^2 n}{32}$ and using \prettyref{eq:chitail}, we conclude that
\[
\prob{(g_1,g_2) \in \calE} \geq 1- 6e^{-\delta^2 n/32}.
\]
	Crucially, $(g'_1,g'_2)$ and $(\|g_1\|,\|g_2\|,\cos\theta)$ are independent. Since the event $\{(g_1,g_2)\in\calE\}$ is measurable with respect to the latter, it is also independent of $(g'_1,g'_2)$. 
	Thus 
\[
1 \leq \frac{\E\left[\exp\left(b^\top x_1+ d^\top x_2+\frac{x_1^\top Dx_1}{2}+\frac{ x_2^\top Dx_2}{2} \right) \right]}{\E\left[\exp\left(b^\top x_1+ d^\top x_2+\frac{x_1^\top Dx_1}{2}+\frac{ x_2^\top Dx_2}{2} \right)  \Indc_{(g_1,g_2) \in \calE}\right ]} = 
\frac{1}{\prob{(g_1,g_2) \in \calE}} \leq \frac{1}{1- 6 e^{-\delta^2 n/32}}.
\]

Define 
\[
\xi \triangleq \frac{\|a\|}{\sqrt{n}} g_1, \quad \zeta \triangleq \|c\|\pth{\cos\phi\frac{g_1}{\sqrt{n}}+\sin\phi\frac{g_2}{\sqrt{n}}},
\]
which satisfy
$(\xi_i,\zeta_i) \overset{iid}{\sim} \N(0,M)$, with $M$ defined in
\prettyref{eq:covM}. On the event $\calE$, for an absolute constant $C'>0$,
we have the approximations
\begin{equation}
|b^\top x_1-b^\top \xi| \leq C'\delta \|a\|\|b\|, \quad |d^\top x_2-d^\top \zeta|
\leq C'\delta \|c\|\|d\|,\label{eq:xapprox1}
\end{equation}
\begin{equation}
|x_1^\top Dx_1 -\xi^\top D\xi| \leq C'\delta \Opnorm{D}\|a\|^2, \quad |x_2^\top
Dx_2 -\zeta^\top D\zeta| \leq C'\delta \Opnorm{D}\|c\|^2,\label{eq:xapprox2}
\end{equation}
\begin{equation}
|\|a\|^2-\|\xi\|^2| \leq C'\delta \|a\|^2, \quad |\|c\|^2-\|\zeta\|^2| \leq
C'\delta \|c\|^2, \quad |a^\top c-\xi^\top\zeta| \leq C'\delta \|a\|\|c\|.\label{eq:xapprox3}
\end{equation}

Fix any $(\gamma,\rho,\nu) \in\cD_\eps$ such that 
$\Lambda =\begin{pmatrix} \gamma & \nu \\ \nu & \rho \end{pmatrix} \succeq
(d_++\eps)I_{2 \times 2}$. 
Let $\Lambda' = \begin{pmatrix} \gamma' & \nu' \\ \nu' & \rho' \end{pmatrix} \triangleq \Lambda - M^{-1}$. 
Define
\[
\tau \triangleq C'\delta \left(\|a\|\|b\|+\|c\|\|d\|+\Opnorm{D}\|a\|^2+
\Opnorm{D}\|c\|^2+\frac{|\gamma'|}{2} \|a\|^2 +\frac{|\rho'|}{2}\|c\|^2+
|\nu'|\|a\|\|c\|\right).
\]
By the assumption of $(a,b,c,d) \in \Omega_n$, we have
\begin{equation}
\tau \leq C_0 \delta n (1+\|\Lambda'\|)
\label{eq:tau-bound}
\end{equation}
for some $C_0$ depending on $G(d_++\eps)$, $C'$, and the constant $C$ defining
$\Omega_n$.

Recall that $s_1,\ldots,s_n \in \R^2$ are the rows of
$(b,d) \in \R^{n \times 2}$, and write $z_1,\ldots,z_n \in \R^2$ for the rows
of $(g_1,g_2) \in \R^{n \times 2}$.
Then $z_i \overset{iid}{\sim} \N(0,I_2)$ and $(\xi_i,\zeta_i)=Tz_i$ for a matrix
$T$ satisfying $TT^\top=M$. Define
\[\mu_i \triangleq T^{-1}(\Lambda-d_i I)^{-1} s_i, \quad \Sigma_i
\triangleq T^{-1}(\Lambda-d_i I)^{-1}(T^{-1})^\top
\]
so that $\det \Sigma_i=\det(M)^{-1} \det(\Lambda-d_i I)^{-1}$.
Since $\Lambda \succeq (d_++\eps) I$, each $\Sigma_i$ is well-defined and positive definite.
By \prettyref{eq:xapprox1}--\prettyref{eq:xapprox3}, for some error term $r_n$ that satisfies $|r_n| \leq \tau$, we have
\begin{align*}
& ~ \E\left[\exp\left(b^\top x_1+ d^\top x_2+\frac{x_1^\top Dx_1}{2}+\frac{ x_2^\top Dx_2}{2} \right)  \Indc_{(g_1,g_2) \in \calE}\right ] \\
= & ~ \E\bigg[\exp\bigg(b^\top \xi+ d^\top \zeta+\frac{\xi^\top D\xi}{2}+\frac{
\zeta^\top D\zeta}{2}+\frac{\gamma'}{2}(\|a\|^2-\|\xi\|^2)\\
&\hspace{2.5in}+\frac{\rho'}{2}(\|c\|^2-\|\zeta\|^2)
+\nu'(a^\top c -\xi^\top \zeta)+r_n \bigg)
\Indc_{(g_1,g_2) \in \calE}\bigg ] \\
= & ~ \exp\pth{\frac{n}{2} \Tr\Lambda'M +r_n} \int\Indc_{(g_1,g_2) \in \calE}
\prod_{i=1}^n \exp\pth{s_i^\top Tz_i - \frac{1}{2}z_i^\top T^\top(\Lambda'-d_i
I)Tz_i } \frac{1}{2\pi} \exp\pth{-\frac{1}{2} \|z_i\|^2} \\
= & ~ \exp\pth{\frac{n}{2} (\Tr\Lambda M-2) +r_n} \int\Indc_{(g_1,g_2) \in
\calE} \prod_{i=1}^n \frac{1}{2\pi} \exp\pth{ - \frac{1}{2}(z_i-\mu_i)^\top
\Sigma_i^{-1} (z_i-\mu_i) + 
\frac{1}{2} \mu_i^\top \Sigma_i^{-1} \mu_i   }  \\
= & ~ \exp\sth{\frac{n}{2} (\Tr \Lambda M-2-\log\det M) + \frac{1}{2}
\sum_{i=1}^n \pth{-\log\det(\Lambda-d_i I) + s_i^\top (\Lambda-d_i I)^{-1} s_i}
+r_n}\\
&\hspace{2in}\times \prob{(\tilde g_1,\tilde g_2) \in \calE} \\
= & ~ \exp\sth{\frac{n}{2} \calF_n(\Lambda) +r_n}
\prob{(\tilde g_1,\tilde g_2) \in \calE},
\end{align*}
where $(\tilde g_1,\tilde g_2)$ consists of independent pairs $(\tilde
g_{i1},\tilde g_{i2}) \overset{\text{ind}}{\sim}\N(\mu_i,\Sigma_i)$.
	
	Now 
choose $\delta=n^{-1/4}$ and $\Lambda=\Lambda^*$ as in
\prettyref{lmm:Lambdaloc}. Then $\calF_n(\Lambda^*) = \inf_{\Lambda \succeq (d_++\eps)I}\calF_n(\Lambda)=E_n(a,b,c,d)$.
By \prettyref{lmm:Lambdaloc}, $\|\Lambda'\| = \|\Lambda^*-M^{-1}\| \leq
2C+\Opnorm{D}$. By \prettyref{eq:tau-bound}, we have 
$\tau \leq C_1 n^{3/4}$, which yields the desired upper bound in
\prettyref{eq:Ointegralrank2}. For the lower bound, we analyze $\prob{(\tilde
g_1,\tilde g_2) \in \calE}$ by a union bound:
\begin{equation}
\prob{(\tilde g_1,\tilde g_2) \notin \calE}
\leq \prob{\left|\frac{1}{n}\sum_{i=1}^n \tilde g_{i1}^2-1\right| \geq \delta} + \prob{\left|\frac{1}{n}\sum_{i=1}^n \tilde g_{i2}^2-1\right| \geq \delta}  + \prob{\frac{1}{n} \left|\sum_{i=1}^n \tilde g_{i1} \tilde g_{i2} \right|
 \geq \frac{\delta}{2}}.
\label{eq:Pg1g2}
\end{equation}
Furthermore, the gradient equation \prettyref{eq:Lambdastar} reads
\[
TT^\top=M = \frac{1}{n} \sum_{i=1}^n (\Lambda-d_i I)^{-1} + \frac{1}{n} \sum_{i=1}^n  (\Lambda-d_i I)^{-1} s_is_i^\top (\Lambda-d_i I)^{-1}.
\]
Thus at $\Lambda=\Lambda^*$ which satisfies this equation, we have
$n^{-1}\sum_{i=1}^n \pth{\mu_i\mu_i^\top + \Sigma_i}=I_2$, i.e.,
\[\frac{1}{n}\sum_{i=1}^n \Expect\tilde g_{i1}^2=\frac{1}{n}\sum_{i=1}^n
\Expect\tilde g_{i2}^2=1, \qquad
\frac{1}{n}\sum_{i=1}^n \Expect\tilde g_{i1}\tilde g_{i2}=0.\]
Note that 
$\Var(\tilde g_{i1}^2) = 4 \mu_{i1}^2 \Sigma_{i,11} + 2\Sigma_{i,11}^2$, 
$\Var(\tilde g_{i2}^2) = 4 \mu_{i2}^2 \Sigma_{i,22} + 2\Sigma_{i,22}^2$,
and 
$\Var(\tilde g_{i1}\tilde g_{i2}) = \mu_{i1}^2 \Sigma_{i,22} + \mu_{i2}^2 \Sigma_{i,11}+2\mu_{i1}\mu_{i2} \Sigma_{i,12}+\Sigma_{i,11}\Sigma_{i,22}+\Sigma_{i,12}^2$.
Applying $\|TT^\top\|=\|M\| \leq G(d_++\eps)$, we have
$\|\mu_i\|^2 = s_i^\top T \Sigma_i^2 T^\top s_i
\leq G(d_++\eps) \|\Sigma_i\|^2 \|s_i\|^2$.
Then applying Chebyshev's inequality to \prettyref{eq:Pg1g2}, we have
\[
\prob{(\tilde g_1,\tilde g_2) \notin \calE} \leq
\frac{100(1+G(d_++\eps))}{n\delta^2} \sum_{i=1}^n (\|s_i\|^2 \Tr(\Sigma_i)^3 +
\Tr(\Sigma_i)^2 ).
\]
Let $M=\sum_{j=1}^2 \alpha_j u_ju_j^\top$ be its eigenvalue decomposition,
and let $T=\sum_{j=1}^2 \sqrt{\alpha_j} u_jv_j^\top$ be the associated singular
value decomposition of $T$. Then 
\[
\Tr(\Sigma_i) = \sum_{j=1}^2 v_j^\top \Sigma_i v_j = \sum_{j=1}^2
\frac{1}{\alpha_j} u_j^\top (\Lambda^*-d_i I)^{-1} u_j.
\]
Recall from \prettyref{lmm:Lambdaloc} that 
$\Lambda^*\succeq (d_++\eps)I$ and $\Lambda^*\succeq M^{-1} - C_2 I$, where
$C_2=2C+\Opnorm{D}$.
Thus $u_j^\top (\Lambda^*-d_i I)^{-1} u_j \leq \frac{1}{\eps}$ always, 
and 
$u_j^\top (\Lambda^*-d_i I)^{-1} u_j \leq \frac{\alpha_j}{1-(C_2+d_i) \alpha_j}$
provided $\alpha_j < \frac{1}{C_2+d_i}$.
Overall, we have $\Tr(\Sigma_i) \leq C_3$ for some $C_3$ depending on
$(C,\Opnorm{D},\eps)$.
Consequently,
$\prob{(\tilde g_1,\tilde g_2) \notin \calE} \leq C_4/\sqrt{n}$, for some
constant $C_4$ depending on $(C,\Opnorm{D},\eps,G(d_++\eps))$.
This completes the required lower estimate for \prettyref{eq:Ointegralrank2}.
\end{proof}

\subsection{Proof of \prettyref{prop:infgamma}}

\begin{proof}
For part (a), write $\cH(\gamma,\alpha)$ for the function inside the infimum.
This is strictly convex over $\gamma>d_+$, and its derivative is
$\partial_\gamma \cH(\gamma,\alpha)=\alpha-G(\gamma)$.
For $\alpha \in (0,G(d_+))$, this derivative vanishes at
$\gamma=G^{-1}(\alpha)$, so $\gamma=G^{-1}(\alpha)$
must be the minimizer by convexity. At this
minimizer, writing $G^{-1}(\alpha)=R(\alpha)+\alpha^{-1}$ and combining the
logarithmic terms,
\[\cH(G^{-1}(\alpha),\alpha)=\alpha R(\alpha)-\int \log(\alpha R(\alpha)+1-
\alpha x)\mu_D(dx).\]
This evaluates to 0 at $\alpha=0$. Its derivative in $\alpha$ is
\begin{align*}
&R(\alpha)+\alpha R'(\alpha)-\frac{1}{\alpha}\int
\frac{R(\alpha)+\alpha R'(\alpha)-x}{R(\alpha)+\alpha^{-1}-x}\mu_D(dx)\\
&=R(\alpha)+\alpha R'(\alpha)-\alpha^{-1}+\alpha^{-1}
\int \frac{\alpha^{-1}-\alpha R'(\alpha)}{R(\alpha)+\alpha^{-1}-x}
\mu_D(dx)\\
&=R(\alpha)+\alpha R'(\alpha)-\alpha^{-1}+\alpha^{-1}
\left(\alpha^{-1}-\alpha R'(\alpha)\right) \cdot
G\left(G^{-1}(\alpha)\right)=R(\alpha).
\end{align*}
Hence $\inf_{\gamma>d_+}
\cH(\gamma,\alpha)=\cH(G^{-1}(\alpha),\alpha)=\int_0^\alpha R(z)dz$.

For part (b), applying the orthogonal transformations
\[\begin{pmatrix} \gamma & \nu \\ \nu & \rho \end{pmatrix}
\mapsto Q^\top \begin{pmatrix} \gamma & \nu \\ \nu & \rho \end{pmatrix} Q,
\qquad A \mapsto Q^\top A Q\]
for any orthogonal matrix $Q \in \O(2)$ preserves both the value of the
objective and the optimization domain $\cD_+$. Thus we may assume without loss
of generality that $A=\diag(\alpha_1,\alpha_2)$ is diagonal. In this case, the
function to be minimized is
\[\gamma \alpha_1-(1+\log \alpha_1)+\rho \alpha_2-(1+\log \alpha_2)
-\int \log \det \begin{pmatrix} \gamma-x & \nu \\ \nu & \rho-x \end{pmatrix}
\mu_D(dx).\]
This is strictly convex over $(\gamma,\nu,\rho) \in \cD_+$, and
its gradient is 0 at
$(\gamma,\nu,\rho)=(G^{-1}(\alpha_1),0,G^{-1}(\alpha_2))$ by
(\ref{eq:detvanishes}) and part (a). Thus the minimizer is
\[\begin{pmatrix} \gamma & \nu \\ \nu & \rho \end{pmatrix}
=G^{-1}(A),\]
and the value is $\int_0^{\alpha_1} R(z)dz+\int_0^{\alpha_2} R(z)dz=\Tr f(A)$
also by part (a).
\end{proof}

\section{Auxiliary results}
\label{app:aux}

\begin{proposition}\label{prop:continuousextension}
Let $S,T$ be two fixed metric spaces. For each $n \geq 1$, let
$K_n$ be a compact metric space, $f_n:K_n \to S$ a continuous map, and
$v_n:K_n \to T$ a map that is both continuous and relatively
open.\footnote{That is, $v_n(U_n)$ is open in $v_n(K_n)$ for any open subset
$U_n \subset K_n$.} For each $n \geq 1$, let $U_n$ be a dense subset of $K_n$
such that
\begin{itemize}
\item For some fixed subset $V \subset T$, we have $v_n(U_n)=V$ for every
$n$, and
\item There exists a function $f:V \to S$ such that $f_n(x)-f(v_n(x)) \to 0$
as $n \to \infty$, uniformly over $x \in U_n$.
\end{itemize}
Then $v_n(K_n)=\bar{V}$ (the closure of $V$ in $T$) for every $n$, this
function $f$ is continuous on $V$ and extends continuously to $\bar{V}$,
and $f_n(x)-f(v_n(x)) \to 0$ uniformly also over $x \in K_n$.
\end{proposition}
\begin{proof}
Since $K_n$ is compact and $v_n$ is continuous, $v_n(K_n)$ is also compact, so
$v_n(K_n) \supseteq \bar{V}$. The reverse inclusion
$v_n(K_n) \subseteq \bar{V}$ is immediate by continuity, so $v_n(K_n)=\bar{V}$.

For $x \in K_n$ and $v \in T$, denote $B_\eta(x)=\{x' \in K_n:\|x-x'\|<\eta\}$
and $B_\delta(v)=\{v' \in T:\|v-v'\|<\delta\}$.
To check that $f$ is continuous on $V$ and extends continuously to
$\bar{V}$, it suffices to show that for any $\eps>0$ and any $v \in \bar{V}$,
there exists $\delta>0$ for which
\begin{equation}\label{eq:fcompareV}
\|f(v')-f(v'')\|<\eps \text{ for all } v',v'' \in B_\delta(v) \cap V.
\end{equation}
Fix any such $\eps,v$, and let $n=n(\eps)$ be large enough so
that $\|f_n(x)-f(v_n(x))\|<\eps/3$ for all $x \in U_n$. For this $n$,
let $x_n \in K_n$
be a point where $v_n(x_n)=v$. By continuity of $f_n$, there exists
$\eta=\eta(n)>0$ sufficiently small such that $\|f_n(x')-f_n(x'')\|<\eps/3$
for all $x',x'' \in B_\eta(x_n)$. Then
\begin{equation}\label{eq:fcompareU}
\|f(v_n(x'))-f(v_n(x''))\|<\eps \text{ for all } x',x'' \in B_\eta(x_n)
\cap U_n.
\end{equation}
Since $v_n(x_n)=v$ and $v_n$ is relatively open,
for some $\delta=\delta(n)>0$, the image
$v_n(B_\eta(x_n))$ must contain $B_\delta(v) \cap \bar{V}$.
Then $v_n(B_\eta(x_n) \cap U_n) \supseteq B_\delta(v) \cap V$,
so (\ref{eq:fcompareU}) implies (\ref{eq:fcompareV}) as desired.

Finally, since $U_n$ is dense in $K_n$ and $f_n$, $f$, and $v_n$ are
continuous,
\[\sup_{x \in K_n} |f_n(x)-f(v_n(x))|
=\sup_{x \in K_n} \left(\mathop{\lim_{x' \to x}}_{x' \in U_n}
|f_n(x')-f(v_n(x'))|\right)
\leq \sup_{x' \in U_n} |f_n(x')-f(v_n(x'))|,\]
so the uniform convergence $|f_n(x)-f(v_n(x))| \to 0$
over $x \in K_n$ follows from that over $x \in U_n$.
\end{proof}

\begin{proposition}\label{prop:approxmin}	
	Let 
$D\subset\reals^d$ be a convex set and $f:D\to \reals$ be convex and twice differentiable.
Given $x_*\in D$ such that $B(x_*,\delta)=\{x:\|x-x_*\|<\delta\}\subset D$, suppose $\|\nabla f(x_*)\| \leq \eps$ and $\nabla^2 f(x) \succeq c I$ for all $x\in B(x_*,\delta)$, where $c \delta > 4 \eps$.
Then 
\[
\inf_{x \in D} f(x) \geq f(x_*) - \frac{4\eps^2}{c}.
\]
\end{proposition}
\begin{proof}
	For each $x \in B(x_*,\delta)$, we have $f(x)\geq f(x_*) + (\nabla
f(x_*))^\top(x-x_*) + \frac{c}{2}\|x-x_*\|^2$. So $f(x) > f(x_*)$ for all $\|x-x_*\|\geq 4\eps/c$. Therefore, the local minimum $\min_{\|x-x_*\|\leq 4\eps/c} f(x)$ is achieved at some $\tilde x$ such that $\|\tilde x-x_*\|<4\eps/c$ and hence $\nabla f(\tilde x)=0$. By convexity of $f$, 
$\tilde x$ is also the global minimizer	so $\inf_{x\in D} f(x)=f(\tilde x)$. 
Finally, $f(\tilde x) \geq f(x_*) + (\nabla f(x_*))^\top(x-x_*) \geq f(x_*) -4\eps^2/c$.	
\end{proof}

\section{Spherical model}
\label{app:sphere}
Consider the spherical counterpart of the Ising model \prettyref{eq:intromodel}, with partition function
\[
Z_{\mathrm{sphere}} \triangleq \int_{S^{n-1}(\sqrt{n})} \pi(d\sigma) \exp\pth{\frac{\beta}{2}\sigma^\top J\sigma + h^\top \sigma},
\]
where $J=O^\top DO$ and $\pi$ is the uniform distribution on
$S^{n-1}(\sqrt{n})$, the $n$-sphere of radius $\sqrt{n}$.
The replica-symmetric prediction of the limit free energy is
\begin{equation}\label{eq:RSsphere}
\Psi_{\RS,\mathrm{sphere}}=
\frac{1}{2}\inf_{\gamma > \bar d_+} \left\{\gamma + \Expect[\H^2] \cdot \bar G(\gamma) - \int \log (\gamma -x) \mu_{\bar D}(x) - 1\right\},
\end{equation}
where the rescaled notations $\bar d_+, \bar G, \mu_{\bar D}$ were defined in \prettyref{eq:scaling}. The following theorem justifies this formula. 
\begin{theorem}
\label{thm:sphere}
	Under Assumption \ref{assump:main}, for any fixed $\beta \in (0,G(d_+))$, almost surely
\[\lim_{n \to \infty} \frac{1}{n}\log Z_{\mathrm{sphere}}=\Psi_{\RS,\mathrm{sphere}}.\]
\end{theorem}

A derivation of this result in the special case of $h=0$ is given in
\cite[Section 2.1]{maillard2019high}.\footnote{\cite[Eq.\
(14)]{maillard2019high} studies the unnormalized surface area measure on the
sphere, and hence has an extra additive term of $\frac{1}{2}\log(2\pi e)$.}
We prove \prettyref{thm:sphere} using \prettyref{prop:Ointegralrank1},
which we have stated under the assumption $\beta<G(d_+)$. Dropping this
assumption requires removing the upper-bound condition on $\|a\|$ in
\prettyref{prop:Ointegralrank1}; such an extension
was obtained in \cite[Theorem 6]{guionnet2005fourier} for $b=0$.

\begin{proof}[Proof of \prettyref{thm:sphere}]
We express the uniform distribution of $\sigma \in S^{n-1}(\sqrt{n})$ as
$\sigma=Qa$, where $a \in S^{n-1}(\sqrt{n})$ is any fixed vector on the sphere, and $Q \sim \Haar(\O(n))$ is independent of $J$. By the given
condition $\beta<G(d_+)$, we have $\|a\|^2/n=1<\bar{G}(\bar d_+)=G(d_+)/\beta$.
Thus there exists $\eps>0$ for which $\|a\|^2/n=1<\bar{G}(\bar d_++\eps)-\eps$.
Setting $b=h$ and applying
\prettyref{prop:Ointegralrank1} to evaluate the expectation over $Q$
(conditional on $J$), we obtain
\[\lim_{n \to \infty} \left|\frac{1}{n}\log Z-f(\bar{J})\right|=0,
\qquad f(\bar{J}) \triangleq \frac{1}{2}\inf_{\gamma\geq \bar d_+ +\eps} f(\bar{J},\gamma)\]
where
\begin{align}
f(\bar{J},\gamma) &\triangleq  \gamma +\frac{h^\top (\gamma I-\bar J)^{-1}h}{n}
-\frac{1}{n}\log \det (\gamma I-\bar J) - 1\nonumber\\
&=\gamma +\frac{(Oh)^\top (\gamma I-\bar D)^{-1}(Oh)}{n}
-\frac{1}{n}\log \det (\gamma I-\bar D) - 1.\label{eq:fO}
\end{align}
For any $\gamma \geq \bar{d}+\eps$ and all large $n$,
note that $f(\bar{J},\gamma) \geq \gamma-\frac{1}{n}\log \det (\gamma I-\bar D) - 1$, where the right side diverges as $\gamma \to \infty$. 
Thus there exists some constant $\Gamma>0$ independent of $\bar{J}$ and $n$ such that
\begin{equation}
f(\bar{J}) = \frac{1}{2}\inf_{\gamma \in [\bar d_++\eps,\Gamma]} f(\bar{J},\gamma).
\label{eq:fO-restricted}
\end{equation}

Writing
$\Psi_{\RS,\mathrm{sphere}}=\frac{1}{2}\inf_{\gamma>\bar{d}_+} \Psi(\gamma)$
where $\Psi(\gamma)$ is the function in (\ref{eq:RSsphere}),
by the same reasoning, this infimum may be
restricted to $\gamma \leq \Gamma$. For $\gamma \in (\bar{d}_+,\bar{d}_++\eps)$,
we have $\Psi'(\gamma)=1+\E[\H^2] \cdot \bar{G}'(\gamma)-\bar{G}(\gamma)
\leq 1-\bar{G}(\gamma)<0$, and hence the infimum may also be restricted to
$\gamma \geq \bar{d}_++\eps$. So
\begin{equation}
\Psi_{\RS,\mathrm{sphere}}=\frac{1}{2}\inf_{\gamma \in [\bar d_++\eps,\Gamma]}
\Psi(\gamma).
\label{eq:EfO-restricted}
\end{equation}

Finally, we check the convergence of
$f(\bar{J},\gamma)$ to $\Psi(\gamma)$. Note that
$\Expect[(Oh)^\top (\gamma I-\bar D)^{-1}(Oh)] = \sum_{i=1}^n
\Expect[(Oh)_i^2]/(\gamma-\bar d_i) $, where 
$\Expect[(Oh)_i^2] = \|h\|^2 \Expect[O_{1i}^2] = \frac{\|h\|^2}{n}$ by symmetry.
Thus, applying Assumption \ref{assump:main}(b) and (c),
\begin{equation}
\Expect[f(\bar{J},\gamma)]  = \gamma +  \frac{\|h\|^2}{n^2}\sum_{i=1}^n \frac{1}{\gamma-\bar d_i} - \frac{1}{n}\log \det (\gamma I-\bar D) - 1
\xrightarrow{n\to\infty} \Psi(\gamma).
\label{eq:EfO-limit}
\end{equation}
Next we argue that $f(\bar{J},\gamma)$ concentrates, similar to the proof
of \prettyref{thm:replicasymmetric}. Viewing $f(\bar{J},\gamma)$ as a function
of $O$ via (\ref{eq:fO}), we may compute its derivative
\[\partial_O f(\bar{J},\gamma) = \frac{2}{n} hh^\top O^\top (\gamma I - \bar
D)^{-1}.\] Thus for large enough $n$ and any $\gamma \geq d_++\eps$,
we have $\Fnorm{\partial_O f(\bar{J},\gamma)} \leq \frac{4\|h\|^2}{n \eps} \Opnorm{O}$. By \prettyref{assump:main}(c), for all sufficiently large $n$, $\frac{1}{n} \|h\|^2 \leq 2 \Expect[\H^2]$ and hence
$O\mapsto f(\bar{J},\gamma)$ is $L$-Lipschitz on $\O(n)$ with
$L=\frac{8\Expect[\H^2]}{\eps}$.
Then by the same argument that leads to \prettyref{eq:concentrationprelim} and
\prettyref{eq:concentration}, we have for each $\gamma \geq \bar d_++\eps$,
\begin{equation}
\prob{|f(\bar{J},\gamma) - \Expect[f(\bar{J},\gamma)]| \geq \delta} \leq 2 \exp\sth{-\frac{(\frac{n}{4}-\frac{1}{2}) \delta^2}{2 L^2}}.
\label{eq:fO-concentrate}
\end{equation}
Furthermore, $|\partial_\gamma f(\bar{J},\gamma)| \leq 1 +
\frac{\|h\|^2}{n\eps} +\frac{1}{\eps}$. Thus for all $O\in \O(n)$ and all
sufficiently large $n$, $\gamma \mapsto f(\bar{J},\gamma)$ is $L'$-Lipschitz
with $L'=1 + (2 \Expect[\H^2]+1)/\eps$ on $[\bar d_++\eps,\Gamma]$. The same
Lipschitz continuity holds for $\Psi(\gamma)$. Combining
(\ref{eq:EfO-limit}) and (\ref{eq:fO-concentrate}), and applying Borel-Cantelli
and a union bound
over a sufficiently fine grid of values $\gamma \in [\bar d_++\eps,\Gamma]$, we
obtain the almost-sure convergence
$f(\bar{J},\gamma) \to \Psi(\gamma)$ uniformly over $\gamma \in [\bar d_++\eps,
\Gamma]$. Then by (\ref{eq:fO-restricted}) and (\ref{eq:EfO-restricted}),
also $f(\bar{J}) \to \Psi_{\RS,\mathrm{sphere}}$, completing the proof.
\end{proof}

\section{Cavity-method derivation of the TAP equations}\label{appendix:cavity}

We provide a brief review of the heuristic approach
in \cite{opper2001adaptive} for deriving the TAP equations \eqref{eq:TAP}. Let
\begin{equation}\label{eq:mchidef}
m_i=\langle \sigma_i \rangle, \qquad
\chi_{ij}=\langle \sigma_i \sigma_j \rangle-\langle \sigma_i \rangle
\langle \sigma_j \rangle
\end{equation}
where $\langle \cdot \rangle$ is the expectation
under the law $P(\sigma)$ in our model of interest (\ref{eq:gibbsmeasure}).
Define the cavity field $\theta_i=\sum_{j \neq i} \bar J_{ij}\sigma_j$. Then
the single-spin marginals of $P(\sigma)$ are
\[P(\sigma_i)=\frac{1}{Z_i}\langle e^{\sigma_i(h_i+\theta_i)} \rangle_{\setminus
i}, \qquad Z_i=\sum_{\sigma_i \in \{\pm 1\}} \langle e^{\sigma_i(h_i+\theta_i)}
\rangle_{\setminus i}\]
where $\langle \cdot \rangle_{\setminus i}$ denotes the expectation over
$\{\sigma_j\}_{j \neq i}$ (defining $\theta_i$)
in the cavity system with the spin $\sigma_i$ removed.
We have the exact identities
\begin{equation}\label{eq:michiij}
m_i=\partial_{h_i} \log Z_i, \qquad \chi_{ij}=\partial_{h_j} m_i.
\end{equation}

Approximating the law of the cavity field $\theta_i$ under $\langle \cdot
\rangle_{\setminus i}$ by a Gaussian law $\N(\mu_i,v_i)$, one obtains
from the Gaussian moment-generating-function
\begin{equation}\label{eq:cavityapprox}
P(\sigma_i) \approx \frac{1}{Z_i}
e^{\sigma_i(h_i+\mu_i)+v_i/2}, \qquad
m_i \approx \tanh(h_i+\mu_i), \qquad
Z_i \approx \sum_{\sigma_i \in \{\pm 1\}}
e^{\sigma_i(h_i+\mu_i)+v_i/2}.
\end{equation}
Furthermore, the law of $\theta_i$ under $\langle \cdot \rangle$ in the original
model is then approximately a two-component Gaussian mixture
$P(\theta_i) \propto
\sum_{\sigma_i \in \{\pm 1\}} e^{\sigma_i(h_i+\theta_i)-(\theta_i-\mu_i)^2/2v_i}
\propto \sum_{\sigma_i \in \{\pm 1\}}
e^{\sigma_i(h_i+\mu_i)} e^{-(\theta_i-\mu_i-\sigma_i v_i)^2/2v_i}$, from which
one obtains
$\langle \theta_i \rangle \approx \mu_i+v_i \tanh(h_i+\mu_i) \approx 
\mu_i+v_i m_i$.
Equating this with $\langle \theta_i \rangle=\sum_{k \neq i} \bar
J_{ik}m_k$ from the definition of $\theta_i$ gives
\begin{equation}\label{eq:muapprox}
\mu_i \approx \sum_k \bar J_{ik} m_k-v_im_i.
\end{equation}

Finally, the approach of \cite{opper2001adaptive} is to derive
an equation for the
cavity field variances $\{v_i\}$ by implicit differentiation in $h_j$,
assuming $\partial_{h_j} v_i \approx 0$. Then, differentiating
\eqref{eq:muapprox} and applying \eqref{eq:michiij},
\begin{equation}\label{eq:cavity1}
\partial_{h_j} \mu_i \approx \sum_k \bar J_{ik} \chi_{kj}-v_i\chi_{ij}.
\end{equation}
Differentiating the first equality of \eqref{eq:michiij} using the
approximation for $Z_i$ in \eqref{eq:cavityapprox},
\begin{equation}\label{eq:cavity2}
\chi_{ij}=\partial_{h_j} m_i
\approx \partial_{h_j}
\frac{\sum_{\sigma_i} \sigma_ie^{\sigma_i(h_i+\mu_i)+v_i/2}}{\sum_{\sigma_i}
e^{\sigma_i(h_i+\mu_i)+v_i/2}}
\approx \chi_{ii}\Big(\mathbf{1}\{i=j\}+\partial_{h_j} \mu_i\Big).
\end{equation}
Combining \eqref{eq:cavity1} and \eqref{eq:cavity2}, and denoting
$\chi=(\chi_{ij})_{i,j=1}^N$,
$X=\diag(\chi_{ii})_{i=1}^N$, and $V=\diag(v_i)_{i=1}^N$, one obtains
$\chi \approx X(I+\bar J \chi-V\chi)$, hence
$\chi \approx (V+X^{-1}-\bar J)^{-1}$. Taking the trace and assuming
further by the symmetries of the model that $\chi_{ii} \approx 1-q_*$ and
$v_i \approx v$ for some $v>0$ and all $i=1,\ldots,N$,
this gives $1-q_* \approx \bar G(v+\frac{1}{1-q_*})$, i.e.\
$v \approx \bar R(1-q_*)$, where
$\frac{1}{N}\Tr(z I - \bar J) \to \bar G(z)$ as $N\to\infty$
and $\bar G,\bar R$ are the Cauchy and R transforms of the limiting law $\bar \mu_{\bar D}$
defined in \prettyref{eq:GR}.
Substituting this into
(\ref{eq:cavityapprox}--\ref{eq:muapprox}) yields
$m \approx \tanh(h+\mu) \approx \tanh(h+\bar J m-\bar R(1-q_*)m)$
which are the TAP equations \eqref{eq:TAP}.


\bibliographystyle{plain}
\bibliography{main}

\end{document}